\documentclass[a4paper,11pt]{article}
\title{The class of $(2P_3,C_4,C_6)$-free graphs, part I: $(2P_3,C_4,C_6)$-free graphs that contain an induced $C_7$ or an induced $T_0$} 
\author{Irena Penev\thanks{Computer Science Institute (I\'UUK, MFF), Charles University, Prague, Czech Republic. Supported by GA\v{C}R grant 25-17377S. Email: {ipenev@iuuk.mff.cuni.cz}.}}

\usepackage{fullpage}

\usepackage{amsmath}
\usepackage{amsthm}
\usepackage{amssymb}
\usepackage{amsbsy}
\usepackage{mathrsfs}
\usepackage{amstext}  
\usepackage{enumerate} 

\usepackage{graphics}
\usepackage{graphicx}

\usepackage{chngcntr}
\numberwithin{figure}{section}

\usepackage{changepage}

\newtheorem{theorem}{Theorem}[section]  
\newtheorem{proposition}[theorem]{Proposition} 
\newtheorem{lemma}[theorem]{Lemma} 
\newtheorem{corollary}[theorem]{Corollary} 
 
\newtheorem{claim}{Claim}[theorem]

\begin{document} 
\maketitle 
\noindent 

\begin{abstract} 
\noindent 
This is the first in a series of two papers dealing with $(2P_3,C_4,C_6)$-free graphs, or equivalently, $(2P_3,\text{even hole})$-free graphs. In this two-paper series, we give a full structural description of $(2P_3,C_4,C_6)$-free graphs that contain no simplicial vertices, and we show that such graphs have bounded clique-width. This implies that \textsc{Graph Coloring} can be solved in polynomial time for $(2P_3,C_4,C_6)$-free graphs. In this paper, we describe the structure of $(2P_3,C_4,C_6)$-free graphs that contain an induced $C_7$ or an induced $T_0$ (where $T_0$ is a certain 2-connected graph on nine vertices in which all holes are of length five), and we show that such graphs either contain a simplicial vertex or have bounded clique-width. In the second part of this series, we describe the structure of all $(2P_3,C_4,C_6,C_7,T_0)$-free graphs that contain no simplicial vertices, and we show that such graphs have bounded clique-width. The full statement of the theorem describing the structure of $(2P_3,C_4,C_6)$-free graphs that contain no simplicial vertices is given in the second paper of this series. 
\end{abstract}

\section{Introduction} 

All graphs in this paper are finite, simple, and nonnull. We mostly use standard terminology and notation, formally introduced in section~\ref{sec:prelim}. For now, let us define a few basic terms. As usual, for a positive integer $k$, $K_k$ is the complete graph on $k$ vertices, $P_k$ is the path on $k$ vertices and $k-1$ edges, and  $C_k$ ($k \geq 3$) is the cycle on $k$ vertices. For a positive integer $k$ and a graph $H$, the disjoint union of $k$ copies of the graph $H$ is denoted by $kH$. (See Figure~\ref{fig:ForbiddenIndSgs} for some graphs relevant to this paper.) For a graph $H$, we say that a graph $G$ is {\em $H$-free} if no induced subgraph of $G$ is isomorphic to $H$. For a family of graphs $\mathcal{H}$, we say that a graph $G$ is {\em $\mathcal{H}$-free} if $G$ is $H$-free for all $H \in \mathcal{H}$.  A {\em hole} in a graph $G$ is an induced cycle on at least four vertices, and the {\em length} of a hole is the number of vertices (equivalently: edges) that it contains. A hole is {\em even} (resp.\ {\em odd}) if its length is even (resp.\ odd). A {\em clique} (resp.\ {\em stable set}) is a (possibly empty) set of pairwise adjacent (resp.\ nonadjacent) vertices, and a {\em simplicial vertex} is a vertex whose neighbors form a (possibly empty) clique. A {\em proper coloring} of a graph $G$ is an assignment of colors to the vertices of $G$ in such a way that no two adjacent vertices receive the same color. For an integer $k$, a graph $G$ is said to be {\em $k$-colorable} if there exists a proper coloring of $G$ that uses at most $k$ colors. The {\em chromatic number} of $G$, denoted by $\chi(G)$, is the smallest integer $k$ such that $G$ is $k$-colorable. \textsc{Graph Coloring} is the following problem. 

\bigskip 

\begin{minipage}{\textwidth}  
\textsc{Graph Coloring} 

\textbf{Instance:} A graph $G$ and an integer $k$. 

\textbf{Question:} Is $G$ $k$-colorable? 
\end{minipage} 

\bigskip 

This paper is the first in a series of two papers that deal with $(2P_3,C_4,C_6)$-free graphs. In this series, we give a full structural description of $(2P_3,C_4,C_6)$-free graphs that contain no simplicial vertices, and we show that \textsc{Graph Coloring} can be solved in polynomial time for $(2P_3,C_4,C_6)$-free graphs. 

\begin{figure}
\begin{center}
\includegraphics[scale=0.5]{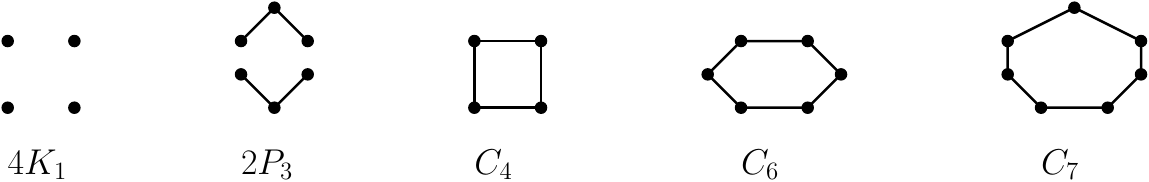}
\end{center} 
\caption{From left to right: graphs $4K_1$, $2P_3$, $C_4$, $C_6$, and $C_7$.} \label{fig:ForbiddenIndSgs} 
\end{figure} 

Before describing our results in more detail, let us make some general remarks and give the relevant context for our work. $(2P_3,C_4,C_6)$-free graphs are trivially recognizable in $O(n^6)$ time. Further, the \textsc{Maximum Clique} problem can be solved in polynomial time for $C_4$-free graphs~\cite{Alekseev-C4free, Farber, MakinoUno, TIAS};\footnote{Indeed, any $C_4$-free graph has only $O(n^2)$ maximal cliques~\cite{Alekseev-C4free, Farber}, and if a graph $G$ has $K$ maximal cliques, they can all be found in $O(Kn^3)$ time~\cite{MakinoUno, TIAS}. So, all maximal cliques of a $C_4$-free graph can be found in $O(n^5)$ time, and it immediately follows that the \textsc{Maximum Clique} problem can be solved in $O(n^5)$ time for $C_4$-free graphs.} the \textsc{Maximum Stable Set} problem can be solved in polynomial time for all $(\text{banner},P_8)$-free graphs~\cite{BannerP8FreeStable}, where the {\em banner} is the five-vertex graph consisting of a $C_4$ with a pendant edge;\footnote{Since the banner contains an induced $C_4$, and $P_8$ contains an induced $2P_3$, we see that all $(2P_3,C_4,C_6)$-free graphs are, in particular, $(\text{banner},P_8)$-free.} and the \textsc{3-Coloring} problem can be solved in polynomial time for $sP_3$-free graphs for any fixed positive integer $s$~\cite{sP3Update}. Therefore, the \textsc{Maximum Clique}, \textsc{Maximum Stable Set}, and \textsc{3-Coloring} problems are all solvable in polynomial time for $(2P_3,C_4,C_6)$-free graphs. As stated above, one of the main goals of our two-paper series is to show that the \textsc{Graph Coloring} problem can be solved in polynomial time for $(2P_3,C_4,C_6)$-free graphs. 

We note that $4K_1$ is an induced subgraph of $2P_3$, and so $(2P_3,C_4,C_6)$-free graphs form a proper superclass of the class of $(4K_1,C_4,C_6)$-free graphs. The structure of $(4K_1,C_4,C_6)$-free graphs is almost completely understood: a full structural description of $(4K_1,C_4,C_6)$-free graphs that contain an induced $C_7$ was given in~\cite{ChinhHoang}, and a full structural description of $(4K_1,C_4,C_6,C_7)$-free graphs that contain no simplicial vertices was given in~\cite{4K1C4C6C7Free}. Moreover, all the following problems can be solved in $O(n^3)$ time for $(4K_1,C_4,C_6)$-free graphs: recognition, \textsc{Maximum Clique}, and \textsc{Graph Coloring}~\cite{4K1C4C6Alg}.\footnote{Clearly, the stability number of any $4K_1$-free graph is at most three, and so the \textsc{Maximum Stable Set} problem can trivially be solved in $O(n^3)$ for $4K_1$-free graphs, and consequently, for $(4K_1,C_4,C_6)$-free graphs as well.} On the other hand, we note that every hole of length at least eight contains an induced $2P_3$, and consequently, all holes in a $(2P_3,C_4,C_6)$-free graph are of length five or seven; thus, $(2P_3,C_4,C_6)$-free graphs are even-hole-free, and in fact, $(2P_3,C_4,C_6)$-free graphs are precisely the $(2P_3,\text{even hole})$-free graphs. The class of even-hole-free graphs has received a great deal of attention over the past couple of decades. In particular, decomposition theorems have been proven for even-hole-free graphs~\cite{EvenHoleFreeDecomp, EvenHoleFreeDecompSV}, such graphs can be recognized in polynomial time~\cite{EvenHoleFreeRecFast, EvenHoleFreeRec}, and they are $\chi$-bounded by a linear function~\cite{EvenHoleFreeBisimplicial}. As mentioned above, the \textsc{Maximum Clique} problem can be solved in polynomial time for $C_4$-free graphs, and consequently, for even-hole-free graphs as well. On the other hand, the time complexity of the \textsc{Maximum Stable Set} and \textsc{Graph Coloring} problems is unknown for even-hole-free graphs. Finally, let us mention that the the structure of ``$\ell$-holed graphs,'' i.e.\ graphs in which all holes are of length $\ell$, is well understood for each integer $\ell \geq 7$ (see~\cite{L-Holed}). As pointed out above, all holes in a $(2P_3,C_4,C_6)$-free graph are of length five or seven, and consequently, all $(2P_3,C_4,C_6,C_7)$-free graphs are 5-holed. So, while general 5-holed graphs are not well understood, our results yield a good description of one particular subclass of 5-holed graphs, namely, that of $(2P_3,C_4,C_6,C_7)$-free graphs. 

Let us now describe our results in a bit more detail. Our main goal in this two-paper series is to give a full structural description of all $(2P_3,C_4,C_6)$-free graphs that contain no simplicial vertices, and to prove that such graphs have bounded clique-width (for a formal definition of ``clique width,'' see section~\ref{sec:cwd}). Since simplicial vertices pose no obstacle to coloring in polynomial time,\footnote{Indeed, if $k$ is an integer, and $x$ is a simplicial vertex of a graph $G$ on at least two vertices, then $G$ is $k$-colorable if and only if $d_G(x) \leq k-1$ and $G \setminus x$ is $k$-colorable. Moreover, the simplicial vertices of a graph (if any) can easily be found by simply examining the neighborhood of each vertex and checking if that neighborhood is a clique.} and since \textsc{Graph Coloring} can be solved in polynomial time for graphs of bounded clique-width~\cite{Rao}, it follows that the \textsc{Graph Coloring} problem can be solved in polynomial time for $(2P_3,C_4,C_6)$-free graphs. We note, however, that $(2P_3,C_4,C_6)$-free graphs do not have bounded clique-width: only those $(2P_3,C_4,C_6)$-free graphs that contain no simplicial vertices do. Indeed, it was shown in~\cite{CW4Vertex} that $4K_1$-free chordal graphs have unbounded clique-width (a {\em chordal graph} is a graph that contains no holes, i.e.\ a graph in which all induced cycles are triangles), and clearly, $4K_1$-free chordal graphs form a proper subclass of $(2P_3,C_4,C_6)$-free graphs. However, it is well known that all chordal graphs have a simplicial vertex~\cite{D61}. 

$T_0$ is the graph on nine vertices represented in Figure~\ref{fig:3pentagonT0T1} (center); clearly, $T_0$ is 2-connected, and it is not hard to check that $T_0$ is $(2P_3,C_4,C_6,C_7)$-free (for a formal proof, see Proposition~\ref{prop-T0T1-2P3C4C6Free}), and in particular, all holes in $T_0$ are of length five. In this paper, we give a full structural description of $(2P_3,C_4,C_6)$-free graphs that contain an induced $C_7$ or an induced $T_0$ (see Theorem~\ref{thm-main-withC7T0-full}); we note that this structural description yields an $O(n^3)$ time recognition algorithm for $(2P_3,C_4,C_6)$-free graphs that contain an induced $C_7$ or an induced $T_0$ (see section~\ref{sec:Alg}). We also give a full structural description of all $(2P_3,C_4,C_6)$-free graphs that contain an induced $C_7$ or an induced $T_0$, and contain no simplicial vertices (see Theorem~\ref{thm-main-withC7T0}), and we show that all such graphs have bounded clique-width (see Theorem~\ref{thm-cwd-in-class-with-C7-T0}). In the second part of this series~\cite{2P3C4C6FreePart2}, we give a full structural description of $(2P_3,C_4,C_6,C_7,T_0)$-free graphs that contain no simplicial vertices, and we show that such graphs have bounded clique-width. Together, these two papers yield a full structural description of $(2P_3,C_4,C_6)$-free graphs that contain no simplicial vertices, and the fact that all such graphs have bounded clique-width; as explained above, this implies that \textsc{Graph Coloring} can be solved in polynomial time for $(2P_3,C_4,C_6)$-free graphs. The full formal statement of our structure theorem for $(2P_3,C_4,C_6)$-free graphs that contain no simplicial vertices is given in the second paper of our series.\footnote{We omit the statement of this theorem in the present paper because, to make sense of the theorem, we would need some rather lengthy definitions that are relevant for $(2P_3,C_4,C_6,C_7,T_0)$-free graphs, and those definitions are only given in the second paper.} 

\begin{figure}
\begin{center}
\includegraphics[scale=0.5]{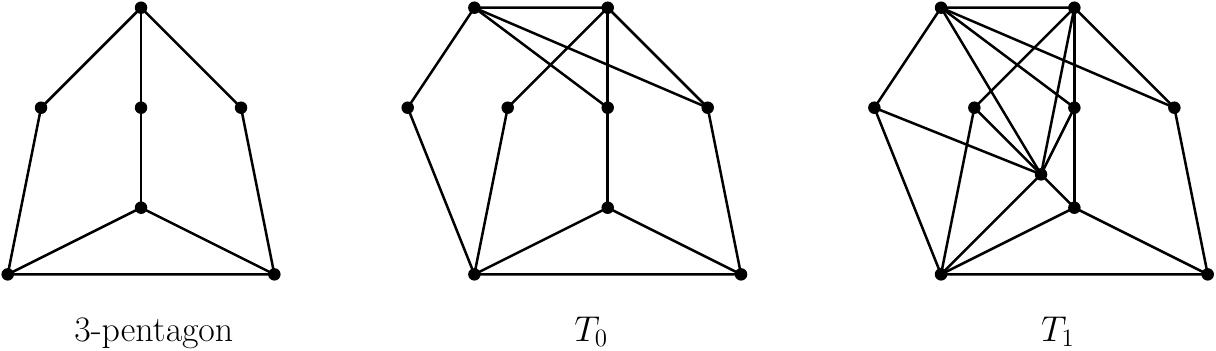}
\end{center} 
\caption{The 3-pentagon (left), $T_0$ (center), and $T_1$ (right).} \label{fig:3pentagonT0T1} 
\end{figure}

It may be worth emphasizing that, while we were able to fully describe the structure of $(2P_3,C_4,C_6)$-free graphs that contain an induced $C_7$ or an induced $T_0$ (regardless of the presence or absence of simplicial vertices), in the $(2P_3,C_4,C_6,C_7,T_0)$-free case, simplicial vertices complicate matters more substantially, and for this reason, in the second paper of our series, we restrict our attention to those $(2P_3,C_4,C_6,C_7,T_0)$-free graphs that contain no simplicial vertices. Thus, even though we ultimately get a full structural description of $(2P_3,C_4,C_6)$-free graphs that contain no simplicial vertices, this is not quite a full structure theorem for all $(2P_3,C_4,C_6)$-free graphs. This is essentially because the class of $(2P_3,C_4,C_6)$-free graphs is not closed under the addition of simplicial vertices, that is, adding a simplicial vertex to a $(2P_3,C_4,C_6)$-free graph may possibly produce a graph that contains an induced $2P_3$,\footnote{This is because $2P_3$ itself contains a simplicial vertex. On the other hand, since holes do not contain simplicial vertices, it is not possible to generate a new hole by simply adding one or more simplicial vertices, and in particular, any $(C_4,C_6)$-free graph remains $(C_4,C_6)$-free if we add a simplicial vertex to it.} and is consequently not $(2P_3,C_4,C_6)$-free. Of course, for many algorithmic problems (including \textsc{Graph Coloring}, but by no means limited to it), simplicial vertices can be handled with ease.

\subsection{Further remarks on our results and some related open problems} 

Let us now briefly discuss some problems related to our work that still remain open. First, while \textsc{Graph Coloring} is indeed solvable in polynomial time for graphs of bounded clique-width, all known polynomial-time coloring algorithms for graphs of clique-width at most $k$ have running time $O(n^{f(k)})$ for some fast growing function $f$ (see~\cite{CWcol} for an overview). So, it might be interesting to try to develop a faster coloring algorithm for $(2P_3,C_4,C_6)$-free graphs, perhaps using the structural results described in our two-paper series. Despite substantial effort, we have thus far not been able to develop a coloring algorithm for this class that does not rely on clique-width. More precisely, using the results of~\cite{Koutecky, Lampis} and the structural results of the present paper, we can in fact color $(2P_3,C_4,C_6)$-free graphs that contain an induced $C_7$ or an induced $T_0$ in only $O(n^3)$ time (see section~\ref{sec:Alg}). (Here, we use the results of~\cite{Koutecky, Lampis} as a black box, but we do note that the algorithms in question rely on integer programming and are therefore not combinatorial.) However, in the second paper of our series, we are forced to rely on bounded clique-with to color $(2P_3,C_4,C_6,C_7,T_0)$-free graphs, which yields a slow running time. 

Second, note that $(2P_3,C_4,C_6)$-free graphs form a proper subclass of the class of $(P_7,C_4,C_6)$-free graphs. It would, of course, be interesting to describe the structure of $(P_7,C_4,C_6)$-free graphs and to determine the time complexity of \textsc{Graph Coloring} for this class. Despite quite some effort, we have not succeeded in doing this. Nevertheless, in section~\ref{sec:withC7}, we in fact describe the structure of $(P_7,C_4,C_6)$-free graphs that contain an induced $C_7$ and do not admit a clique-cutset. This reduces the problem of describing the structure of $(P_7,C_4,C_6)$-free graphs to that of describing the structure of $(P_7,C_4,C_6,C_7)$-free graphs, as long as we are willing to regard the clique-cutset as an acceptable decomposition in this context.\footnote{Note, however, that $P_7$ admits a clique-cutset, and consequently, the class of $(P_7,C_4,C_6)$-free graphs is not closed under the operation of gluing along a clique (the operation that ``reverses'' clique-cutsets). The clique-cutset decomposition is nevertheless reasonable to accept in this context, since there are powerful algorithmic tools for handling this decomposition (see~\cite{Tarjan}).} However, so far, we have not been able to make any meaningful progress on the structure of $(P_7,C_4,C_6,C_7)$-free graphs. 

\begin{figure} 
\begin{center}
\includegraphics[scale=0.5]{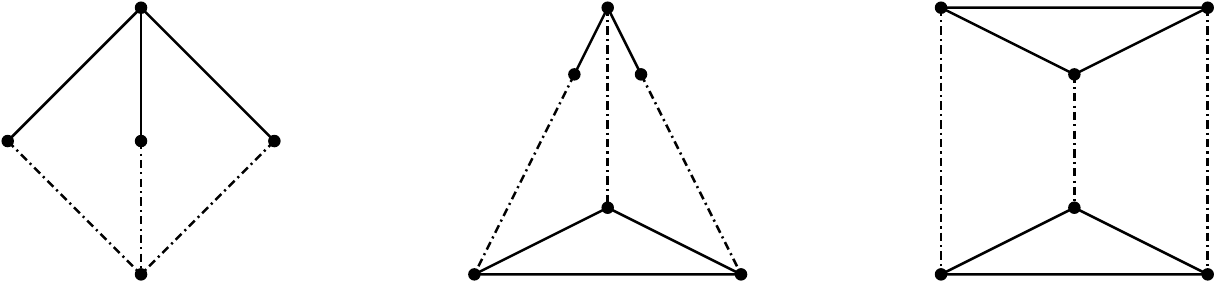}
\end{center} 
\caption{Three-path-configurations (3PCs): theta (left), pyramid (middle), and prism (right). A full line represents an edge, and a dashed line represents a path that has at least one edge.} \label{fig:Truemper} 
\end{figure}

\begin{figure}
\begin{center}
\includegraphics[scale=0.4]{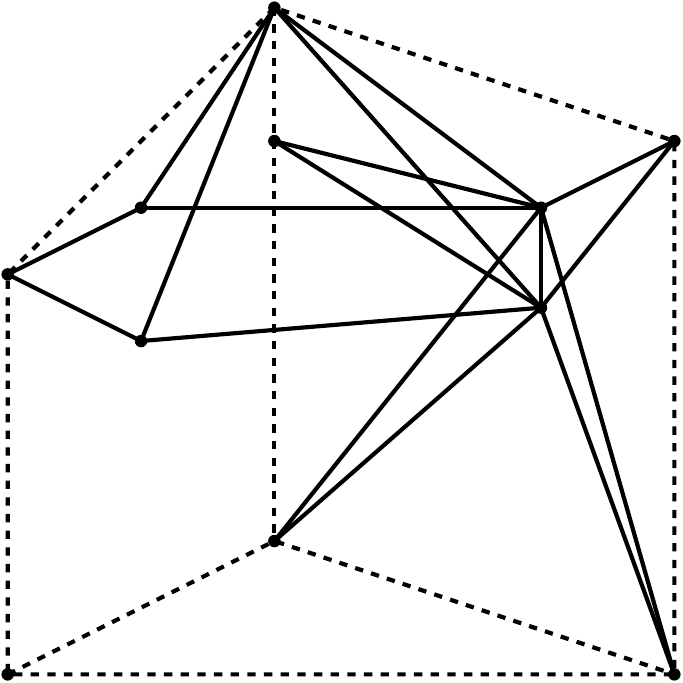} 
\hspace*{2cm}
\includegraphics[scale=0.4]{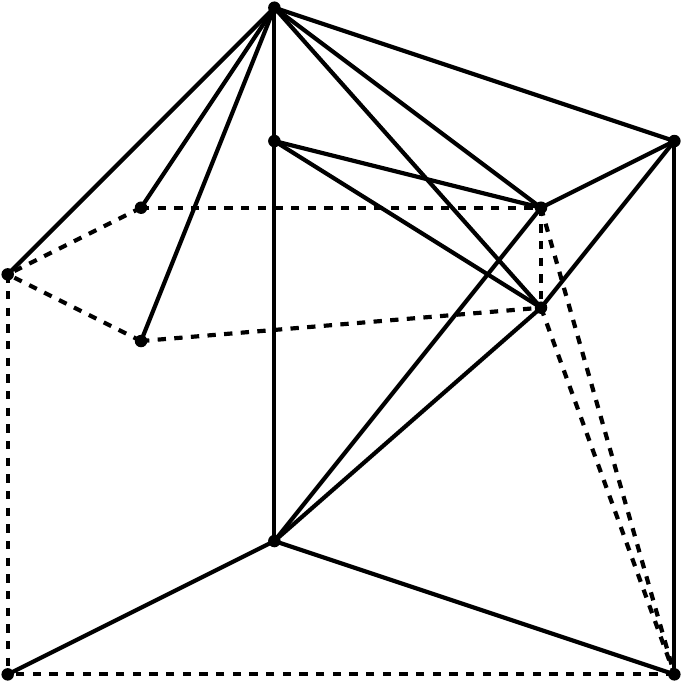}
\end{center} 
\caption{Two copies of the same $(P_7,C_4,C_6,C_7)$-free graph on eleven vertices, with induced 3-pentagons intersecting in an inconvenient way. Both full and dashed lines represent edges. The edges of two induced 3-pentagons are indicated by dashed lines.} \label{fig:BadAttach} 
\end{figure} 

\begin{figure}
\begin{center}
\includegraphics[scale=0.5]{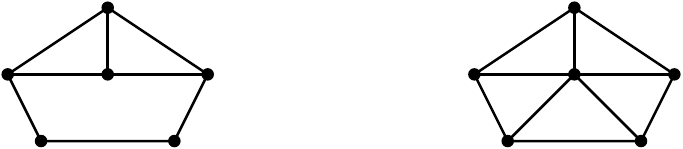}
\end{center} 
\caption{The only wheels in which all holes are of length five.} \label{fig:TwoWheels} 
\end{figure}

Here, it may be worth giving a bit more context. A {\em three-path-configuration} (or {\em 3PC} for short) is any theta, pyramid, or prism (see Figure~\ref{fig:Truemper}). A {\em wheel} is any graph that consists of a hole, plus an additional vertex that has at least three neighbors in the hole. A {\em Truemper configuration} is any 3PC or wheel. For a (slightly dated) survey on Truemper configurations, see~\cite{Truemper-survey}. Clearly, all holes in a $(P_7,C_4,C_6,C_7)$-free graph are of length five. Meanwhile, the only 3PC in which all holes are of length five is the 3-pentagon (see Figure~\ref{fig:3pentagonT0T1}, and note that the 3-pentagon is a type of pyramid). Unfortunately, 3-pentagons may intersect in rather inconvenient ways in $(P_7,C_4,C_6,C_7)$-free graphs (see Figure~\ref{fig:BadAttach} for an example), and this is the main reason why we have not been able to make meaningful progress on $(P_7,C_4,C_6,C_7)$-free graphs. Meanwhile, induced 3-pentagons in $(2P_3,C_4,C_6,C_7)$-free graphs behave in a more manageable fashion (this is studied in the second paper of our series), which is why we were able to get a good structural description of such graphs, as long as they do not contain simplicial vertices. Finally, let us mention in passing that the two wheels represented in Figure~\ref{fig:TwoWheels} are the only wheels in which all holes are of length five; so, those two wheels are the only wheels that a $(P_7,C_4,C_6,C_7)$-free graph may possibly contain as induced subgraphs.

\subsection{Paper outline} 

We complete the introduction with an outline of the paper. In section~\ref{sec:prelim}, we introduce the terminology and notation that we will use throughout the paper (see subsection~\ref{subsec:terminology}), and we also prove a few simple results that will be of use to us as we proceed (see subsection~\ref{subsec:SimpleResults}). 

In section~\ref{sec:withC7}, we prove a structure theorem for $(2P_3,C_4,C_6)$-free graphs that contain an induced $C_7$ (see Theorem~\ref{thm-7-saucer}). We further give a full structural description of $(P_7,C_4,C_6)$-free graphs that contain an induced $C_7$ and do not admit a clique-cutset; as we shall see, these graphs are precisely the same as $(2P_3,C_4,C_6)$-free graphs that contain an induced $C_7$ and contain no simplicial vertices, which are, in turn, precisely the $(4K_1,C_4,C_6)$-free graphs that contain an induced $C_7$ (see Theorem~\ref{thm-main-withC7}). The structure of $(4K_1,C_4,C_6)$-free graphs that contain an induced $C_7$ was originally described in~\cite{ChinhHoang}. 

In section~\ref{sec:withT0}, we give a full structural description of $(2P_3,C_4,C_6,C_7)$-free graphs that contain an induced $T_0$ (see Theorem~\ref{thm-T0-tent}). As a corollary, we obtain a full structural description of $(2P_3,C_4,C_6)$-free graphs that contain an induced $T_0$ and do not contain a simplicial vertex (see Corollary~\ref{cor-T0-tent}). 

In section~\ref{sec:structure}, we combine the results of sections~\ref{sec:withC7} and~\ref{sec:withT0} to obtain the main structural results of this paper: a full structural description of $(2P_3,C_4,C_6)$-free graphs that contain an induced $C_7$ or an induced $T_0$ (see Theorem~\ref{thm-main-withC7T0-full}), and a full structural description of $(2P_3,C_4,C_6)$-free graphs that contain an induced $C_7$ or an induced $T_0$ and contain no simplicial vertices (see Theorem~\ref{thm-main-withC7T0}). 

In section~\ref{sec:cwd}, we prove that every $(2P_3,C_4,C_6)$-free graph that contains an induced $C_7$ or an induced $T_0$, either has a simplicial vertex or has clique-width at most 12 (see Theorem~\ref{thm-cwd-in-class-with-C7-T0}). 

Finally, in section~\ref{sec:Alg}, we show that $(2P_3,C_4,C_6)$-free graphs that contain an induced $C_7$ or an induced $T_0$ can be recognized and colored in $O(n^3)$ time.

\section{Preliminaries} \label{sec:prelim}

\subsection{Terminology and notation} \label{subsec:terminology} 

When we say that ``sets $X_1,\dots,X_{\ell}$ form a partition of the set $X$,'' or that ``$(X_1,\dots,X_{\ell})$ is a partition of the set $X$,'' we mean that sets $X_1,\dots,X_{\ell}$ are pairwise disjoint and that $X = X_1 \cup \dots \cup X_{\ell}$. However, slightly nonstandardly, we allow some (or all) of the $X_i$'s to be empty. 

The vertex and edge set of a graph $G$ are denoted by $V(G)$ and $E(G)$, respectively. For a vertex $x$ in a graph $G$, the {\em open neighborhood} (or simply {\em neighborhood}) of $x$ in $G$, denoted by $N_G(x)$, is the set of all neighbors of $x$ in $G$; the {\em closed neighborhood} of $x$ in $G$, denoted by $N_G[x]$, is defined as $N_G[x] := \{x\} \cup N_G(x)$; the {\em degree} of $x$ in $G$, denoted by $d_G(x)$, is the number of neighbors that $x$ has in $G$, i.e.\ $d_G(x) := |N_G(x)|$. 

A vertex $v$ of a graph $G$ is {\em universal} in $G$ if $N_G[v] = V(G)$, i.e.\ if $v$ is adjacent to all other vertices of the graph $G$. 

Distinct vertices $x,y$ of a graph $G$ are {\em twins} in $G$ if $N_G[x] = N_G[y]$. (Note that twins are, in particular, adjacent to each other.) 

For a graph $G$ and a set $X \subseteq V(G)$, the {\em open neighborhood} of $X$ in $G$, denoted by $N_G(X)$, is the set of all vertices in $V(G) \setminus X$ that have a neighbor in $X$, and the {\em closed neighborhood} of $X$ in $G$ is the set $N_G[X] := X \cup N_G(X)$. 

For a graph $G$ and a nonempty set $X \subseteq V(G)$, we denote by $G[X]$ the subgraph of $G$ induced by $X$; for vertices $x_1,\dots,x_t \in V(G)$, we sometimes write $G[x_1,\dots,x_t]$ instead of $G[\{x_1,\dots,x_t\}]$. For a set $X \subsetneqq V(G)$, we set $G \setminus X := G[V(G) \setminus X]$. If $G$ has at least two vertices and $x \in V(G)$, we sometimes write $G \setminus x$ instead of $G \setminus \{x\}$.\footnote{Since our graphs are nonnull, if $x$ is the only vertex of a graph $G$, then $G \setminus x$ is not defined.} 

Given a graph $G$, a vertex $x \in V(G)$, and a set $Y \subseteq V(G) \setminus \{x\}$, we say that $x$ is {\em complete} (resp.\ {\em anticomplete}) to $Y$ in $G$ provided that $x$ is adjacent (resp.\ nonadjacent) to all vertices of $Y$ in $G$, and we say that $x$ is {\em mixed} on $Y$ in $G$ if $x$ is neither complete nor anticomplete to $Y$ in $G$ (i.e.\ if $x$ has both a neighbor and a nonneighbor in $Y$). 

Given a graph $G$ and disjoint sets $X,Y \subseteq V(G)$, we say that $X$ is {\em complete} (resp.\ {\em anticomplete}) to $Y$ in $G$ if every vertex in $X$ is complete (resp.\ anticomplete) to $Y$ in $G$. 

When we say, for graphs $G$ and $Q$, that ``$G$ can be obtained from $Q$ by possibly adding universal vertices to it,'' we mean that $Q$ is an induced subgraph of $G$, and $V(G) \setminus V(Q)$ is a (possibly empty) clique of $G$, complete to $V(Q)$ in $G$. Note that in this case, every vertex in $V(G) \setminus V(Q)$ is a universal vertex of $G$. However, it is possible that $G = Q$, in which case $G$ contains no universal vertices unless $Q$ does. 

A {\em thickening} of a graph $G$ is a graph $G^*$ for which there exists a family $\{X_v\}_{v \in V(G)}$ of pairwise disjoint, nonempty sets such that all the following hold: 
\begin{itemize} 
\item $V(G^*) = \bigcup_{v \in V(G)} X_v$; 
\item for all $v \in V(G)$, $X_v$ is a clique of $G^*$; 
\item for all distinct $u,v \in V(G)$, the following hold: 
\begin{itemize} 
\item if $u,v$ are adjacent in $G$, then $X_u$ and $X_v$ are complete to each other in $G^*$; 
\item if $u,v$ are nonadjacent in $G$, then $X_u$ and $X_v$ are anticomplete to each other in $G^*$. 
\end{itemize} 
\end{itemize} 

When we write, for a graph $G$, that ``$x_0,\dots,x_t$ is an induced path in $G$'' ($t \geq 0$), we mean that $x_0,\dots,x_t$ are pairwise distinct vertices, and that the edges of $G[x_0,\dots,x_t]$ are precisely $x_0x_1,x_1x_2,\dots,x_{t-1}x_t$; the {\em endpoints} of this path are $x_0$ and $x_t$ (note that if $t = 0$, then our path has only only one endpoint), the {\em interior vertices} of the path are all its vertices other than the endpoints, and furthermore, we say that our path is {\em between} its endpoints $x_0$ and $x_t$. The {\em length} of a path is the number of edges that it contains; we may also refer to a path of length $k$ as a ``$k$-edge path'' or as a ``$(k+1)$-vertex path.'' A path is {\em trivial} if it is of length zero (i.e.\ if it has only one vertex and no edges), and it is {\em nontrivial} if it is of length at least one. 

For an integer $k \geq 4$, a {\em $k$-hole} in a graph $G$ is an induced $C_k$ in $G$. When we write that ``$x_0,x_1,\dots,x_{k-1},x_0$ is a $k$-hole in $G$'' ($k \geq 4$), we mean that $x_0,x_1,\dots,x_{k-1}$ are pairwise distinct vertices of $G$, and that the edges of $G[x_0,x_1,\dots,x_{k-1}]$ are precisely the following: $x_0x_1,x_1x_2,\dots,x_{k-2}x_{k-1},x_{k-1}x_0$. 

The {\em complement} of a graph $G$, denoted by $\overline{G}$, is the graph whose vertex set is $V(G)$ and in which two distinct vertices are adjacent if and only if they are nonadjacent in $G$. A graph is {\em anticonnected} if its complement is connected. Note that every anticonnected graph on at least two vertices contains a pair of nonadjacent vertices. 

As usual, a {\em connected component} (or simply {\em component}) of a graph $G$ is a maximal connected induced subgraph of $G$. An {\em anticomponent} of a graph $G$ is a maximal anticonnected induced subgraph of $G$; clearly, $Q$ is an anticomponent of $G$ if and only if $\overline{Q}$ is a component of $\overline{G}$. An anticomponent is {\em trivial} if it has only one vertex, and it is {\em nontrivial} if it has at least two vertices. Clearly, the vertex sets of the anticomponents of a graph $G$ are complete to each other in $G$. Moreover, note that the unique vertex of any trivial anticomponent of a graph is a universal vertex of that graph. 

A {\em cutset} of a graph $G$ is a (possibly empty) set $C \subsetneqq V(G)$ such that $G \setminus C$ is disconnected. A {\em clique-cutset} of a graph $G$ is a cutset of $G$ that is also a clique of $G$. (In particular, if $G$ is disconnected, then $\emptyset$ is a clique-cutset of $G$.) Note that if $v$ is a simplicial vertex of a graph $G$, then either $G$ is complete, or $N_G(v)$ is a clique-cutset of $G$.

\subsection{A few simple results} \label{subsec:SimpleResults}

\begin{proposition} \label{prop-non-trivial-anticomp-univ-vertices} Let $G$ and $Q$ be graphs. Assume that $Q$ is anticonnected and contains at least two vertices. Then the following are equivalent: 
\begin{itemize} 
\item $G$ contains exactly one nontrivial anticomponent, and that anticomponent is $Q$; 
\item $G$ can be obtained from $Q$ by possibly adding universal vertices to it. 
\end{itemize} 
\end{proposition} 
\begin{proof} 
This follows immediately from the appropriate definitions. 
\end{proof}

\begin{proposition} \label{prop-one-nontrivial-anticomp-simplicial} Assume that a graph $G$ contains exactly one nontrivial anticomponent, call it $Q$. Then $G$ contains a simplicial vertex if and only if $Q$ contains a simplicial vertex. 
\end{proposition} 
\begin{proof} 
Set $W := V(G) \setminus V(Q)$. Since $Q$ is the only nontrivial anticomponent of $G$, we see that $W$ is a (possibly empty) clique of $G$, complete to $V(Q)$ in $G$. 

Suppose first that $Q$ contains a simplicial vertex, call it $v$. Then $N_G(v) = N_Q(v) \cup W$. Since $v$ is simplicial, we know that $N_Q(v)$ is a clique of $Q$, and therefore of $G$ as well. Since $W$ is a clique, complete to $V(Q)$ in $G$, we deduce that $N_G(v)$ is a clique of $G$, and consequently, $v$ is simplicial in $G$. 

Suppose conversely that $G$ contains a simplicial vertex, call it $v$. Since every vertex of $W$ is complete to $V(Q)$, and since $Q$ contains a pair of distinct, nonadjacent vertices (because $Q$ is a nontrivial anticomponent of $G$), we see that no vertex of $W$ is simplicial in $G$. Therefore, $v \notin W$, and we deduce that $v \in V(Q)$. But clearly, $N_Q(v) = N_G(v) \setminus W$; since $N_G(v)$ is a clique of $G$, it is clear that $N_Q(v)$ is a clique of $Q$, and we conclude that $v$ is simplicial in $Q$. 
\end{proof} 

\begin{proposition} \label{prop-H-free-no-universal} Assume that a graph $G$ contains exactly one nontrivial anticomponent, call it $Q$. Let $H$ be a graph that contains no universal vertices. Then $G$ is $H$-free if and only if $Q$ is $H$-free. 
\end{proposition} 
\begin{proof} 
If $G$ is $H$-free, then it is clear that $Q$ is also $H$-free. Suppose now that $G$ is not $H$-free; we must show that $Q$ is not $H$-free either. Fix a nonempty set $X \subseteq V(G)$ such that $G[X]$ is isomorphic to $H$. If $X \subseteq V(Q)$, then $Q[X] = G[X]$, and we deduce that $Q$ is not $H$-free. We may now assume that $X \not\subseteq V(Q)$. Fix any $x \in X \setminus V(Q)$. Since $Q$ is the only nontrivial anticomponent of $G$, we know that $V(G) \setminus V(Q)$ is a clique, complete to $V(Q)$ in $G$, and we deduce that $x$ is a universal vertex of $G$. But then $x$ is also a universal vertex of $G[X]$, which is impossible since $G[X]$ is isomorphic to $H$, and $H$ contains no universal vertices. 
\end{proof} 

\begin{proposition} \label{prop-non-adj-comp-clique} Let $G$ be a $C_4$-free graph, let $S \subseteq V(G)$, and let $v_1,v_2 \in V(G) \setminus S$ be distinct, nonadjacent vertices, complete to $S$ in $G$. Then $S$ is a clique. 
\end{proposition} 
\begin{proof} 
Suppose otherwise, and fix distinct, nonadjacent vertices $s_1,s_2 \in S$. Then $v_1,s_1,v_2,s_2,v_1$ is a 4-hole in $G$, contrary to the fact that $G$ is $C_4$-free. 
\end{proof} 

\begin{proposition} \label{prop-P3-free} A graph is $P_3$-free if and only if all its components are complete graphs. 
\end{proposition} 
\begin{proof} 
Fix a graph $G$. If all components of $G$ are complete graphs, then it is clear that $G$ is $P_3$-free. Suppose, conversely, that $G$ is $P_3$-free, and let $C$ be any component of $G$; we must show that $C$ is a complete graph. Suppose otherwise, and fix distinct, nonadjacent vertices $c,c'$ of the component $C$. Since $C$ is connected, we see that it contains an induced path $P$ between $c$ and $c'$. Since $c,c'$ are nonadjacent, the path $P$ contains at least three vertices. But now any three consecutive vertices of $P$ induce a $P_3$ in $G$, a contradiction. 
\end{proof} 

\begin{proposition} \label{prop-C4Free-CoBip} Let $G$ be a $C_4$-free graph, and let $X$ and $Y$ be disjoint cliques in $G$. Then all the following hold: 
\begin{enumerate}[(a)] 
\item for all $x,x' \in X$, one of $N_G(x) \cap Y$ and $N_G(x') \cap Y$ is a subset of the other; 
\item $X$ can be ordered as $X = \{x_1,\dots,x_{|X|}\}$ so that $N_G(x_{|X|}) \cap Y \subseteq \dots \subseteq N_G(x_1) \cap Y$; 
\item for all $y,y' \in Y$, one of $N_G(y) \cap X$ and $N_G(y') \cap X$ is a subset of the other; 
\item $Y$ can be ordered as $Y = \{y_1,\dots,y_{|Y|}\}$ so that $N_G(y_{|Y|}) \cap X \subseteq \dots \subseteq N_G(y_1) \cap X$. 
\end{enumerate} 
\end{proposition} 
\begin{proof} 
By symmetry, it suffices to prove~(a) and~(b). Moreover, it is clear that~(a) immediately implies~(b), and so it is enough to prove~(a). Suppose that~(a) is false, and fix $x,x' \in X$ such that neither one of $N_G(x) \cap Y$ and $N_G(x') \cap Y$ is a subset of the other. Then there exist $y,y' \in Y$ such that $y$ is adjacent to $x$ and nonadjacent to $x'$, and such that $y'$ is adjacent to $x'$ and nonadjacent to $x$. But then $x,x',y',y,x$ is a 4-hole in $G$, contrary to the fact that $G$ is $C_4$-free. 
\end{proof}


\begin{proposition} \label{prop-2P3C4-free-clique-cut-simplicial} Let $G$ be a $(2P_3,C_4)$-free graph that admits a clique-cutset. Then $G$ contains a simplicial vertex. 
\end{proposition} 
\begin{proof} 
Fix a clique-cutset $C$ of $G$, and let $A$ and $B$ be the vertex sets of two distinct components of $G \setminus C$. Then at least one of $A,B$ is a clique, for otherwise, Proposition~\ref{prop-P3-free} guarantees that $G[A]$ and $G[B]$ both contain an induced $P_3$, and consequently (since $A$ and $B$ are anticomplete to each other), $G$ contains an induced $2P_3$, a contradiction. By symmetry, we may assume that $A$ is a clique. Set $t := |A|$; since $A \neq \emptyset$, we know that $t \geq 1$. We now apply Proposition~\ref{prop-C4Free-CoBip}(b) to the cliques $A$ and $C$, and we fix an ordering $A = \{a_1,\dots,a_t\}$ of $A$ such that $N_G(a_t) \cap C \subseteq \dots \subseteq N_G(a_1) \cap C$. Since $A$ is a clique, and since $N_G(A) \subseteq C$, we see that $N_G[a_t] = A \cup \big(N_G(a_t) \cap C\big)$. But now $N_G[a_t]$ is  a clique: indeed, $A$ and $C$ are both cliques, and our ordering of $A$ guarantees that $N_G(a_t) \cap C$ is complete to $A$. So, $a_t$ is simplicial in $G$. 
\end{proof}

\begin{proposition} \label{prop-G-x-disjoint-cliques-2P3-free} Let $G$ be a graph, and let $X \subsetneqq V(G)$. Assume that for all $x \in X$, the set $V(G) \setminus N_G[x]$ can be partitioned into cliques of $G$, pairwise anticomplete to each other.\footnote{It is possible that $V(G) \setminus N_G[x] = \emptyset$, i.e.\ that $N_G[x] = V(G)$.} Then $G$ is $2P_3$-free if and only if $G \setminus X$ is $2P_3$-free. 
\end{proposition} 
\begin{proof} 
If $G$ is $2P_3$-free, then so is its induced subgraph $G \setminus X$. We now assume that $G$ is not $2P_3$-free, and we show that $G \setminus X$ is not $2P_3$-free either. Fix pairwise distinct vertices $a_0,a_1,a_2,b_0,b_1,b_2 \in V(G)$ such that $a_0,a_1,a_2$ and $b_0,b_1,b_2$ are both induced paths in $G$, and such that $\{a_0,a_1,a_2\}$ and $\{b_0,b_1,b_2\}$ are anticomplete to each other in $G$. Clearly, it is enough to show that $X \cap \{a_0,a_1,a_2,b_0,b_1,b_2\} = \emptyset$. Suppose otherwise. By symmetry, we may assume that there exists an index $i \in \{0,1,2\}$ such that $a_i \in X$. So, by hypothesis and by Proposition~\ref{prop-P3-free}, we have that $G \setminus N_G[a_i]$ is $P_3$-free. But on the other hand, since $\{a_1,a_2,a_3\}$ and $\{b_1,b_2,b_3\}$ are anticomplete to each other, we see that $b_1,b_2,b_3 \in V(G) \setminus N_G[a_i]$, and consequently, $G \setminus N_G[a_i]$ contains an induced $P_3$, a contradiction. 
\end{proof} 

\begin{proposition} \label{prop-add-complete-components} Let $G$ be a graph, and let $A \subsetneqq V(G)$. Assume that $A$ is either empty or can be partitioned into nonempty cliques, pairwise anticomplete to each other. Assume furthermore that $A$ is anticomplete to $V(G) \setminus A$. Then the following hold: 
\begin{enumerate}[(a)] 
\item for any graph $H$ such that no component of $H$ is a complete graph, we have that $G$ is $H$-free if and only if $G \setminus A$ is $H$-free; 
\item $G$ is $2P_3$-free if and only if $G \setminus A$ is $2P_3$-free.  
\end{enumerate} 
\end{proposition} 
\begin{proof} 
Clearly, (a) implies (b). So, it is enough to prove (a). Fix a graph $H$ such that no component of $H$ is a complete graph. If $G$ is $H$-free, then $G \setminus A$ is obviously $H$-free as well. Now, suppose that $G$ is not $H$-free; we must show that $G \setminus A$ is not $H$-free either. Fix $X \subseteq V(G)$ such that $G[X]$ is isomorphic to $H$. It now suffices to show that $X \cap A = \emptyset$, for this will immediately imply that $G \setminus A$ is not $H$-free, which is what we need to show. Suppose otherwise, that is, suppose that $X \cap A \neq \emptyset$. This, in particular, implies that $A \neq \emptyset$, and so by hypothesis, $A$ can be partitioned into nonempty cliques $A_1,\dots,A_{\ell}$, pairwise anticomplete to each other, and anticomplete to $V(G) \setminus A$. By symmetry, we may assume that $X \cap A_1 \neq \emptyset$. Since $A_1$ is a clique, anticomplete to $V(G) \setminus A_1$, we deduce that $G[X \cap A_1]$ is a complete graph that is also a component of $G[X]$. But this is impossible since $G[X]$ is isomorphic to $H$, and no component of $H$ is a complete graph. 
\end{proof} 

\begin{proposition} \label{prop-thickening} Let $G^*$ be a thickening of a graph $G$. Then all the following hold: 
\begin{enumerate}[(a)] 
\item \label{ref-prop-thickening-simplicial} $G^*$ contains a simplicial vertex if and only if $G$ contains a simplicial vertex; 
\item \label{ref-prop-thickening-H-universal} $G^*$ contains a universal vertex if and only if $G$ contains a universal vertex; 
\item \label{ref-prop-thickening-H-anticonn} if $G$ has at least two vertices, then $G^*$ is anticonnected if and only if $G$ is anticonnected; 
\item \label{ref-prop-thickening-H-free} for any graph $H$ that contains no pair of twins, $G^*$ is $H$-free of and only if $G$ is $H$-free. 
\end{enumerate} 
\end{proposition} 
\begin{proof} 
Using the fact that $G^*$ is a thickening of $G$, we fix a family $\{X_v\}_{v \in V(G)}$ of pairwise disjoint, nonempty sets such that all the following hold: 
\begin{itemize} 
\item $V(G^*) = \bigcup_{v \in V(G)} X_v$; 
\item for all $v \in V(G)$, $X_v$ is a clique of $G^*$; 
\item for all distinct $u,v \in V(G)$, the following hold: 
\begin{itemize} 
\item if $u,v$ are adjacent in $G$, then $X_u$ and $X_v$ are complete to each other in $G^*$; 
\item if $u,v$ are nonadjacent in $G$, then $X_u$ and $X_v$ are anticomplete to each other in $G^*$. 
\end{itemize} 
\end{itemize} 

We first prove~(\ref{ref-prop-thickening-simplicial}-\ref{ref-prop-thickening-H-universal}). If $v$ is a simplicial (respectively: universal) vertex of $G$, then clearly, any vertex of $X_v$ is simplicial (respectively: universal) in $G^*$. On the other hand, if a vertex $x$ is simplicial (respectively: universal) of $G^*$, then we fix $v \in V(G)$ such that $x \in X_v$, and we observe that $v$ is a simplicial (respectively: universal) vertex of $G$. This proves~(\ref{ref-prop-thickening-simplicial}-\ref{ref-prop-thickening-H-universal}). 

Next, we prove~(\ref{ref-prop-thickening-H-anticonn}). Assume that $G$ has at least two vertices. We will prove that $G^*$ is not anticonnected if and only if $G$ is not anticonnected. 

Suppose first that $G$ is not anticonnected. Then $V(G)$ can be partitioned into nonempty sets $A$ and $B$ that are complete to each other in $G$. But then $\bigcup_{v \in A} X_v$ and $\bigcup_{v \in B} X_v$ form a partition of $V(G^*)$ into nonempty sets, complete to each other in $G^*$, and we deduce that $G^*$ is not anticonnected. 

Suppose now that $G^*$ is not anticonnected. Then $V(G^*)$ can be partitioned into nonempty sets $A^*$ and $B^*$, complete to each other in $G^*$. If there exists some $v \in V(G)$ such that $X_v$ intersects both $A^*$ and $B^*$, then $v$ is a universal vertex of $G$,\footnote{Indeed, fix any $u \in V(G) \setminus \{v\}$. Then $X_u$ intersects at least one of $A^*$ and $B^*$. By symmetry, we may assume that $X_u \cap A^* \neq \emptyset$. Since $A^*$ is complete to $B^*$ in $G^*$, we know that $X_u \cap A^*$ is complete to $X_v \cap B^*$ in $G^*$. Since $X_u \cap A^*$ and $X_v \cap B^*$ are both nonempty, we see that there is at least one edge between $X_u$ and $X_v$ in $G^*$. So, by definition, $u$ is adjacent to $v$.} which (since $G$ has at least two vertices) implies that $G$ is not anticonnected. So, assume that no such vertex $v$ exists. Then $A := \{v \in V(G) \mid X_v \subseteq A^*\}$ and $B := \{v \in V(G) \mid X_v \subseteq B^*\}$ form a partition of $V(G)$ into nonempty sets, complete to each other in $G$. So, $G$ is not anticonnected. This proves~(\ref{ref-prop-thickening-H-anticonn}). 

It remains to prove~(\ref{ref-prop-thickening-H-free}). Fix a graph $H$ that contains no pair of twins. Clearly, $G$ is (isomorphic to) an induced subgraph of $G^*$, and so if $G^*$ is $H$-free, so is $G$. Now, suppose that $G^*$ is not $H$-free; we must show that $G$ is not $H$-free either. Fix $X \subseteq V(G^*)$ such that $G^*[X]$ is isomorphic to $H$. Note that for all $v \in V(G)$, any two distinct vertices of $X_v$ are twins in $G^*$. Since $H$ contains no pair of twins, we deduce that $|X \cap X_v| \leq 1$ for all $v \in V(G)$. Now, for all $v \in V(G)$ such that $|X \cap X_v| = 1$, let $v^*$ be the unique vertex of $X \cap X_v$. On the other hand, for all $v \in V(G)$ such that $X \cap X_v = \emptyset$, let $v^*$ be an arbitrarily chosen vertex of $X_v$. Set $X^* := \{v^* \mid v \in V(G)\}$. Then $G[X^*]$ is isomorphic to $G$ and contains $G^*[X]$ as an induced subgraph. Since $G^*[X]$ is isomorphic to $H$, we deduce that $G$ is not $H$-free. This proves~(\ref{ref-prop-thickening-H-free}). 
\end{proof}

\section{$\boldsymbol{(P_7,C_4,C_6)}$-free graphs that contain an induced $\boldsymbol{C_7}$} \label{sec:withC7}

In this section, we give a full structural description of $(P_7,C_4,C_6)$-free graphs that contain an induced $C_7$ and do not admit a clique-cutset (see Theorem~\ref{thm-main-withC7}). As mentioned in the introduction, it turns out that these graphs are exactly the same as the $(2P_3,C_4,C_6)$-free graphs that contain an induced $C_7$ and contain no simplicial vertices, which are in turn exactly the same as the $(4K_1,C_4,C_6)$-free graphs that contain an induced $C_7$. We remind the reader that the structure of $(4K_1,C_4,C_6)$-free graphs that contain an induced $C_7$ was originally described in~\cite{ChinhHoang}. 

In this section, we further fully describe the structure of all $(2P_3,C_4,C_6)$-free graphs that contain an induced $C_7$ (see Theorem~\ref{thm-7-saucer}, and note that this theorem also applies to graphs that may possibly contain simplicial vertices).

\subsection{The special partition and the 7-saucer} \label{subsec:SpecPartDef}

A {\em special partition} of a graph $G$ is a partition 
$(X_0,\dots,X_6;Y_0,\dots,Y_6;Z_0,\dots,Z_6;W)$ of $V(G)$ into cliques such that all the following are satisfied:\footnote{Indices are understood to be in $\mathbb{Z}_7$.} 
\begin{enumerate}[(a)]
\item cliques $X_0,\dots,X_6$ are all nonempty;\footnote{Cliques $Y_0,\dots,Y_6,Z_0,\dots,Z_6,W$ may possibly be empty.} 
\item for all $i \in \mathbb{Z}_7$, $X_i$ is complete to $X_{i-1},X_{i+1}$ and anticomplete to $X_{i+2},X_{i+3},X_{i+4},X_{i+5}$; 
\item for all $i \in \mathbb{Z}_7$, $X_i$ is complete to $Y_i,Y_{i+3},Y_{i+6},Z_i,Z_{i+3},Z_{i+4},Z_{i+5},Z_{i+6},W$ and anticomplete to $Y_{i+1},Y_{i+2},Y_{i+4},Y_{i+5},Z_{i+1},Z_{i+2}$; 
\item \label{item:typo} for all $i \in \mathbb{Z}_7$, if $Y_i \neq \emptyset$, then $Y_{i+1},Y_{i+2},Y_{i+5},Y_{i+6},Z_{i+5},Z_{i+6}$ are all empty, and at most one of $Y_{i+3},Y_{i+4}$ is nonempty;\footnote{In particular, at most two of the cliques $Y_0,\dots,Y_6$ are nonempty.} 
\item for all $i \in \mathbb{Z}_7$, if $Z_i \neq \emptyset$, then $Z_{i+2},Z_{i+5}$ are empty;\footnote{In particular, at most three of the cliques $Z_0,\dots,Z_6$ are nonempty.} 
\item for all $i \in \mathbb{Z}_7$, $Y_i$ is complete to $Y_{i+3},Y_{i+4},Z_i,Z_{i+1},Z_{i+3},Z_{i+4},W$ and anticomplete to $Z_{i+2}$; 
\item for all $i \in \mathbb{Z}_7$, $Z_i$ is complete to $Z_{i+1},Z_{i+3},Z_{i+4},Z_{i+6},W$. 
\end{enumerate} 

Note that if $(X_0,\dots,X_6;Y_0,\dots,Y_6;Z_0,\dots,Z_6;W)$ is a special partition of a graph $G$, then by definition, the following hold: 
\begin{itemize} 
\item for all $i \in \mathbb{Z}_7$, $Y_i$
is complete to $X_i,X_{i+1},X_{i+4}$ and anticomplete to
$X_{i+2},X_{i+3},X_{i+5},X_{i+6}$; 
\item for all $i \in \mathbb{Z}_7$, $Z_i$
is complete to $X_i,X_{i+1},X_{i+2},X_{i+3},X_{i+4}$ and anticomplete to $X_{i+5},X_{i+6}$. 
\item the clique $W$ is complete to $V(G) \setminus W = (X_0 \cup \dots \cup X_6) \cup (Y_0 \cup \dots \cup Y_6) \cup (Z_0 \cup \dots \cup Z_6)$, and consequently, all vertices in $W$ are universal vertices of $G$.\footnote{However, it is possible that $W = \emptyset$, in which case, $G$ contains no universal vertices.}
\end{itemize} 

The following lemma may help the reader gain a more intuitive understanding of the special partition. (For further intuition regarding the special partition, see the definition of the family $\mathcal{M}$ in subsection~\ref{subsec:familyM}, and then see Proposition~\ref{prop-spec-part-thickening-MM}.)

\begin{lemma} [Lemma~7.2 of~\cite{4K1C4C6Alg}] \label{lemma-YiZi-in-special-part} Let $(X_0,\dots,X_6;Y_0,\dots,Y_6;Z_0,\dots,Z_6;W)$ be a special partition of a graph $G$, and set $Y := \bigcup_{i \in \mathbb{Z}_7} Y_i$ and $Z := \bigcup_{i \in \mathbb{Z}_7} Z_i$. Then exactly one of the following holds: 
\begin{enumerate}[(a)] 
\item there exists an index $i \in \mathbb{Z}_7$ such that $Y = Y_i \cup Y_{i+3}$ and $Z = Z_i \cup Z_{i+3} \cup Z_{i+4}$; 
\item there exists an index $i \in \mathbb{Z}_7$ such that all the following hold: 
\begin{itemize} 
\item $Y = Y_i$ and $Z = Z_{i+1} \cup Z_{i+2} \cup Z_{i+3}$, 
\item $Y_i,Z_{i+2}$ are both nonempty, 
\item at most one of $Z_{i+1},Z_{i+3}$ is nonempty. 
\end{itemize} 
\end{enumerate} 
Moreover, all the following hold: 
\begin{itemize} 
\item $Y$ and $Z$ are both cliques; 
\item if (a) holds, then $Y \cup Z$ is a clique; 
\item if (b) holds, then $Y \cup Z$ is not a clique. 
\end{itemize} 
\end{lemma}

It was shown in~\cite{ChinhHoang} that every $(4K_1,C_4,C_6)$-graph that contains an induced $C_7$ admits a special partition.\footnote{The term ``special partition'' was actually introduced in~\cite{4K1C4C6Alg}, not in~\cite{ChinhHoang}. However, as the reader can check, the structure described in~\cite{ChinhHoang} is precisely our ``special partition.''} However, our proof does not actually rely on this result. Instead, we only use the converse, obtained via routine case checking in~\cite{4K1C4C6Alg} and stated below. 

\begin{lemma} [Lemma~7.3(a) of~\cite{4K1C4C6Alg}] \label{lemma-spec-part-4K1C4C6FreeWithC7} Any graph that admits a special partition is $(4K_1,C_4,C_6)$-free and contains an induced $C_7$. 
\end{lemma} 

A {\em 7-saucer} is a graph $G$ whose vertex set can be partitioned into sets $X_0,\dots,X_6,Y_0,\dots,Y_6,$ $Z_0,\dots,Z_6,W,A$ (with indices undertsood to be in $\mathbb{Z}_7$) such that all the following hold: 
\begin{itemize} 
\item $(X_0,\dots,X_6;Y_0,\dots,Y_6;Z_0,\dots,Z_6;W)$ is a special partition of $G \setminus A$; 
\item $A$ is anticomplete to $X_0,\dots,X_6$;\footnote{Thus, $N_G(A) \subseteq (Y_0 \cup \dots \cup Y_6) \cup (Z_0 \cup \dots \cup Z_6) \cup W$.} 
\item for all $i \in \mathbb{Z}_7$, either $A$ is anticomplete to $Y_i$, or $Z_{i+2} = \emptyset$; 
\item either $A = \emptyset$, or $A$ can be partitioned into nonempty cliques $A_1,\dots,A_{\ell}$, pairwise anticomplete to each other, such that for all $i \in \{1,\dots,\ell\}$, $A_i$ can be ordered as $A_i = \{a_1^i,\dots,a_{r_i}^i\}$ so that $N_G[a_{r_i}^i] \subseteq \dots \subseteq N_G[a_1^i]$. 
\end{itemize} 
Under these circumstances, we also say that $(X_0,\dots,X_6;Y_0,\dots,Y_6;Z_0,\dots,Z_6;W;A)$ is a {\em 7-saucer partition} of the 7-saucer $G$. 

One of the main goals of this section is to prove that a graph is $(2P_3,C_4,C_6)$-free and contains an induced $C_7$ if and only if it is a 7-saucer (see Theorem~\ref{thm-7-saucer}; for the ``if'' part, see Proposition~\ref{prop-7-saucer-in-class}, and for the ``only if'' part, see Lemma~\ref{lemma-main-withC7}).

\begin{proposition} \label{prop-7-saucer-NGA-clique} Let $G$ be a 7-saucer, and let $(X_0,\dots,X_6;Y_0,\dots,Y_6;Z_0,\dots,Z_6;W;A)$ be an associated 7-saucer partition of $G$. Set $X := \bigcup_{i \in \mathbb{Z}_7} X_i$, $Y := \bigcup_{i \in \mathbb{Z}_7} Y_i$ and $Z := \bigcup_{i \in \mathbb{Z}_7} Z_i$. Then all the following hold: 
\begin{enumerate}[(a)] 
\item $N_G(A) \subseteq Y \cup Z \cup W$; 
\item $N_G(A)$ is a clique; 
\item for all $v \in N_G(A)$, the set $V(G) \setminus N_G[v]$ can be partitioned into cliques, pairwise anticomplete to each other.\footnote{Here, it is possible that $V(G) \setminus N_G[v] = \emptyset$, i.e.\ that $V(G) = N_G[v]$, so that $v$ is a universal vertex of $G$.} 
\end{enumerate} 
\end{proposition} 
\begin{proof}
In our proof, we will repeatedly use the fact that $(X_0,\dots,X_6;Y_0,\dots,Y_6;Z_0,\dots,Z_6;W)$ is a special partition of $G \setminus A$, as per the definition of a 7-saucer. 

For~(a), we simply observe that $A$ is anticomplete to $X$ (by the definition of a 7-saucer), and so, $N_G(A) \subseteq Y \cup Z \cup W$. 

Next, we prove~(b). If $Y \cup Z \cup W$ is a clique, then we are done by~(a). We may therefore assume that $Y \cup Z \cup W$ is not a clique. By the definition of a special partition, $W$ is a clique, complete to $X \cup Y \cup Z$. So, since $Y \cup Z \cup W$ is not a clique, we know that $Y \cup Z$ is not a clique. By Lemma~\ref{lemma-YiZi-in-special-part} and by symmetry, we may now assume that all the following hold: 
\begin{itemize} 
\item $Y = Y_0$ and $Z = Z_1 \cup Z_2 \cup Z_3$, 
\item $Y_0,Z_2$ are both nonempty, 
\item at most one of $Z_1,Z_3$ is nonempty. 
\end{itemize} 
Since $Z_2$ is nonempty, the definition of a 7-saucer guarantees that $A$ is anticomplete to $Y_0$. We now deduce that $N_G(A) \subseteq Z_1 \cup Z_2 \cup Z_3 \cup W$. But by the definition of a special partition, $Z_1 \cup Z_2 \cup Z_3 \cup W$ is a clique; consequently, $N_G(A)$ is also a clique. This proves~(b). 

It remains to prove~(c). Fix a vertex $v \in N_G(A)$. We must show that $V(G) \setminus N_G[v]$ can be partitioned into cliques, pairwise anticomplete to each other. Recall that $N_G(A) \subseteq Y \cup Z \cup W$; thus, $v \in Y \cup Z \cup W$. 

Suppose first that $v \in W$. By the definitions of a 7-saucer and a special partition, $W$ is a clique, complete to $V(G) \setminus (W \cup A)$. So, $V(G) \setminus N_G[v] \subseteq A$. But by the definition of a 7-saucer, $A$ can be partitioned into (possibly empty) cliques, pairwise anticomplete to each other. 

Next, suppose that $v \in Y$. By symmetry, we may assume that $v \in Y_0$. Since $v \in N_G(A) \cap Y_0$, we see that $A$ is not anticomplete to $Y_0$, and so by the definition of a 7-saucer, we know that $Z_2 = \emptyset$. Then the definitions of a 7-saucer and a special partition imply that $Y_0$ is a clique, complete to $(X_0 \cup X_1 
\cup X_4) (Y \setminus Y_0) \cup Z \cup W$, and furthermore, that $V(G) \setminus N_G[v] \subseteq (X_2 \cup X_3) \cup (X_5 \cup X_6) \cup A$. But once again by the definitions of a 7-saucer and a special partition, all the following hold: 
\begin{itemize} 
\item $A$ can be partitioned into (possibly empty) cliques, pairwise anticomplete to each other; 
\item $A$ is anticomplete to $X$ (and therefore, to $X_2 \cup X_3$ and $X_5 \cup X_6$ as well); 
\item $X_2 \cup X_3$ and $X_5 \cup X_6$ are disjoint cliques, anticomplete to each other. 
\end{itemize} 
Thus, $V(G) \setminus N_G[v]$ can be partitioned into cliques, pairwise anticomplete to each other. 

Finally, suppose that $v \in Z$. By symmetry, we may assume that $v \in Z_2$. In particular, $Z_2 \neq \emptyset$, and so by the definition of a 7-saucer, $A$ is anticomplete to $Y_0$. On the other hand, the definition of a special partition implies that $Z_2$ is a clique, complete to $(X_2 \cup X_3 \cup X_4 \cup X_5 \cup X_6) \cup (Y \setminus Y_0) \cup (Z \setminus Z_2) \cup W$. Thus, $V(G) \setminus N_G[v] \subseteq X_0 \cup X_1 \cup Y_0 \cup A$. It is now enough to show that $X_0 \cup X_1 \cup Y_0 \cup A$ can be partitioned into cliques, pairwise anticomplete to each other. By the definition of a special partition, $X_0 \cup X_1 \cup Y_0$ is a clique. Next, the definition of a 7-saucer guarantees that $A$ is anticomplete to $X_0 \cup X_1$, and we already saw that $A$ is anticomplete to $Y_0$. Finally, the definition of a special partition guarantees that $A$ can be partitioned into (possibly empty) cliques, pairwise anticomplete to each other. This completes the proof of~(c). 
\end{proof}

\begin{proposition} \label{prop-7-saucer-in-class} 
Every 7-saucer is $(2P_3,C_4,C_6)$-free and contains an induced $C_7$. 
\end{proposition} 
\begin{proof} 
Fix a 7-saucer $G$, and let $(X_0,\dots,X_6;Y_0,\dots,Y_6;Z_0,\dots,Z_6;W;A)$ be an associated 7-saucer partition of $G$. By definition, $(X_0,\dots,X_6;Y_0,\dots,Y_6;Z_0,\dots,Z_6;W)$ is a special partition of $G \setminus A$, and so by Lemma~\ref{lemma-spec-part-4K1C4C6FreeWithC7}, $G \setminus A$ is $(4K_1,C_4,C_6)$-free and contains an induced $C_7$. Since $G \setminus A$ contains an induced $C_7$, so does $G$. It remains to show that $G$ is $(2P_3,C_4,C_6)$-free. 

\begin{adjustwidth}{1cm}{1cm} 
\begin{claim} \label{prop-7-saucer-in-class-claim-C4C6-free} 
Every hole of $G$ is in fact a hole of $G \setminus A$. Consequently, $G$ is $(C_4,C_6)$-free. 
\end{claim} 
\end{adjustwidth} 
{\em Proof of Claim~\ref{prop-7-saucer-in-class-claim-C4C6-free}.} Fix a hole $v_0,v_1,\dots,v_{k-1},v_0$ (with $k \geq 4$, and with indices understood to be in $\mathbb{Z}_k$) in $G$. We must show that $\{v_0,v_1,\dots,v_{k-1}\} \cap A = \emptyset$. Suppose otherwise. Let $B$ be the vertex set of a component of $G[A]$ such that $\{v_0,v_1,\dots,v_{k-1}\} \cap B \neq \emptyset$. By the definition of a 7-saucer, $B$ is a clique, and there exists an ordering $B = \{b_1,\dots,b_r\}$ of $B$ such that $N_G[b_r] \subseteq \dots \subseteq N_G[b_1]$. Since $N_G(A)$ is clique (by Proposition~\ref{prop-7-saucer-NGA-clique}), we see that for all $i \in \{1,\dots,r\}$, $b_i$ is simplicial in $G \setminus \{b_{i+1},\dots,b_r\}$. Now, fix the largest index $i \in \{1,\dots,r\}$ such that $b_i \in \{v_0,v_1,\dots,v_{k-1}\}$; by symmetry, we may assume that $b_i = v_0$. Then $v_0,v_1,\dots,v_{k-1},v_0$ is a hole in $G \setminus \{b_{i+1},\dots,b_r\}$, and we deduce that $v_0 = b_i$ is a simplicial vertex in the hole $v_0,v_1,\dots,v_{k-1},v_0$, contrary to the fact that holes contain no simplicial vertices. This proves that every hole of $G$ is in fact a hole of $G \setminus A$. Since $G \setminus A$ is $(C_4,C_6)$-free, we deduce that $G$ is also $(C_4,C_6)$-free.~$\blacklozenge$

\begin{adjustwidth}{1cm}{1cm} 
\begin{claim} \label{prop-7-saucer-in-class-claim-2P3-free} 
$G$ is $2P_3$-free. 
\end{claim} 
\end{adjustwidth} 
{\em Proof of Claim~\ref{prop-7-saucer-in-class-claim-2P3-free}.} By Propositions~\ref{prop-G-x-disjoint-cliques-2P3-free} and~\ref{prop-7-saucer-NGA-clique}, it suffices to show that $G' := G \setminus N_G(A)$ is $2P_3$-free. First, by the definition of a 7-saucer, $A$ is either empty or can be partitioned into cliques, pairwise anticomplete to each other. Moreover, it is clear that $A$ is anticomplete to $V(G') \setminus A = V(G) \setminus N_G[A]$ in $G'$. So, by Proposition~\ref{prop-add-complete-components}(b), it is enough to show that $G' \setminus A = G \setminus N_G[A]$ is $2P_3$-free; since $G' \setminus A = G \setminus N_G[A]$ is an induced subgraph of $G \setminus A$, and since $4K_1$ is an induced subgraph of $2P_3$, it in fact suffices to show that $G \setminus A$ is $4K_1$-free. But this follows from Lemma~\ref{lemma-spec-part-4K1C4C6FreeWithC7}, since by the definition of a 7-saucer, $G \setminus A$ admits a special partition, namely $(X_0,\dots,X_6;Y_0,\dots,Y_6;Z_0,\dots,Z_6;W)$.~$\blacklozenge$  

\medskip 

Claims~\ref{prop-7-saucer-in-class-claim-C4C6-free} and~\ref{prop-7-saucer-in-class-claim-2P3-free} together guarantee that $G$ is $(2P_3,C_4,C_6)$-free, which completes the argument. \end{proof} 

\subsection{The family $\boldsymbol{\mathcal{M}}$} \label{subsec:familyM}

In this subsection, we define the family $\mathcal{M}$ of graphs, which we need in order to properly state Theorem~\ref{thm-main-withC7}, our structure theorem for $(P_7,C_4,C_6)$-free graphs that contain an induced $C_7$ and do not admit a clique-cutset. We first define graphs $M_0$, $M_1$, $M_2$, and $M_3$, as follows. 

$M_0$ is the graph with vertex set $\{x_0,x_1,x_2,x_3,x_4,x_5,x_6,y_0,y_3,z_0,z_3,z_4\}$ and with adjacency as follows: 
\begin{itemize} 
\item $x_0,x_1,x_2,x_3,x_4,x_5,x_6,x_0$ is a 7-hole; 
\item $y_0$ is complete to $\{x_0,x_1,x_4\}$ and anticompelte to $\{x_2,x_3,x_5,x_6\}$; 
\item $y_3$ is complete to $\{x_3,x_4,x_0\}$ and anticomplete to $\{x_1,x_2,x_5,x_6\}$; 
\item $z_0$ is complete to $\{x_0,x_1,x_2,x_3,x_4\}$ and anticomplete to $\{x_5,x_6\}$; 
\item $z_3$ is complete to $\{x_3,x_4,x_5,x_6,x_0\}$ and anticompelte to $\{x_1,x_2\}$; 
\item $z_4$ is complete to $\{x_4,x_5,x_6,x_0,x_1\}$ and anticompelte to $\{x_2,x_3\}$; 
\item $\{y_0,y_3,z_0,z_3,z_4\}$ is a clique. 
\end{itemize} 

$M_1$ is the graph with vertex set $\{x_0,x_1,x_2,x_3,x_4,x_5,x_6,y_0,z_2\}$ and with adjacency as follows: 
\begin{itemize} 
\item $x_0,x_1,x_2,x_3,x_4,x_5,x_6,x_0$ is a 7-hole; 
\item $y_0$ is complete to $\{x_0,x_1,x_4\}$ and anticomplete to $\{x_2,x_3,x_5,x_6\}$; 
\item $z_2$ is complete to $\{x_2,x_3,x_4,x_5,x_6\}$ and anticomplete to $\{x_0,x_1\}$; 
\item $y_0$ and $z_2$ are nonadjacent.  
\end{itemize} 

$M_2$ is the graph with vertex set $\{x_0,x_1,x_2,x_3,x_4,x_5,x_6,y_0,z_1,z_2\}$ and with adjacency as follows: 
\begin{itemize} 
\item $M_2 \setminus z_1 = M_1$; 
\item $z_1$ is complete to $\{x_1,x_2,x_3,x_4,x_5\}$ and anticomplete to $\{x_6,x_0\}$; 
\item $z_1$ is complete to $\{y_0,z_2\}$. 
\end{itemize} 

$M_3$ is the graph with vertex set $\{x_0,x_1,x_2,x_3,x_4,x_5,x_6,y_0,z_2,z_3\}$ and with adjacency as follows: 
\begin{itemize} 
\item $M_3 \setminus z_3 = M_1$; 
\item $z_3$ is complete to $\{x_3,x_4,x_5,x_6,x_0\}$ and anticomplete to $\{x_1,x_2\}$; 
\item $z_3$ is complete to $\{y_0,z_2\}$. 
\end{itemize} 

We now define the family $\mathcal{M}$, as follows: 
\begin{displaymath} 
\begin{array}{rcl} 
\mathcal{M} & := & \big\{M_0 \setminus S \mid S \subseteq \{y_0,y_3,z_0,z_3,z_4\}\big\} \cup \{M_1,M_2,M_3\}. 
\end{array} 
\end{displaymath} 
Thus, $\mathcal{M}$ is the collection of graphs whose members are $M_1,M_2,M_3$, plus all induced subgraphs of $M_0$ that still contain the hole $x_0,x_1,x_2,x_3,x_4,x_5,x_6,x_0$. We note that each graph in $\mathcal{M}$ has at most 12 vertices and contains an induced $C_7$.

\begin{proposition} \label{prop-MM-anticonn} Every graph in $\mathcal{M}$ is anticonnected and contains no simplicial vertices and no universal vertices. Consequently, any thickening of a graph in $\mathcal{M}$ is anticonnected contains no simplicial vertices and no universal vertices. 
\end{proposition} 
\begin{proof} 
In view of Proposition~\ref{prop-thickening}, it is enough to prove the first statement. So, fix any graph $M \in \mathcal{M}$; we must show that $M$ is anticonnected and contains no simplicial vertices and no universal vertices. First of all, $x_0,x_1,\dots,x_6,x_0$ is a 7-hole in $M$. Since $\overline{C_7}$ is connected, we see that $M[x_0,x_1,\dots,x_6]$ is anticonnected. Moreover, by the definition of $\mathcal{M}$, every vertex in $V(M) \setminus \{x_0,x_1,\dots,x_6\}$ has a nonneighbor in $\{x_0,x_1,\dots,x_6\}$. Thus, $M$ is anticonnected. Since $M$ has more than one vertex, this implies, in particular, that $M$ contains no universal vertices. It remains to show that $M$ contains no simplicial vertices. First, for all $i \in \mathbb{Z}_7$, $x_i$ is complete to $\{x_{i-1},x_{i+1}\}$, and so since $x_{i-1},x_{i+1}$ are nonadjacent, $x_i$ is not a simplicial vertex of $M$. On the other hand, by the definition of the family $\mathcal{M}$, every vertex in $V(M) \setminus \{x_0,x_1,\dots,x_6\}$ has at least three neighbors in $\{x_0,x_1,\dots,x_6\}$; since $M[x_0,x_1,\dots,x_6]$ is triangle-free, we deduce that no vertex in $V(M) \setminus \{x_0,x_1,\dots,x_6\}$ is simplicial. Therefore, $M$ contains no simplicial vertices. 
\end{proof}

\begin{proposition} \label{prop-spec-part-thickening-MM} For any graph $G$, the following are equivalent: 
\begin{enumerate}[(a)] 
\item $G$ admits a special partition; 
\item $G$ has exactly one nontrivial anticomponent, and this anticomponent is a thickening of a graph in $\mathcal{M}$; 
\item $G$ can be obtained from a thickening of a graph in $\mathcal{M}$ by possibly adding universal vertices to it.  
\end{enumerate} 
\end{proposition} 
\begin{proof} By Proposition~\ref{prop-MM-anticonn}, any thickening of a graph in $\mathcal{M}$ is anticonnected. So, by Proposition~\ref{prop-non-trivial-anticomp-univ-vertices}, (b) and (c) are equivalent. It remains to prove that (a) and (c) are equivalent. 

Suppose first that (a) holds. Fix a special partition $(X_0,\dots,X_6;Y_0,\dots,Y_6;Z_0,\dots,Z_6;W)$ of $G$. Then $W$ is a clique, complete to $V(G) \setminus W$. Thus, $G$ can be obtained from $G \setminus W$ by possibly adding universal vertices to it. But Lemma~\ref{lemma-YiZi-in-special-part} and the relevant definitions immediately imply that $G \setminus W$ is a thickening of a graph in $\mathcal{M}$. So, (c) holds. 

Suppose conversely that (c) holds, so that $G$ can be obtained from a thickening $M^*$ of a graph $M \in \mathcal{M}$ by possibly adding universal vertices to $M^*$. By the definition of $\mathcal{M}$ and of the special partition, it is clear that $M^*$ admits a special partition $(X_0,\dots,X_6;Y_0,\dots,Y_6;Z_0,\dots,Z_6;\emptyset)$. But now for $W := V(G) \setminus V(M^*)$, we have that $W$ is a clique, complete to $V(M^*) = V(G) \setminus W$, and we deduce that $(X_0,\dots,X_6;Y_0,\dots,Y_6;Z_0,\dots,Z_6;W)$ is a special partition of $G$. So, (a) holds. 
\end{proof}

\subsection{Decomposing $\boldsymbol{(P_7,C_4,C_6)}$-free graphs that contain an induced $\boldsymbol{C_7}$}

\begin{lemma} \label{lemma-attach-to-C7} Let $G$ be a $(P_7,C_4,C_6)$-free graph, let $x_0,x_1,x_2,\dots,x_6,x_0$ (with indices in $\mathbb{Z}_7$) be a 7-hole in $G$, and let $v \in V(G) \setminus \{x_0,x_1,x_2,\dots,x_6\}$. Then one of the following holds: 
\begin{enumerate}[(a)] 
\item \label{ref-withC7-attach-ac} $v$ is anticomplete to $\{x_0,x_1,x_2,\dots,x_6\}$; 
\item \label{ref-withC7-attach-twin} there exists an index $i \in \mathbb{Z}_7$ such that $N_G(v) \cap \{x_0,x_1,x_2,\dots,x_6\} = \{x_{i-1},x_i,x_{i+1}\}$; 
\item \label{ref-withC7-attach-Yi} there exists an index $i \in \mathbb{Z}_7$ such that $N_G(v) \cap \{x_0,x_1,x_2,\dots,x_6\} = \{x_i,x_{i+1},x_{i+4}\}$; 
\item \label{ref-withC7-attach-Zi} there exists an index $i \in \mathbb{Z}_7$ such that $N_G(v) \cap \{x_0,x_1,x_2,\dots,x_6\} = \{x_i,x_{i+1},x_{i+2},x_{i+3},x_{i+4}\}$; 
\item \label{ref-withC7-attach-c} $v$ is complete to $\{x_0,x_1,x_2,\dots,x_6\}$. 
\end{enumerate} 
\end{lemma} 
\begin{proof} 
To simplify notation, set $X := \{x_0,x_1,x_2,\dots,x_6\}$. If $v$ is anticomplete to $X$, then~(\ref{ref-withC7-attach-ac}) holds, and we are done. Next, if $v$ has exactly one neighbor in $X$, then by symmetry, we may assume that $N_G(v) \cap X = \{x_0\}$, and we see that $v,x_0,x_1,x_2,x_3,x_4,x_5$ is an induced 7-vertex path in $G$, contrary to the fact that $G$ is $P_7$-free. Next, suppose that $v$ has exactly two neighbors in $X$; by symmetry, we may assume that $N_G(v) \cap X = \{x_0,x_k\}$ for some $k \in \{1,2,3\}$. But if $k = 1$, then $v,x_1,x_2,x_3,x_4,x_5,x_6$ is an induced 7-vertex path in $G$; if $k = 2$, then $v,x_0,x_1,x_2,v$ is a 4-hole in $G$; and if $k = 3$, then $v,x_3,x_4,x_5,x_6,x_0,v$ is a 6-hole in $G$. Since $G$ is $(P_7,C_4,C_6)$-free, none of these three outcomes is possible. 

From now on, we may assume that $3 \leq |N_G(v)| \leq 6$. Let us define a {\em sector} to be a nontrivial induced path in the hole $x_0,x_1,\dots,x_6,x_0$ such that $v$ is adjacent to both endpoints of the path and is anticomplete to the interior of the path. The {\em length} of a sector is simply its length as a path (i.e.\ the number of edges that it contains). Note that $v$, together with the vertices of any sector of length $\ell$, induces a cycle of length $\ell+2$ in $G$. Since $G$ is $(P_7,C_4,C_6)$-free, we know that all induced cycles of $G$ are of length 3, 5, or 7; therefore, all sectors are of length 1, 3, or 5. Meanwhile, since every edge of the hole $x_0,x_1,\dots,x_6,x_0$ belongs to exactly one sector, we see that the sum of lengths of all the sectors is 7. So, if the number of sectors of length 1, 3, and 5 is $a$, $b$, and $c$, respectively, then we have that $a+3b+5c = 7$. But the equation $a+3b+5c = 7$ has (exactly) the following four nonnegative integer solutions: 
\begin{enumerate}[(1)] 
\item $a = 2$, $b = 0$, $c = 1$; 
\item $a = 1$, $b = 2$, $c = 0$; 
\item $a = 4$, $b = 1$, $c = 0$; 
\item $a = 7$, $b = 0$, $c = 0$.  
\end{enumerate} 
Clearly, these four solutions correspond (respectively) to the outcomes (\ref{ref-withC7-attach-twin}), (\ref{ref-withC7-attach-Yi}), (\ref{ref-withC7-attach-Zi}), (\ref{ref-withC7-attach-c}) of the lemma. This completes the argument. 
\end{proof}

\begin{lemma} \label{lemma-main-withC7} Let $G$ be a $(P_7,C_4,C_6)$-free graph that contains an induced $C_7$. Then either $G$ admits a special partition, or $G$ admits a clique-cutset. Moreover, if $G$ is additionally $2P_3$-free, then $G$ is a 7-saucer. 
\end{lemma} 
\begin{proof} 
Let $x_0,x_1,x_2,x_3,x_4,x_5,x_6,x_0$ (with indices in $\mathbb{Z}_7$) be a 7-hole in $G$, and set $\widehat{X} := \{x_0,x_1,x_2,\dots,x_6\}$. Let $W$ be the set of all vertices in $V(G) \setminus \widehat{X}$ that are complete to $\widehat{X}$, let $A$ be the set of all vertices in $V(G) \setminus \widehat{X}$ that are anticomplete to $\widehat{X}$, and for all $i \in \mathbb{Z}_7$, set 
\begin{itemize} 
\item $X_i := \{x_i\} \cup \big\{v \in V(G) \setminus \widehat{X} \mid N_G(v) \cap \widehat{X} = \{x_{i-1},x_i,x_{i+1}\}\big\}$; 
\item $Y_i := \big\{v \in V(G) \setminus \widehat{X} \mid N_G(v) \cap \widehat{X} = \{x_i,x_{i+1},x_{i+4}\}\big\}$; 
\item $Z_i := \big\{v \in V(G) \setminus \widehat{X} \mid N_G(v) \cap \widehat{X} = \{x_i,x_{i+1},x_{i+2},x_{i+3},x_{i+4}\}\big\}$. 
\end{itemize} 
Furthermore, set $X := \bigcup_{i \in \mathbb{Z}_7} X_i$, $Y := \bigcup_{i \in \mathbb{Z}_7} Y_i$, and $Z := \bigcup_{i \in \mathbb{Z}_7} Z_i$. 

By Lemma~\ref{lemma-attach-to-C7}, $(X_0,\dots,X_6;Y_0,\dots,Y_6;Z_0,\dots,Z_6;W;A)$ is a partition of $V(G)$.\footnote{Consequently, sets $X,Y,Z,W,A$ also form a partition of $V(G)$.} Our goal is to prove all the following: 
\begin{itemize} 
\item $(X_0,\dots,X_6;Y_0,\dots,Y_6;Z_0,\dots,Z_6;W)$ is a special partition of $G \setminus A$ (see Claim~\ref{lemma-main-withC7-claim-spec-part}); 
\item if $A \neq \emptyset$, then $G$ admits a clique-cutset (see Claim~\ref{lemma-main-withC7-claim-NGB-clique}); 
\item if $G$ is $2P_3$-free, then $G$ is a 7-saucer (see Claim~\ref{lemma-main-withC7-claim-2P3-free-7-saucer}). 
\end{itemize} 

Our first goal is to show that $(X_0,\dots,X_6;Y_0,\dots,Y_6;Z_0,\dots,Z_6;W)$ is a special partition of $G \setminus A$. The fact that $X_0,\dots,X_6$ are all nonempty follows from the construction. Further, all the following hold: 
\begin{itemize} 
\item for all $i \in \mathbb{Z}_7$, distinct, nonadjacent vertices $x_{i-1},x_{i+1}$ are complete to $X_i$; 
\item for all $i \in \mathbb{Z}_7$, distinct, nonadjacent vertices $x_i,x_{i+4}$ are complete to $Y_i$; 
\item for all $i \in \mathbb{Z}_7$, distinct, nonadjacent vertices $x_i,x_{i+2}$ are complete to $Z_i$; 
\item distinct, nonadjacent vertices $x_0,x_2$ are complete to $W$. 
\end{itemize} 
In view of Proposition~\ref{prop-non-adj-comp-clique}, this implies that $X_0,\dots,X_6,Y_0,\dots,Y_6,Z_0,\dots,Z_6,W$ are all cliques. It remains to check that items (b)-(g) from the definition of a special partition are satisfied. We prove this via a sequence of claims (see Claims~\ref{lemma-main-withC7-claim-Xi-hyperhole}-\ref{lemma-main-withC7-claim-Zi-attachment} below). 

\begin{adjustwidth}{1cm}{1cm}
\begin{claim} \label{lemma-main-withC7-claim-Xi-hyperhole} For all $i \in \mathbb{Z}_7$, $X_i$ is complete to $X_{i-1},X_{i+1}$ and anticomplete to $X_{i+2},X_{i+3},X_{i+4},X_{i+5}$. 
\end{claim} 
\end{adjustwidth}
\noindent 
{\em Proof of Claim~\ref{lemma-main-withC7-claim-Xi-hyperhole}.} By symmetry, it suffices show that $X_0$ is complete to $X_1$ and anticomplete to $X_2,X_3$. 

If some $x_0' \in X_0$ and $x_1' \in X_1$ are nonadjacent, then $x_1',x_2,x_3,x_4,x_5,x_6,x_0'$ is an induced 7-vertex path in $G$, contrary to the fact that $G$ is $P_7$-free. So, $X_0$ is complete to $X_1$. 

If some $x_0' \in X_0$ is adjacent to some $x_2' \in X_2$, then $x_0',x_2',x_3,x_4,x_5,x_6,x_0'$ is a 6-hole in $G$, contrary to the fact that $G$ is $C_6$-free. Thus, $X_0$ is anticomplete to $X_2$. 

If some $x_0' \in X_0$ is adjacent to some $x_3' \in X_3$, then $x_0',x_1,x_2,x_3',x_0'$ is a 4-hole in $G$, contrary to the fact that $G$ is $C_4$-free. Thus, $X_0$ is anticomplete to $X_3$.~$\blacklozenge$ 

\begin{adjustwidth}{1cm}{1cm}
\begin{claim} \label{lemma-main-withC7-claim-Xi-attachment} For all $i \in \mathbb{Z}_7$, $X_i$ is complete to $Y_i,Y_{i+3},Y_{i+6},Z_i,Z_{i+3},Z_{i+4},Z_{i+5},Z_{i+6},W$ and anticomplete to $Y_{i+1},Y_{i+2},Y_{i+4},Y_{i+5},Z_{i+1},Z_{i+2}$. 
\end{claim} 
\end{adjustwidth} 
\noindent 
{\em Proof of Claim~\ref{lemma-main-withC7-claim-Xi-attachment}.} By symmetry, it suffices to prove the claim for $i = 0$. If some $x_0'$ has a nonneighbor $v \in Y_0 \cup Y_3 \cup Y_6 \cup Z_0 \cup Z_3 \cup Z_4 \cup Z_5 \cup Z_6 \cup W$, then $v$ has either exactly two, or exactly four, or exactly six neighbors in the 7-hole $x_0',x_1,x_2,x_3,x_4,x_5,x_6,x_0'$, contrary to Lemma~\ref{lemma-attach-to-C7}. Thus, $X_0$ is complete to $Y_0,Y_3,Y_6,Z_0,Z_3,Z_4,Z_5,Z_6,W$. 

Similarly, if some $x_0' \in X_0$ has a neighbor $v \in Y_1 \cup Y_2 \cup Y_4 \cup Y_5 \cup Z_1 \cup Z_2$, then $v$ has either exactly four or exactly six neighbors in the 7-hole $x_0',x_1,x_2,x_3,x_4,x_5,x_6,x_0'$, contrary to Lemma~\ref{lemma-attach-to-C7}. Thus, $X_0$ is anticomplete to $Y_1,Y_2,Y_4,Y_5,Z_1,Z_2$.~$\blacklozenge$ 

\begin{adjustwidth}{1cm}{1cm}
\begin{claim} \label{lemma-main-withC7-claim-if-Yi-nonempty}
For all $i \in \mathbb{Z}_7$, if $Y_i \neq \emptyset$, then $Y_{i+1},Y_{i+2},Y_{i+5},Y_{i+6},Z_{i+5},Z_{i+6}$ are all empty, and at most one of $Y_{i+3},Y_{i+4}$ is nonempty. 
\end{claim} 
\end{adjustwidth} 
\noindent 
{\em Proof of Claim~\ref{lemma-main-withC7-claim-if-Yi-nonempty}.} By symmetry, it suffices to prove the claim for $i = 0$. So, assume that $Y_0 \neq \emptyset$. We must show that $Y_1,Y_2,Y_5,Y_6,Z_5,Z_6$ are all empty, and that at most one of $Y_3,Y_4$ is nonempty. Fix $y_0 \in Y_0$. 

Suppose that $Y_1 \neq \emptyset$, and fix $y_1 \in Y_1$. If $y_0$ is adjacent to $y_1$, then $y_0,y_1,x_5,x_4,y_0$ is a 4-hole in $G$, contrary to the fact that $G$ is $C_4$-free. On the other hand, if $y_0$ is nonadjacent to $y_1$, then $y_1,x_2,x_3,x_4,y_0,x_0,x_6$ is an induced 7-vertex path in $G$, contrary to the fact that $G$ is $P_7$-free. This proves that $Y_1 = \emptyset$. An analogous argument establishes that $Y_6 = \emptyset$. 

Next, suppose that $Y_2 \neq \emptyset$, and fix $y_2 \in Y_2$. If $y_0$ is adjacent to $y_2$, then $y_0,x_1,x_2,y_2,y_0$ is a 4-hole in $G$, contrary to the fact that $G$ is $C_4$-free. On the other hand, if $y_0$ is nonadjacent to $y_2$, then $y_0,x_0,x_6,y_2,x_3,x_4,y_0$ is a 6-hole in $G$, contrary to the fact that $G$ is $C_6$-free. This proves that $Y_2 = \emptyset$. An analogous argument establishes that $Y_5 = \emptyset$. 

Next, suppose that $Z_5 \neq \emptyset$, and fix $z_5 \in Z_5$. If $y_0$ is adjacent to $z_5$, then $y_0,x_4,x_5,z_5,y_0$ is a 4-hole in $G$, contrary to the fact that $G$ is $C_4$-free. On the other hand, if $y_0$ is nonadjacent to $z_5$, then $y_0,x_0,z_5,x_2,x_3,x_4,y_0$ is a 6-hole in $G$, contrary to the fact that $G$ is $C_6$-free. This proves that $Z_5 = \emptyset$. An analogous argument establishes that $Z_6 = \emptyset$. 

It remains to show that at most one of $Y_3,Y_4$ is nonempty. But the proof of this is analogous to the above proof of the fact that, if $Y_0 \neq \emptyset$, then $Y_1 = \emptyset$.~$\blacklozenge$ 

\begin{adjustwidth}{1cm}{1cm}
\begin{claim} \label{lemma-main-withC7-claim-if-Zi-nonempty} For all $i \in \mathbb{Z}_7$, if $Z_i \neq \emptyset$, then $Z_{i+2},Z_{i+5}$ are empty. 
\end{claim} 
\end{adjustwidth} 
\noindent 
{\em Proof of Claim~\ref{lemma-main-withC7-claim-if-Zi-nonempty}.} By symmetry, it suffices to prove the claim for $i = 0$. So, assume that $Z_0 \neq \emptyset$, and fix $z_0 \in Z_0$. Suppose that $Z_2 \neq \emptyset$, and fix $z_2 \in Z_2$. If $z_0$ is adjacent to $z_2$, then $z_0,z_2,x_6,x_0,z_0$ is a 4-hole in $G$, contrary to the fact that $G$ is $C_4$-free. On the other hand, if $z_0$ is nonadjacent to $z_2$, then $z_0,x_2,z_2,x_4,z_0$ is a 4-hole in $G$, contrary to the fact that $G$ is $C_4$-free. This proves that $Z_2 = \emptyset$. An analogous argument establishes that $Z_5 = \emptyset$.~$\blacklozenge$

\begin{adjustwidth}{1cm}{1cm}
\begin{claim} \label{lemma-main-withC7-claim-Yi-attachment} 
For all $i \in \mathbb{Z}_7$, $Y_i$ is complete to $Y_{i+3},Y_{i+4},Z_i,Z_{i+1},Z_{i+3},Z_{i+4},W$ and anticomplete to $Z_{i+2}$. 
\end{claim} 
\end{adjustwidth} 
\noindent 
{\em Proof of Claim~\ref{lemma-main-withC7-claim-Yi-attachment}.} By symmetry, it suffices to prove the claim for $i = 0$. Note that: 
\begin{itemize} 
\item distinct, nonadjacent vertices $x_0,x_4$ are complete to $Y_0 \cup Y_3 \cup Z_0 \cup Z_3 \cup Z_4 \cup W$; 
\item distinct, nonadjacent vertices $x_1,x_4$ are complete to $Y_0 \cup Y_4 \cup Z_1$. 
\end{itemize} 
So, by Proposition~\ref{prop-non-adj-comp-clique}, $Y_0 \cup Y_3 \cup Z_0 \cup Z_3 \cup Z_4 \cup W$ and $Y_0 \cup Y_4 \cup Z_1$ are both cliques, and it follows that $Y_0$ is complete to $Y_3,Y_4,Z_0,Z_1,Z_3,Z_4,W$. On the other hand, if some $y_0 \in Y_0$ and $z_2 \in Z_2$ are adjacent, then $y_0,z_2,x_2,x_1,y_0$ is a 4-hole in $G$, contrary to the fact that $G$ is $C_4$-free. It follows that $Y_0$ is anticomplete to $Z_2$.~$\blacklozenge$

\begin{adjustwidth}{1cm}{1cm}
\begin{claim} \label{lemma-main-withC7-claim-Zi-attachment} For all $i \in \mathbb{Z}_7$, $Z_i$ is complete to $Z_{i+1},Z_{i+3},Z_{i+4},Z_{i+6},W$. 
\end{claim} 
\end{adjustwidth} 
\noindent 
{\em Proof of Claim~\ref{lemma-main-withC7-claim-Zi-attachment}.}  By symmetry, it suffices to prove the claim for $i = 0$. Note that: 
\begin{itemize} 
\item distinct, nonadjacent vertices $x_1,x_4$ are complete to $Z_0 \cup Z_1 \cup Z_4 \cup W$; 
\item distinct, nonadjacent vertices $x_0,x_3$ are complete to $Z_0 \cup Z_3 \cup Z_6$. 
\end{itemize} 
So, by Proposition~\ref{prop-non-adj-comp-clique}, $Z_0 \cup Z_1 \cup Z_4 \cup W$ and $Z_0 \cup Z_3 \cup Z_6$ are both cliques, and it follows that $Z_0$ is complete to $Z_1,Z_3,Z_4,Z_6,W$.~$\blacklozenge$

\begin{adjustwidth}{1cm}{1cm} 
\begin{claim} \label{lemma-main-withC7-claim-spec-part} 
$(X_0,\dots,X_6;Y_0,\dots,Y_6;Z_0,\dots,Z_6;W)$ is a special partition of $G \setminus A$. 
\end{claim}  
\end{adjustwidth} 
{\em Proof of Claim~\ref{lemma-main-withC7-claim-spec-part}.} Recall  that sets $X_0,\dots,X_6,Y_0,\dots,Y_6,Z_0,\dots,Z_6,W$ are cliques that together partition $V(G) \setminus A$, and that cliques $X_0,\dots,X_6$ are all nonempty.\footnote{Cliques $Y_0,\dots,Y_6,Z_0,\dots,Z_6,W$ may possibly be empty.} The result now follows immediately from the definition of a special partition, and from Claims~\ref{lemma-main-withC7-claim-Xi-hyperhole}-\ref{lemma-main-withC7-claim-Zi-attachment}.~$\blacklozenge$

\begin{adjustwidth}{1cm}{1cm}
\begin{claim} \label{lemma-main-withC7-claim-AX-anticomp} $A$ is anticomplete to $X$, and consequently, $N_G(A) \subseteq Y \cup Z \cup W$. 
\end{claim} 
\end{adjustwidth} 
{\em Proof of Claim~\ref{lemma-main-withC7-claim-AX-anticomp}.} Since $V(G) = X \cup Y \cup Z \cup W \cup A$, it suffices to show that $A$ is anticomplete to $X$. Suppose otherwise. By symmetry, we may assume that some $a \in A$ and $x_0' \in X_0$ are adjacent. But then $a$ has exactly one neighbor in the 7-hole $x_0',x_1,x_2,x_3,x_4,x_5,x_6,x_0'$, contrary to Lemma~\ref{lemma-attach-to-C7}.~$\blacklozenge$

\begin{adjustwidth}{1cm}{1cm} 
\begin{claim} \label{lemma-main-withC7-claim-NGB-clique} 
If $A \neq \emptyset$, and if $B$ is the vertex set of a component of $G[A]$, then $N_G(B)$ is a clique-cutset of $G$. 
\end{claim} 
\end{adjustwidth} 
{\em Proof of Claim~\ref{lemma-main-withC7-claim-NGB-clique}.} Assume that $A \neq \emptyset$, and let $B$ be the vertex set of some component of $G[A]$. By Claim~\ref{lemma-main-withC7-claim-AX-anticomp}, $B$ is anticomplete to $X$, and we have that $N_G(B) \subseteq Y \cup Z \cup W$. In particular, $N_G(B)$ is a cutset of $G$ that separates $B$ from $X$ in $G$. It remains to show that $N_G(B)$ is a clique. Suppose otherwise. It then follows from the definition of a special partition (or alternatively, from Lemma~\ref{lemma-YiZi-in-special-part}) that there exists some $i \in \mathbb{Z}_7$ such that $N_G(B)$ intersects both $Y_i$ and $Z_{i+2}$. By symmetry, we may assume that $N_G(B)$ intersects both $Y_0$ and $Z_2$. Fix vertices $y_0 \in Y_0$ and $z_2 \in Z_2$ that both have a neighbor in $B$. Since $G[B]$ is connected, there exists an induced path $b_0,\dots,b_t$ ($t \geq 0$) in $B$ such that $b_0$ is adjacent to $y_0$, and $b_t$ is adjacent to $z_2$; we may assume that $b_0,\dots,b_t$ is the shortest path with this property, so that $y_0$ is anticomplete to $\{b_1,\dots,b_t\}$, and $z_2$ is anticomplete to $\{b_0,\dots,b_{t-1}\}$. Clearly, $\{b_0,\dots,b_t\}$ is anticomplete to $X$ (because $B$ is anticomplete to $X$). But now if $t$ is odd, then $y_0,b_0,\dots,b_t,z_2,x_2,x_1,y_0$ is an even hole in $G$; and if $t$ is even, then $y_0,b_0,\dots,b_t,z_2,x_4,y_0$ is an even hole in $G$. In either case, $G$ contains an even hole, contrary to the fact that $G$ is $(P_7,C_4,C_6)$-free, and therefore even-hole-free.~$\blacklozenge$

\begin{adjustwidth}{1cm}{1cm} 
\begin{claim} \label{lemma-main-withC7-claim-2P3-free-7-saucer} 
If $G$ is $2P_3$-free, then $G$ is a 7-saucer with an associated 7-saucer partition $(X_0,\dots,X_6;Y_0,\dots,Y_6;Z_0,\dots,Z_6;W;A)$. 
\end{claim} 
\end{adjustwidth} 
{\em Proof of Claim~\ref{lemma-main-withC7-claim-2P3-free-7-saucer}.} Assume that $G$ is $2P_3$-free. First of all, by Claim~\ref{lemma-main-withC7-claim-spec-part}, we know that $(X_0,\dots,X_6;Y_0,\dots,Y_6;Z_0,\dots,Z_6;W)$ is a special partition of $G \setminus A$. So, if $A = \emptyset$, then $G$ is a 7-saucer with an associated 7-saucer partition $(X_0,\dots,X_6;Y_0,\dots,Y_6;Z_0,\dots,Z_6;W;A)$, and we are done. We may therefore assume that $A \neq \emptyset$. By Claim~\ref{lemma-main-withC7-claim-AX-anticomp}, we know that $N_G(A) \subseteq Y \cup Z \cup W$. 

First, we show that for all $i \in \mathbb{Z}_7$, either $A$ is anticomplete to $Y_i$, or $Z_{i+2} = \emptyset$. Suppose otherwise. By symmetry, we may assume that $A$ is not anticomplete to $Y_0$ and that $Z_2 \neq \emptyset$. Fix adjacent vertices $a \in A$ and $y_0 \in Y_0$, and fix any $z_2 \in Z_2$. Since $(X_0,\dots,X_6;Y_0,\dots,Y_6;Z_0,\dots,Z_6;W)$ is a special partition of $G \setminus A$, we know that $y_0$ and $z_2$ are nonadjacent, and so by Claim~\ref{lemma-main-withC7-claim-NGB-clique}, $a$ is nonadjacent to $z_2$.\footnote{Indeed, let $B$ be the vertex set of a component of $G[A]$ such that $a \in B$. By Claim~\ref{lemma-main-withC7-claim-NGB-clique}, $N_G(B)$ is a clique, and so since $y_0,z_2$ are nonadjacent, $N_G(B)$ may contain at most one of them. So, $a$ can be adjacent to at most one of $y_0,z_2$. Since $a$ is in fact adjacent to $y_0$, it follows that $a$ is nonadjacent to $z_2$.} But now $G[a,y_0,x_0,x_2,z_2,x_5]$ is a $2P_3$, a contradiction. 

Next, if $a_0,a_1,a_2$ is an induced two-edge path in $G[A]$, then $G[a_0,a_1,a_2,x_0,x_1,x_2]$ is a $2P_3$, a contradiction. This proves that $G[A]$ is $P_3$-free. Now Lemma~\ref{prop-P3-free} guarantees that $A$ can be partitioned into nonempty cliques, say $A_1,\dots,A_{\ell}$, pairwise anticomplete to each other. 

Finally, fix any $i \in \{1,\dots,\ell\}$. By Claim~\ref{lemma-main-withC7-claim-NGB-clique}, $N_G(A_i)$ is a clique. Now Proposition~\ref{prop-C4Free-CoBip} guarantees that 
$A_i$ can be ordered as $A_i = \{a_1^i,\dots,a_{r_i}^i\}$ so that $N_G[a_{r_i}^i] \subseteq \dots \subseteq N_G[a_1^i]$.\footnote{Indeed, Proposition~\ref{prop-C4Free-CoBip}(b) applied to the cliques $A_i$ and $N_G(A_i)$ guarantees that $A_i$ can be ordered as $A_i = \{a_1^i,\dots,a_{r_i}^i\}$ so that $N_G(a_{r_i}^i) \cap N_G(A_i) \subseteq \dots \subseteq N_G(a_1^i) \cap N_G(A_i)$. Since $A_i$ is a clique, this implies $N_G[a_{r_i}^i] \subseteq \dots \subseteq N_G[a_1^i]$.} 

We have now shown that $G$ is a 7-saucer, and that $(X_0,\dots,X_6;Y_0,\dots,Y_6;Z_0,\dots,Z_6;W;A)$ is an associated 7-saucer partition of $G$.~$\blacklozenge$ 

\medskip 

We are now done: Claims~\ref{lemma-main-withC7-claim-spec-part} and~\ref{lemma-main-withC7-claim-NGB-clique} together guarantee that $G$ admits either a special partition or a clique-cutset, whereas Claim~\ref{lemma-main-withC7-claim-2P3-free-7-saucer} guarantees that if $G$ is $2P_3$-free, then $G$ is a 7-saucer. 
\end{proof}

\begin{theorem} \label{thm-7-saucer} For any graph $G$, the following are equivalent: 
\begin{itemize} 
\item $G$ is $(2P_3,C_4,C_6)$-free and contains an induced $C_7$; 
\item $G$ is a 7-saucer. 
\end{itemize} 
\end{theorem} 
\begin{proof} 
This follows immediately from Proposition~\ref{prop-7-saucer-in-class} and Lemma~\ref{lemma-main-withC7}. 
\end{proof} 

\begin{lemma} \label{lemma-4K1FreeWithC7-no-clique-cut} If a $4K_1$-free graph $G$ contains an induced $C_7$, then $G$ does not admit a clique-cutset. 
\end{lemma} 
\begin{proof} 
Fix a $4K_1$-free graph $G$ that contains a 7-hole (say, $x_0,x_1,x_2,\dots,x_6,x_0$), and suppose that $G$ admits a clique-cutset $C$. Let $(A,B)$ be a partition of $V(G) \setminus C$ into nonempty sets, anticomplete to each other. Since $C_7$ does not admit a clique-cutset, we see that $X := \{x_0,x_1,\dots,x_6\}$ intersects at most one of $A$ and $B$; by symmetry, we may assume that $X \cap B = \emptyset$, and consequently, $X \subseteq A \cup C$. But since $C$ is a clique, we see that $G[X \setminus C]$ contains an induced $P_5$, and consequently, an induced $3K_1$. The three vertices of this $3K_1$ together with any vertex of $B$ induce a $4K_1$ in $G$, a contradiction. 
\end{proof}

\begin{theorem} \label{thm-main-withC7} For any graph $G$, the following are equivalent: 
\begin{enumerate}[(a)] 
\item $G$ is a $(P_7,C_4,C_6)$-free graph that contains an induced $C_7$ and does not admit a clique-cutset; 
\item $G$ is a $(2P_3,C_4,C_6)$-free graph that contains an induced $C_7$ and contains no simplicial vertices; 
\item $G$ is a $(4K_1,C_4,C_6)$-free graph that contains an induced $C_7$; 
\item $G$ admits a special partition; 
\item $G$ has exactly one nonetrivial anticomponent, and this anticomponent is a thickening of a graph in $\mathcal{M}$; 
\item $G$ can be obtained from a thickening of a graph in $\mathcal{M}$ by possibly adding universal vertices to it. 
\end{enumerate} 
\end{theorem} 
\begin{proof} 
By Proposition~\ref{prop-spec-part-thickening-MM}, (d),~(e), and~(f) are equivalent. By Lemma~\ref{lemma-main-withC7}, (a) implies (d). Further, by Lemma~\ref{lemma-spec-part-4K1C4C6FreeWithC7}, (d) implies (c). It remains to show that (c) implies (b), and that (b) implies (a). 

We first assume (c) and prove (b). Since $G$ is $(4K_1,C_4,C_6)$-free, and since $4K_1$ is an induced subgraph of $2P_3$, we know that $G$ is $(2P_3,C_4,C_6)$-free. Moreover, by Lemma~\ref{lemma-4K1FreeWithC7-no-clique-cut}, $G$ does not admit a clique-cutset; since the graph $G$ is not complete (because it contains an induced $C_7$), it follows that $G$ contains no simplicial vertices.\footnote{This is because the neighborhood of any simplicial vertex in a graph that is not complete is a clique-cutset of that graph.} This proves~(b). 
 
To complete the proof, it remains to show that (b) implies (a). Since $2P_3$ is an induced subgraph of $P_7$, it is clear that any $(2P_3,C_4,C_6)$-free graph that contains an induced $C_7$ is, in particular, a $(P_7,C_4,C_6)$-free graph that contains an induced $C_7$. This, together with Proposition~\ref{prop-2P3C4-free-clique-cut-simplicial}, guarantees that (b) implies (a). 
\end{proof}

\section{$\boldsymbol{(2P_3,C_4,C_6,C_7)}$-free graphs that contain an induced $\boldsymbol{T_0}$} \label{sec:withT0}

In this section, we give a full structural description of all $(2P_3,C_4,C_6,C_7)$-free graphs that contain an induced $T_0$ (see Theorem~\ref{thm-T0-tent}). As an easy corollary, we obtain a full structural description of $(2P_3,C_4,C_6,C_7)$-free graphs that contain an induced $T_0$ and contain no simplicial vertices (see Corollary~\ref{cor-T0-tent}). 

We begin with some definitions. First of all, the {\em 3-pentagon}, as well as graphs $T_0$ and $T_1$, are represented in Figure~\ref{fig:3pentagonT0T1}. Note that the 3-pentagon is an induced subgraph of $T_0$, and $T_0$ is an induced subgraph of $T_1$. Next, recall that a graph is {\em chordal} if it contains no holes. A {\em simplicial elimination ordering} of a graph $G$ is an ordering $v_1,\dots,v_t$ of the vertices of $G$ such that for all $i \in \{1,\dots,t\}$, the vertex $v_i$ is simplicial in the graph $G \setminus \{v_1,\dots,v_{i-1}\}$. It is well known that a graph is chordal if and only if it admits a simplicial elimination ordering~\cite{FulkersonGross}. 

\begin{figure}
\begin{center}
\includegraphics[scale=0.5]{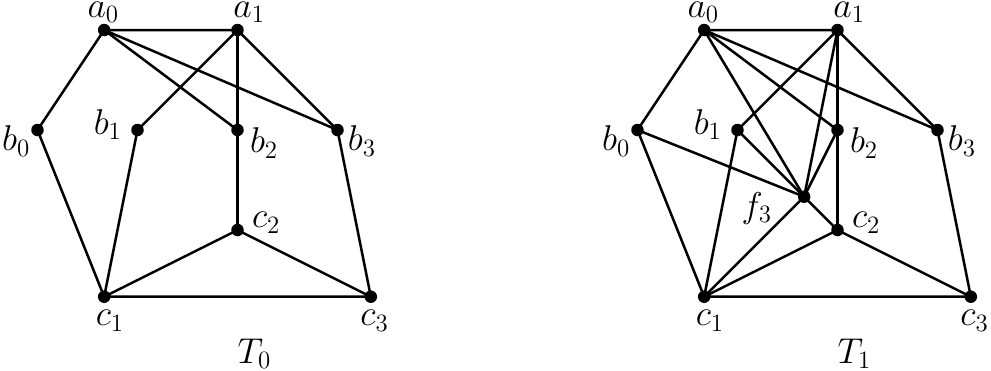}
\end{center} 
\caption{Graphs $T_0$ (left) and $T_1$ (right).} \label{fig:T0T1-labeled} 
\end{figure} 

\begin{proposition} \label{prop-T0T1-2P3C4C6Free} $T_0$ and $T_1$ are both anticonnected, contain no simplicial vertices and no universal vertices, and are $(2P_3,C_4,C_6,C_7)$-free. Consequently, any thickening of $T_0$ or $T_1$ is anticonnected, contains no simplicial and no universal vertices, and is $(2P_3,C_4,C_6,C_7)$-free. 
\end{proposition} 
\begin{proof} 
In view of Proposition~\ref{prop-thickening}, it is enough to prove that $T_0$ and $T_1$ are both anticonnected, contain no simplicial vertices and no universal vertices, and are $(2P_3,C_4,C_6,C_7)$-free. We will use the vertex labeling of $T_0$ and $T_1$ from Figure~\ref{fig:T0T1-labeled}, so that $T_1 \setminus f_3 = T_0$. In both $T_0$ and $T_1$, vertices $b_0,b_1,b_2,b_3$ together form a stable set of size four, and all the other vertices have at least one nonneighbor in this stable set; therefore, $T_0$ and $T_1$ are both anticonnected. The fact that $T_0$ and $T_1$ contain no simplicial vertices and no universal vertices can easily be seen by inspection. It remains to show that $T_0$ and $T_1$ are $(2P_3,C_4,C_6,C_7)$-free. Since $T_0$ is an induced subgraph of $T_1$, it in fact suffices to show that $T_1$ is $(2P_3,C_4,C_6,C_7)$-free. 

\begin{adjustwidth}{1cm}{1cm} 
\begin{claim} \label{prop-T0T1-2P3C4C6Free-claim-holes-in-T1-C5} 
All holes in $T_1$ are of length five. Consequently, $T_1$ is $(C_4,C_6,C_7)$-free. 
\end{claim} 
\end{adjustwidth} 
{\em Proof of Claim~\ref{prop-T0T1-2P3C4C6Free-claim-holes-in-T1-C5}.} Suppose that $x_0,x_1,\dots,x_{k-1},x_0$ (where $k \geq 4$ and the indices are understood to be in $\mathbb{Z}_k$) is a hole in $T_1$, and set $X := \{x_0,x_1,\dots,x_{k-1}\}$. We must show that $k = 5$. 

Suppose first that $f_3 \in X$. By symmetry, we may assume that $x_0 = f_3$. Since $V(T_1) \setminus N_{T_1}[f_3] = \{b_3,c_3\}$, and since the nonneighbors of $x_0 = f_3$ in the hole $x_0,x_1,\dots,x_{k-1},x_0$ are precisely the vertices $x_2,\dots,x_{k-2}$, we deduce that $x_2,\dots,x_{k-2} \in \{b_3,c_3\}$. Consequently, $4 \leq k \leq 5$. If $k = 5$, then we are done. We may therefore assume that $k = 4$ and $x_2 \in \{b_3,c_3\}$. Now, note that $x_0 = f_3$ and $x_2$ have two nonadjacent common neighbors (namely $x_1$ and $x_3$) in the hole $x_0,x_1,x_2,x_3,x_0$. But the two common neighbors of $f_3$ and $b_3$ (namely $a_0$ and $a_1$) are adjacent, as are the two common neighbors of $f_3$ and $c_3$ (namely $c_1$ and $c_2$). Thus, $x_2 \neq b_3$ and $x_2 \neq c_3$, contrary to the fact that $x_2 \in \{b_3,c_3\}$. 

From now on, we may assume that $f_3 \notin X$, which means that our hole $x_0,x_1,\dots,x_{k-1},x_0$ is in fact a hole in $T_1 \setminus f_3 = T_0$. First of all, note that $b_0,b_1,b_2,b_3,c_1,c_2,c_3$ is a simplicial elimination ordering of $T_0 \setminus \{a_0,a_1\}$, and so by~\cite{FulkersonGross}, $T_0 \setminus \{a_0,a_1\}$ is chordal, i.e.\ it contains no holes. So, at last one of $a_0,a_1$ belongs to $X$. 

Suppose first that $a_0,a_1 \in X$. Since $a_0$ and $a_1$ are adjacent, we may assume by symmetry that $x_0 = a_0$ and $x_1 = a_1$. Then $x_{k-1} \in N_{T_0}(a_0) \setminus N_{T_0}(a_1)$ and $x_2 \in N_{T_0}(a_1) \setminus N_{T_0}(a_0)$, and we deduce that $x_{k-1} = b_0$ and $x_2 = b_1$. Then $x_{k-2} \in N_{T_0}(x_{k-1}) \setminus N_{T_0}(x_0) = N_{T_0}(b_0) \setminus N_{T_0}(a_0) = \{c_1\}$ and $x_3 \in N_{T_0}(x_2) \setminus N_{T_0}(x_1) = N_{T_0}(b_1) \setminus N_{T_0}(a_1) = \{c_1\}$. Thus, $x_{k-2} = x_3 = c_3$; we now deduce that $k-2 = 3$, and therefore, $k = 5$. 

From now on, we may assume that exactly one of $a_0,a_1$ belongs to $X$. By symmetry, we may assume that $a_1 \in X$ and $a_0 \notin X$. Consequently, $x_0,x_1,\dots,x_{k-1},x_0$ is a hole in $T_0 \setminus a_0$. Since $d_{T_0 \setminus a_0}(b_0) = 1$, we see that $b_0 \notin X$. Therefore, $x_0,x_1,\dots,x_{k-1},x_0$ is a hole in the 3-pentagon $T_0 \setminus \{a_0,b_0\}$, and it is clear that all holes in a 3-pentagon are of length five. Therefore, $k = 5$, and we are done.~$\blacklozenge$ 

\begin{adjustwidth}{1cm}{1cm} 
\begin{claim} \label{prop-T0T1-2P3C4C6Free-claim-T1-2P3-free} $T_1$ is $2P_3$-free. 
\end{claim} 
\end{adjustwidth} 
{\em Proof of Claim~\ref{prop-T0T1-2P3C4C6Free-claim-T1-2P3-free}.} Since $V(T_1) \setminus N_{T_1}[f_3] = \{b_2,c_3\}$ is a clique, Proposition~\ref{prop-G-x-disjoint-cliques-2P3-free} guarantees that $T_1$ is $2P_3$-free if and only if $T_1 \setminus f_3 = T_0$ is a $2P_3$-free. Therefore, it suffices to show that $T_0$ is $2P_3$-free. Suppose otherwise, and fix pairwise distinct vertices $x_0,x_1,x_2,y_0,y_1,y_2 \in V(T_0)$ such that $x_0,x_1,x_2$ and $y_0,y_1,y_2$ are both induced paths in $T_0$, and such that $\{x_0,x_1,x_2\}$ is anticomplete to $\{y_0,y_1,y_2\}$. 

Suppose first that $a_0,a_1 \notin \{x_0,x_1,x_2,y_0,y_1,y_2\}$. Then vertices $x_0,x_1,x_2,y_0,y_1,y_2$ induce a $2P_3$ in the graph $T_0 \setminus \{a_0,a_1\}$. But note that, in the graph $T_0 \setminus \{a_0,a_1\}$, vertices $b_0,b_1,b_2,b_3$ have degree $1$, whereas vertices $x_1,y_1$ have degree at least two. So, $x_1,y_1 \notin \{b_0,b_1,b_2,b_3\}$, and therefore, $x_1,y_1 \in \{c_1,c_2,c_3\}$. But this is impossible since $x_1,y_1$ are nonadjacent, whereas $\{c_1,c_1,c_3\}$ is a clique. 

Suppose next that $a_0,a_1 \in \{x_0,x_1,x_2,y_0,y_1,y_2\}$. Since $a_0,a_1$ are adjacent, we may assume by symmetry that $x_0 = a_0$ and $x_1 = a_1$. Then $x_2 \in N_{T_0}(x_1) \setminus N_{T_0}[x_0] = N_{T_0}(a_1) \setminus N_{T_0}[a_0] = \{b_1\}$. So, $x_2 = b_1$. But now $y_0,y_1,y_2 \in V(T_0) \setminus N_{T_0}[\{x_0,x_1,x_2\}] = V(T_0) \setminus N_{T_0}[\{a_0,a_1,b_1\}] = \{c_2,c_3\}$. Since $y_0,y_1,y_2$ are pairwise distinct, this is clearly impossible. 

Thus, $\{x_0,x_1,x_2,y_0,y_1,y_2\}$ contains exactly one of $a_0,a_1$. By symmetry, we may assume that $a_1 \in \{x_0,x_1,x_2,y_0,y_1,y_2\}$ and $a_0 \notin \{x_0,x_1,x_2,y_0,y_1,y_2\}$, so that vertices $x_0,x_1,x_2,y_0,y_1,y_2$ together induce a $2P_3$ in the graph $T_0 \setminus a_0$. 

Suppose first that $b_0 \in \{x_0,x_1,x_2,y_0,y_1,y_2\}$. Since $d_{T_0 \setminus a_0}(b_0) = 1$, we may assume by symmetry that $b_0 = x_0$. But then $x_1 = c_1$ and $x_2 \in \{b_1,c_2,c_3\}$.\footnote{The fact that $x_1 = c_1$ follows from the fact that $c_1$ is the only neighbor of $x_1 = b_0$ in $T_0 \setminus a_0$. On the other hand, the fact that $x_2 \in \{b_1,c_2,c_3\}$ follows from the fact that $N_{T_0 \setminus a_0}(c_1) = \{b_0,b_1,c_2,c_2\}$ and the fact that $b_0 = x_0$.} In any case, we have that $a_1 \in \{y_0,y_1,y_2\}$, and consequently, $a_1$ is nonadjacent to $x_2$. Therefore, $x_2 \neq b_1$, and it follows that $x_2 \in \{c_2,c_3\}$. By symmetry, we may assume that $x_2 = c_2$. But now $y_0,y_2,y_3$ is an induced two-edge-path in the graph $T_0 \setminus N_{T_0}[\{x_0,x_1,x_2\}] = T_0 \setminus N_{T_0}[\{b_0,c_1,c_2\}] = T_0[a_1,b_3]$, which is impossible, since this graph contains only two vertices. 

Thus, $b_0 \notin \{x_0,x_1,x_2,y_0,y_1,y_2\}$. It follows that $T_0 \setminus \{a_0,b_0\}$ contains an induced $2P_3$. But note that $T_0 \setminus \{a_0,b_0\}$ is a 3-pentagon, and in particular, it contains exactly seven vertices. Therefore, by deleting one suitably chosen vertex of the 3-pentagon $T_0 \setminus \{a_0,b_0\}$, we should obtain a $2P_3$. But it is easy to see by inspection that no such vertex exists. This proves that $T_0$ is, in fact, $2P_3$-free, and we are done.~$\blacklozenge$ 

\medskip 

By Claims~\ref{prop-T0T1-2P3C4C6Free-claim-holes-in-T1-C5} and~\ref{prop-T0T1-2P3C4C6Free-claim-T1-2P3-free}, we have that $T_1$ is $(2P_3,C_4,C_6,C_7)$-free, and we are done. 
\end{proof}

\begin{lemma} \label{lemma-T0-attach} Let $G$ be a $(2P_3,C_4,C_6)$-free graph, and let $T$ an induced subgraph of $G$ that is isomorphic to $T_0$, with the vertices of $T$ labeled as in Figure~\ref{fig:T0T1-labeled} (left). Then for all $x \in V(G)$, one of the following holds: 
\begin{enumerate}[(a)] 
\item \label{ref-T0-attach-clone} there exists some $v \in V(T)$ such that $N_G[x] \cap V(T) = N_T[v]$; 
\item \label{ref-T0-attach-not-bci} there exists some $i \in \{2,3\}$ such that $N_G[x] \cap V(T) = V(T) \setminus \{b_i,c_i\}$; 
\item \label{ref-T0-attach-c23} $N_G[x] \cap V(T) = \{c_2,c_3\}$; 
\item \label{ref-T0-attach-anticenter} $N_G[x] \cap V(T) = \emptyset$; 
\item \label{ref-T0-attach-center} $N_G[x] \cap V(T) = V(T)$. 
\end{enumerate} 
\end{lemma} 
\begin{proof} 
Fix an arbitrary $x \in V(G)$. If $x \in V(T)$, then clearly, $N_G[x] \cap V(T) = N_T[x]$, and in particular,~(\ref{ref-T0-attach-clone}) holds. We may therefore assume that $x \notin V(T)$. We now prove a sequence of claims. 

\begin{adjustwidth}{1cm}{1cm}
\begin{claim} \label{claim-lemma-T0-attach-claim-two-b0b1b2b3}
If $x$ has a neighbor both in $\{b_0,b_1\}$ and in $\{b_2,b_3\}$, then $x$ is complete to $\{a_0,a_1\}$ and has at least three neighbors in $\{b_0,b_1,b_2,b_3\}$. 
\end{claim} 
\end{adjustwidth} 
\noindent 
{\em Proof of Claim~\ref{claim-lemma-T0-attach-claim-two-b0b1b2b3}.} Assume that $x$ has a neighbor both in $\{b_0,b_1\}$ and in $\{b_2,b_3\}$; by symmetry, we may assume that $x$ is adjacent to $b_1$ and $b_2$. Then $x$ is adjacent to $a_1$, for otherwise, $x,b_1,a_1,b_2,x$ would be a 4-hole in $G$, a contradiction. It now remains to show that $x$ has a neighbor in $\{b_0,b_3\}$ and is adjacent to $a_0$.  

Suppose that $x$ is anticomplete to $\{b_0,b_3\}$. Then $x$ is nonadjacent to $c_3$, for otherwise, $x,c_3,b_3,a_1,x$ would be a 4-hole in $G$, a contradiction. Then $x$ has a neighbor in $\{c_1,c_2\}$, for otherwise, $G[x,a_1,b_3,b_0,c_1,c_2]$ would be a $2P_3$, a contradiction. If $x$ is adjacent to $c_1$, but nonadjacent to $c_2$, then $x,c_1,c_2,b_2,x$ is a 4-hole in $G$, a contradiction. Similarly, if $x$ is adjacent to $c_2$, but nonadjacent to $c_1$, then $x,c_2,c_1,b_1,x$ is a 4-hole in $G$, a contradiction. So, $x$ is complete to $\{c_1,c_2\}$. Then $x$ is nonadjacent to $a_0$, for otherwise, $x,a_0,b_0,c_1,x$ would be a 4-hole in $G$, a contradiction. But now $x,c_1,c_3,b_3,a_0,b_2,x$ is a 6-hole in $G$, a contradiction. 

We have now shown that $x$ has a neighbor in $\{b_0,b_3\}$, and it remains to show that $x$ is adjacent to $a_0$. Fix $i \in \{0,3\}$ such that $x$ is adjacent to $b_i$. Then $x$ is adjacent to $a_0$, for otherwise, $x,b_i,a_0,b_2,x$ would be a 4-hole in $G$, a contradiction.~$\blacklozenge$ 

\begin{adjustwidth}{1cm}{1cm}
\begin{claim} \label{claim-lemma-T0-attach-claim-neighbor-in-b0b1-b2b3} If $x$ has a neighbor in both $\{b_0,b_1\}$ and $\{b_2,b_3\}$, then one of (\ref{ref-T0-attach-clone}), (\ref{ref-T0-attach-not-bci}), (\ref{ref-T0-attach-center}) holds. 
\end{claim} 
\end{adjustwidth} 
\noindent 
{\em Proof of Claim~\ref{claim-lemma-T0-attach-claim-neighbor-in-b0b1-b2b3}.} By symmetry, we may assume that $x$ is adjacent to $b_1$ and $b_2$. By Claim~\ref{claim-lemma-T0-attach-claim-two-b0b1b2b3}, $x$ is complete to $\{a_0,a_1\}$ and is adjacent to at least one of $b_0,b_3$. 

Suppose first that $x$ is adjacent to $b_0$. Then $x$ is adjacent to $c_1$, for otherwise, $x,b_0,c_1,b_1,x$ would be a 4-hole, a contradiction. Further, $x$ is adjacent to $c_2$, for otherwise, $x,c_1,c_2,b_2,x$ would be a 4-hole, a contradiction. If $x$ is complete to $\{b_3,c_3\}$, then (\ref{ref-T0-attach-center}) holds, and if $x$ is anticomplete to $\{b_3,c_3\}$, then (\ref{ref-T0-attach-not-bci}) holds (with $i = 3$). So, we may assume that $x$ is mixed on $\{b_3,c_3\}$. But if $x$ is adjacent to $b_3$ and nonadjacent to $c_3$, then $x,c_2,c_3,b_3,x$ is a 4-hole in $G$, a contradiction; and if $x$ is adjacent to $c_3$ and nonadjacent to $b_3$, then $x,a_1,b_3,c_3,x$ is a 4-hole in $G$, a contradiction. 

From now on, we assume that $x$ is nonadjacent to $b_0$, and is consequently adjacent to $b_3$. Then $x$ is nonadajcent to $c_1$, for otherwise, $x,c_1,b_0,a_0,x$ would be a 4-hole in $G$, a contradiction. But then $x$ is in fact anticomplete to $\{c_1,c_2,c_3\}$, for otherwise, we fix an index $i \in \{2,3\}$ such that $x$ is adjacent to $c_i$, and we observe that $x,c_i,c_1,b_1,x$ is a 4-hole in $G$. We now see that $N_G(x) \cap (T) = N_G[a_1]$, and in particular, (\ref{ref-T0-attach-clone}) holds.~$\blacklozenge$ 

\begin{adjustwidth}{1cm}{1cm}
\begin{claim} \label{claim-lemma-T0-attach-claim-two-in-b0b1b2b3} If $x$ has at least two neighbors in $\{b_0,b_1,b_2,b_3\}$, then one of (\ref{ref-T0-attach-clone}), (\ref{ref-T0-attach-not-bci}), (\ref{ref-T0-attach-center}) holds. 
\end{claim} 
\end{adjustwidth} 
\noindent 
{\em Proof of Claim~\ref{claim-lemma-T0-attach-claim-two-in-b0b1b2b3}.} In view of Claim~\ref{claim-lemma-T0-attach-claim-neighbor-in-b0b1-b2b3}, we may assume that $x$ is complete to one of $\{b_0,b_1\}$ and $\{b_2,b_3\}$ and is anticomplete to the other. 

Suppose first that $x$ is complete to $\{b_2,b_3\}$ and anticomplete to $\{b_0,b_1\}$. Then $x$ is adjacent to $c_1$, for otherwise, $G[b_2,x,b_3,b_0,c_1,b_1]$ would be a $2P_3$, a contradiction. But then if $x$ is adjacent to $a_0$, then $x,a_0,b_0,c_1,x$ is a 4-hole in $G$, and if $x$ is nonadjacent to $a_0$, then $x,b_2,a_0,b_3,x$ is a 4-hole in $G$, a contradiction in either case. 

Thus, $x$ is complete to $\{b_0,b_1\}$ and anticomplete to $\{b_2,b_3\}$. Then $x$ is adjacent to $c_1$, for otherwise, $x,b_0,c_1,b_1,x$ would be a 4-hole in $G$, a contradiction. Further, $x$ is adjacent to at least one of $c_2,c_3$, for otherwise, $G[b_0,x,b_1,b_2,c_2,c_3]$ would be a $2P_3$, a contradiction. By symmetry, we may assume that $x$ is adjacent to $c_2$. Then $x$ is anticomplete to $\{a_0,a_1\}$, for otherwise, we fix an index $i \in \{0,1\}$ such that $x$ is adjacent to $a_i$, and we observe that $x,a_i,b_2,c_2,x$ is a 4-hole in $G$, a contradiction. Now $x$ is adjacent to $c_3$, for otherwise, $x,c_2,c_3,b_3,a_0,b_0,x$ would be a 6-hole in $G$, a contradiction. But now $N_G(x) \cap V(P) = N_T[c_1]$, and in particular, (\ref{ref-T0-attach-clone}) holds.~$\blacklozenge$ 

\begin{adjustwidth}{1cm}{1cm}
\begin{claim} \label{claim-lemma-T0-attach-claim-one-in-b0b1b2b3}
If $x$ has a unique neighbor in $\{b_0,b_1,b_2,b_3\}$, then~(\ref{ref-T0-attach-clone}) holds. 
\end{claim} 
\end{adjustwidth} 
\noindent 
{\em Proof of Claim~\ref{claim-lemma-T0-attach-claim-one-in-b0b1b2b3}.} Suppose first that $x$ is adjacent to exactly one of $b_0,b_1$ and is anticomplete to $\{b_2,b_3\}$. By symmetry, we may assume that $x$ is adjacent to $b_0$ and is nonadjacent to $b_1$. 

Suppose that $x$ is nonadjacent to $c_1$. Then $x$ is in fact anticomplete to $\{c_1,c_2,c_3\}$, for otherwise, we fix an index $i \in \{2,3\}$ such that $x$ is adjacent to $c_i$, and we observe that $x,c_i,c_1,b_0,x$ is a 4-hole in $G$, a contradiction. Then $x$ is nonadjacent to $a_1$, for otherwise, $x,a_1,b_2,c_2,c_1,b_0,x$ would be a 6-hole in $G$, a contradiction. But now $G[x,b_0,c_1,b_2,a_1,b_3]$ is a $2P_3$, a contradiction. This proves that $x$ is in fact adjacent to $c_1$. 

We now see that $x$ is nonadjacent to $a_1$, for otherwise, $x,a_1,b_1,c_1,x$ would be a 4-hole in $G$, a contradiction. Moreover, $x$ is adjacent to $a_0$, for otherwise, $G[x,c_1,b_1,b_2,a_0,b_3]$ would be a $2P_3$, a contradiction. Finally, $x$ is anticomplete to $\{c_2,c_3\}$, for otherwise, we fix an index $i \in \{2,3\}$ such that $x$ is adjacent to $c_i$, and we observe that $x,c_i,b_i,a_0,x$ is a 4-hole in $G$, a contradiction. But now $N_G(x) \cap V(T) = N_T[b_0]$, and in particular, (\ref{ref-T0-attach-clone}) holds. 

From now on, we may assume that $x$ is anticomplete to $\{b_0,b_1\}$, and consequently has a unique neighbor in $\{b_2,b_3\}$. By symmetry, we may assume that $x$ is adjacent to $b_3$ and is nonadjacent to $b_2$. 

Suppose first that $x$ is adjacent to $c_1$. Then $x$ is anticomplete to $\{a_0,a_1\}$, for otherwise, we fix an index $i \in \{0,1\}$ such that $x$ is adjacent to $a_i$, and we observe that $x,a_i,b_i,c_1,x$ is a 4-hole in $G$, a contradiction. Further, $x$ is adjacent to $c_3$, for otherwise, $x,b_3,c_3,c_1,x$ would be a 4-hole in $G$, a contradiction. Finally, $x$ is adjacent to $c_2$, for otherwise, $x,c_1,c_2,b_2,a_1,b_3,x$ would be a 6-hole in $G$, a contradiction. We now have that $N_G(x) \cap V(T) = N_T[c_3]$, and in particular, (\ref{ref-T0-attach-clone}) holds. 

From now on, we may assume that $x$ is nonadjacent to $c_1$. Suppose first that $x$ is adjacent to $c_2$. Then $x$ is nonadjacent to $a_1$, for otherwise, $x,a_1,b_2,c_2,x$ would be a 4-hole in $G$, a contradiction. But now $x,b_3,a_1,b_1,c_1,c_2,x$ is a 6-hole in $G$, a contradiction. This proves that $x$ is nonadajcent to $c_2$. Then $x$ is adjacent to $a_0$, for otherwise, $G[x,b_3,a_0,b_1,c_1,c_2]$ would be a $2P_3$, a contradiction. Similarly, $x$ is adjacent to $a_1$, for otherwise, $G[x,b_3,a_1,b_0,c_1,c_2]$ would be a $2P_3$, a contradiction. Finally, $x$ is adjacent to $c_3$, for otherwise, $G[x,a_1,b_2,b_0,c_1,c_3]$ would be a $2P_3$, a contradiction. But now $N_G(x) \cap V(T) = N_T[b_3]$, and in particular, (\ref{ref-T0-attach-clone}) holds.~$\blacklozenge$ 

\begin{adjustwidth}{1cm}{1cm} 
\begin{claim} \label{claim-lemma-T0-attach-claim-anticomp-b0b1b2b3} If $x$ is anticomplete to $\{b_0,b_1,b_2,b_3\}$, then (\ref{ref-T0-attach-c23}) or~(\ref{ref-T0-attach-anticenter}) holds. 
\end{claim} 
\end{adjustwidth} 
\noindent 
{\em Proof of Claim~\ref{claim-lemma-T0-attach-claim-anticomp-b0b1b2b3}.} Assume that $x$ is anticomplete to $\{b_0,b_1,b_2,b_3\}$. We first show that $x$ is anticomplete to $\{a_0,a_1\}$. Suppose otherwise. By symmetry, we may assume that $x$ is adjacent to $a_1$. Then $x$ is anticomplete to $\{c_1,c_2,c_3\}$, for otherwise, we fix an index $i \in \{1,2,3\}$ such that $x$ is adjacent to $c_i$, and we observe that $x,a_1,b_i,c_i,x$ is a 4-hole in $G$, a contradiction. But now $G[x,a_1,b_3,b_0,c_1,c_2]$ is a $2P_3$, a contradiction. This proves that $x$ is indeed anticomplete to $\{a_0,a_1\}$. 

We now have that $N_G(x) \cap V(T) \subseteq \{c_1,c_2,c_3\}$. Then $x$ is nonadjacent to $c_1$, for otherwise, $G[x,c_1,b_0,b_2,a_1,b_3]$ would be a $2P_3$, a contradiction. It follows that $N_G(x) \cap V(T) \subseteq \{c_2,c_3\}$, and it remains to show that $x$ is either complete or anticomplete to $\{c_2,c_3\}$. Suppose otherwise. By symmetry, we may assume that $x$ is adjacent to $c_2$ and nonadjacent to $c_3$. But then $G[x,c_2,c_3,a_0,a_1,b_1]$ is a $2P_3$, a contradiction.~$\blacklozenge$

\medskip 

The result now follows immediately from Claims~\ref{claim-lemma-T0-attach-claim-two-in-b0b1b2b3},~\ref{claim-lemma-T0-attach-claim-one-in-b0b1b2b3}, and~\ref{claim-lemma-T0-attach-claim-anticomp-b0b1b2b3}. 
\end{proof}

A {\em tent} is a graph $G$ whose vertex set can be partitioned into (possibly empty) sets $A_0,A_1,$ $B_0,B_1,B_2,B_3,C_1,C_2,C_3,F_2,F_3,W,Y,Z$ such that all the following hold: 
\begin{itemize} 
\item $A_0,A_1,B_0,B_1,B_2,B_3,C_1,C_2,C_3$ are nonempty cliques; 
\item $F_2,F_3,Y$ are cliques, and at most one of them is nonempty; 
\item $W$ is a (possibly empty) clique; 
\item $A_0$ is complete to $A_1$; 
\item $A_0$ is complete to $B_0,B_2,B_3$ and anticomplete to $B_1,C_1,C_2,C_3$; 
\item $A_1$ is complete to $B_1,B_2,B_3$ and anticomplete to $B_0,C_1,C_2,C_3$; 
\item $B_0,B_1,B_2,B_3$ are pairwise anticomplete to each other; 
\item $C_1,C_2,C_3$ are pairwise complete to each other; 
\item $C_1$ is complete to $B_0,B_1$ and anticomplete to $B_2,B_3$; 
\item $C_2$ is complete to $B_2$ and anticomplete to $B_0,B_1,B_3$; 
\item $C_3$ is complete to $B_3$ and anticomplete to $B_0,B_1,B_2$; 
\item $F_2$ is complete to $A_0,A_1,B_0,B_1,B_3,C_1,C_3$ and anticomplete to $B_2,C_2$; 
\item $F_3$ is complete to $A_0,A_1,B_0,B_1,B_2,C_1,C_2$ and anticomplete to $B_3,C_3$; 
\item $W$ is complete to $A_0,A_1,B_0,B_1,B_2,B_3,C_1,C_2,C_3,F_2,F_3$; 
\item $Y$ is complete to $C_2,C_3$ and anticomplete to $A_0,A_1,B_0,B_1,B_2,B_3,C_1$; 
\item $Z$ is anticomplete to $A_0,A_1,B_0,B_1,B_2,B_3,C_1,C_2,C_3,Y$;\footnote{Note that this implies that $N_G(Z) \subseteq F_2 \cup F_3 \cup W$.} 
\item $Y$ can be ordered as $Y = \{y_1,\dots,y_t\}$ so that $N_G[y_t] \subseteq \dots \subseteq N_G[y_1]$;\footnote{It is possible that $Y = \emptyset$; in this case, we simply have that $t = 0$. On the other hand, our definition implies that if $Y \neq \emptyset$, then $F_2 = F_3 = \emptyset$. Moreover, our definition implies that $Y \cup C_2 \cup C_3 \subseteq N_G[y] \subseteq Y \cup C_2 \cup C_3 \cup W$ for all $y \in Y$, and that $Y$ and $C_2 \cup C_3 \cup W$ are both cliques.} 
\item if $Z \neq \emptyset$, then $Z$ can be partitioned into nonempty cliques $Z_1,\dots,Z_{\ell}$, pairwise anticomplete to each other, and moreover, for each $i \in \{1,\dots,\ell\}$, $Z_i$ can be ordered as $Z_i = \{z_1^i,\dots,z_{t_i}^i\}$ so that $N_G[z_{t_i}^i] \subseteq \dots \subseteq N_G[z_1^i]$. 
\end{itemize} 
Under these conditions, we also say that $(A_0,A_1;B_0,B_1,B_2,B_3;C_1,C_2,C_3;F_2,F_3;W;Y;Z)$ is a {\em tent partition} of the tent $G$.

Note that any thickening of $T_0$ or $T_1$ is a tent.\footnote{This follows immediately form the appropriate definitions, as long as we keep in mind the vertex labelings of $T_0$ and $T_1$ from Figure~\ref{fig:T0T1-labeled}.}

\begin{proposition} \label{prop-tent-simplicial-universal-T0T1} Let $G$ be a tent, and let $(A_0,A_1;B_0,B_1,B_2,B_3;C_1,C_2,C_3;F_2,F_3;W;Y;Z)$ be a tent partition of $G$. Then all the following hold: 
\begin{enumerate}[(a)] 
\item \label{ref-tent-Y-simplicial} if $Y \neq \emptyset$, then some vertex of $Y$ is simplicial in $G$; 
\item \label{ref-tent-Z-simplicial} if $Z \neq \emptyset$, then some vertex of $Z$ is simplicial in $G$; 
\item \label{ref-tent-W-universal} every vertex of $W$ is a universal vertex of $G \setminus (Y \cup Z)$; 
\item \label{ref-no-WYZ-thickening-T0T1} $G \setminus (W \cup Y \cup Z)$ is a thickening of $T_0$ or $T_1$. 
\end{enumerate} 
\end{proposition} 
\begin{proof} 
We first prove~(\ref{ref-tent-Y-simplicial}). Suppose that $Y \neq \emptyset$, and fix an ordering $Y = \{y_1,\dots,y_t\}$ of $Y$ such that $N_G[y_t] \subseteq \dots \subseteq N_G[y_1]$, as in the definition of a tent. Our goal is to show that $N_G[y_t]$ is a clique, which will immediately imply that $y_t$ is a simplicial vertex. Since $Y \neq \emptyset$, the definition of a tent guarantees that $F_2 = F_3 = \emptyset$, and it further guarantees that $Y$ is a clique, complete to $C_2 \cup C_3$ and anticomplete to $A_0 \cup A_1 \cup B_0 \cup B_1 \cup B_2 \cup B_3 \cup C_1$. Therefore, $Y \cup C_2 \cup C_3 \subseteq N_G[y_t] \subseteq \dots \subseteq N_G[y_1] \subseteq Y \cup C_2 \cup C_3 \cup W$. Set $W_t := N_G(y_t) \cap W$, so that $N_G[y_t] = Y \cup C_2 \cup C_3 \cup W_t$. By the definition of a tent, $Y \cup C_2 \cup C_3$ is a clique, as is $W_t$ (because it is a subset of the clique $W$). So, it suffices to show that $Y \cup C_2 \cup C_3$ is complete to $W_t$. By the definition of a tent, $C_2 \cup C_3$ is complete to $W$ (and therefore to $W_t$ as well). On the other hand, since $Y = \{y_1,\dots,y_t\}$ and $W_t \subseteq N_G[y_t] \subseteq \dots N_G[y_1]$, we see that $Y$ is complete to $W_t$. This proves that $Y \cup C_2 \cup C_3$ is indeed complete to $W_t$, and~(\ref{ref-tent-Y-simplicial}) follows. 

Next, we prove~(\ref{ref-tent-Z-simplicial}). Suppose that $Z \neq \emptyset$, and let $(Z_1,\dots,Z_{\ell})$ be a partition of $Z$ into cliques, pairwise anticomplete to each other, as per the definition of a tent. Using the definition of a tent, we fix an ordering $Z_1 = \{z_1,\dots,z_t\}$ of $Z_1$ such that $N_G[z_t] \subseteq \dots \subseteq N_G[z_1]$. Our goal is to show that $N_G[z_t]$ is a clique, which will immediately imply that $z_t$ is simplicial. By the definition of a tent, we have that $N_G(Z_1) \subseteq F_2 \cup F_3 \cup W$. Since $Z_1$ is a clique, we conclude that $Z_1 \subseteq N_G[z_t] \subseteq \dots \subseteq N_G[z_1] \subseteq Z_1 \cup F_2 \cup F_3 \cup W$. Set $W_t := N_G(z_t) \cap (F_2 \cup F_3 \cup W)$, so that $N_G[z_t] = Z_1 \cup W_t$. By the definition of a tent, $F_2 \cup F_3 \cup W$ is a clique,\footnote{Indeed, by the definition of a tent, $F_2,F_3,W$ are all cliques, at most one of $F_2,F_3$ is nonempty, and $W$ is complete to $F_2 \cup F_3$. Thus, $F_2 \cup F_3 \cup W$ is a clique.} and consequently, $W_t$ is a clique. Finally, since $Z_1 = \{z_1,\dots,z_t\}$ and $W_t \subseteq N_G[z_t] \subseteq \dots \subseteq N_G[z_1]$, we see that $Z_1$ is complete to $W_t$. This proves that $N_G[z_t] = Z_1 \cup W_t$ is a clique, and we deduce that $z_t$ is a simplicial vertex. This completes the proof of~(\ref{ref-tent-Z-simplicial}). 

Further,~(\ref{ref-tent-W-universal}) readily follows from the fact that, by the definition of a tent, $W$ is a clique, complete to $A_0 \cup A_1 \cup B_0 \cup B_1 \cup B_2 \cup B_3 \cup C_1 \cup C_2 \cup C_3 \cup F_2 \cup F_3 = V(Q) \setminus (W \cup Y \cup Z)$. 

It remains to prove~(\ref{ref-no-WYZ-thickening-T0T1}). Set $H := G \setminus (W \cup Y \cup Z)$, so that $V(H) = A_0 \cup A_1 \cup B_0 \cup B_1 \cup B_2 \cup B_3 \cup C_1 \cup C_2 \cup C_3 \cup F_2 \cup F_3$. If $F_2 \cup F_3 = \emptyset$, then it is clear that $H$ is a thickening of $T_0$ (this follows from the vertex labeling of $T_0$ from Figure~\ref{fig:T0T1-labeled} and from the definition of a tent). So, assume that $F_2 \cup F_3 \neq \emptyset$. By the definition of a tent, we see that exactly one of $F_2,F_3$ is nonempty. By symmetry, we may assume that $F_3 \neq \emptyset$ and $F_2 = \emptyset$ (otherwise, we simply swap the roles of $B_2,C_2,F_2$ and $B_3,C_3,F_3$, respectively). But then $H$ is a thickening of $T_1$ (this follows from the vertex labeling of $T_1$ from Figure~\ref{fig:T0T1-labeled} and from the definition of a tent). This proves~(\ref{ref-no-WYZ-thickening-T0T1}). 
\end{proof} 

\begin{proposition} \label{prop-tent-2P3C4C6C7-free-with-ind-T0} Every tent is $(2P_3,C_4,C_6,C_7)$-free and contains an induced $T_0$. 
\end{proposition} 
\begin{proof} 
Let $G$ be a tent, and let $(A_0,A_1;B_0,B_1,B_2,B_3;C_1,C_2,C_3;F_2,F_3;W;Y;Z)$ be a tent partition of $G$. By Proposition~\ref{prop-tent-simplicial-universal-T0T1}(\ref{ref-no-WYZ-thickening-T0T1}), $G \setminus (W \cup Y \cup Z)$ is a thickening of $T_0$ or $T_1$, and in particular, it contains an induced $T_0$ or $T_1$. Since $T_0$ is an induced subgraph of $T_1$, we now deduce that $G$ contains an induced $T_0$. It remains to show that $G$ is $(2P_3,C_4,C_6,C_7)$-free. 

\begin{adjustwidth}{1cm}{1cm}  
\begin{claim} \label{prop-tent-2P3C4C6C7-free-with-ind-T0-claim-YZ-hole}
No vertex of $Y \cup Z$ belongs to any hole of $G$. 
\end{claim} 
\end{adjustwidth} 
{\em Proof of Claim~\ref{prop-tent-2P3C4C6C7-free-with-ind-T0-claim-YZ-hole}.} Suppose otherwise, and let $X$ be the vertex set of a hole in $G$ such that $X \cap (Y \cup Z) \neq \emptyset$. Set $Y_X := Y \cap X$ and $Z_X := Z \cap X$; then $\emptyset \neq Y_X \cup Z_X \subseteq X$. Next, we set $G_X := G\setminus \big((Y \setminus Y_X) \cup (Z \setminus Z_X)\big)$, so that $G_X$ is a tent with an associated tent partition $(A_0,A_1;B_0,B_1,B_2,B_3;C_1,C_2,C_3;F_2;F_3;W;Y_X;Z_X)$, and $X$ induces a hole in $G_X$. Since $Y_X \cup Z_X \neq \emptyset$, Proposition~\ref{prop-tent-simplicial-universal-T0T1}(\ref{ref-tent-Y-simplicial}-\ref{ref-tent-Z-simplicial}) guarantees that some vertex of $Y_X \cup Z_X$ is simplicial in $G_X$. Since $Y_X \cup Z_X \subseteq X$, we deduce that some vertex of $X$ is simplicial in $G$, which is impossible since no vertex of a hole is simplicial.~$\blacklozenge$

\begin{adjustwidth}{1cm}{1cm}  
\begin{claim} \label{prop-tent-2P3C4C6C7-free-with-ind-T0-claim-W-hole}
No vertex of $W$ belongs to any hole of $G$. 
\end{claim} 
\end{adjustwidth} 
{\em Proof of Claim~\ref{prop-tent-2P3C4C6C7-free-with-ind-T0-claim-W-hole}.} By Claim~\ref{prop-tent-2P3C4C6C7-free-with-ind-T0-claim-YZ-hole}, no vertex of $Y \cup Z$ belongs to any hole of $G$. Therefore, it suffices to show that no hole of $G \setminus (Y \cup Z)$ belongs to $W$. But by Proposition~\ref{prop-tent-simplicial-universal-T0T1}(\ref{ref-tent-W-universal}), every vertex of $W$ is a universal vertex of $G \setminus (Y \cup Z)$. Since holes contain no universal vertices, we deduce that no hole of $G \setminus (Y \cup Z)$ contains a vertex of $W$.~$\blacklozenge$ 

\begin{adjustwidth}{1cm}{1cm}  
\begin{claim} \label{prop-tent-2P3C4C6C7-free-with-ind-T0-claim-C4C6C7-free}
$G$ is $(C_4,C_6,C_7)$-free. 
\end{claim} 
\end{adjustwidth} 
{\em Proof of Claim~\ref{prop-tent-2P3C4C6C7-free-with-ind-T0-claim-C4C6C7-free}.} In view of Claims~\ref{prop-tent-2P3C4C6C7-free-with-ind-T0-claim-YZ-hole} and~\ref{prop-tent-2P3C4C6C7-free-with-ind-T0-claim-W-hole}, it suffices to show that the graph $H := G \setminus (W \cup Y \cup Z)$ is $(C_4,C_6,C_7)$-free. By Proposition~\ref{prop-tent-simplicial-universal-T0T1}(\ref{ref-no-WYZ-thickening-T0T1}), $H$ is a thickening of $T_0$ or $T_1$, and so by Proposition~\ref{prop-T0T1-2P3C4C6Free}, $H$ is $(C_4,C_6,C_7)$-free.~$\blacklozenge$ 

\medskip 

In view of Claim~\ref{prop-tent-2P3C4C6C7-free-with-ind-T0-claim-C4C6C7-free}, it now remains to show that $G$ is $2P_3$-free. Suppose otherwise, and fix pairwise distinct vertices $x_0,x_1,x_2,y_0,y_1,y_2 \in V(G)$ such that $x_0,x_1,x_2$ and $y_0,y_1,y_2$ are induced two-edge paths in $G$, and such that $\{x_0,x_1,x_2\}$ and $\{y_0,y_1,y_2\}$ are pairwise anticomplete to each other in $G$.

\begin{adjustwidth}{1cm}{1cm}  
\begin{claim} \label{prop-tent-2P3C4C6C7-free-with-ind-T0-claim-W-no-2P3} $\{x_0,x_1,x_2,y_0,y_1,y_2\} \cap W = \emptyset$. 
\end{claim} 
\end{adjustwidth} 
{\em Proof of Claim~\ref{prop-tent-2P3C4C6C7-free-with-ind-T0-claim-W-no-2P3}.} Suppose otherwise. By symmetry, we may assume that there exists some $i \in \{0,1,2\}$ such that $x_i \in W$. Proposition~\ref{prop-tent-simplicial-universal-T0T1}(\ref{ref-tent-W-universal}) then guarantees that $x_i$ is a universal vertex of $G \setminus (Y \cup Z)$. Since $x_i$ is anticomplete to $\{y_0,y_1,y_2\}$, it follows that $y_0,y_1,y_2 \in Y \cup Z$. So, $y_0,y_1,y_2$ is an induced two-edge path in $G[Y \cup Z]$. But this is impossible since, by the definition of a tent, $Y \cup Z$ can be partitioned into cliques, pairwise anticomplete to each other.\footnote{Indeed, by the definition of a tent, $Y$ is a clique, anticomplete to $Z$. Moreover, $Z$ is either empty or can be partitioned into nonempty cliques, pairwise anticomplete to each other.}~$\blacklozenge$

\begin{adjustwidth}{1cm}{1cm}  
\begin{claim} \label{prop-tent-2P3C4C6C7-free-with-ind-T0-claim-F2F3-no-2P3}
$\{x_0,x_1,x_2,y_0,y_1,y_2\} \cap (F_2 \cup F_3) = \emptyset$. 
\end{claim} 
\end{adjustwidth} 
{\em Proof of Claim~\ref{prop-tent-2P3C4C6C7-free-with-ind-T0-claim-F2F3-no-2P3}.} Suppose otherwise. By symmetry, we may assume that there exists some $i \in \{0,1,2\}$ such that $x_i \in F_3$. In particular, $F_3 \neq \emptyset$, and so by the definition of a tent, $F_2 = Y = \emptyset$. Since $x_i$ is anticomplete to $\{y_0,y_1,y_2\}$, we deduce that $y_0,y_1,y_2 \in V(G) \setminus N_G[x_i]$. Now, since $x_i \in F_3$, it follows from the definition of a tent that $A_0 \cup A_1 \cup B_0 \cup B_1 \cup B_2 \cup C_1 \cup C_2 \cup F_3 \cup W \subseteq N_G[x_i]$. Since $F_2 = Y = \emptyset$, we deduce that $V(G) \setminus N_G[x_i] \subseteq B_3 \cup C_3 \cup Z$. So, $y_0,y_1,y_2$ is an induced two-edge path in $G[B_3 \cup C_3 \cup Z]$. But this is impossible, since by the definition of a tent, $B_3 \cup C_3 \cup Z$ can be partitioned into cliques, pairwise anticomplete to each other.\footnote{Indeed, by the definition of a tent, $B_3$ and $C_3$ are cliques, complete to each other and anticomplete to $Z$. (So, $B_2 \cup C_3$ is a clique, aticomplete to $Z$.) Meanwhile, $Z$ is either empty or can be partitioned into nonempty cliques, pairwise anticomplete to each other.}~$\blacklozenge$

\begin{adjustwidth}{1cm}{1cm}  
\begin{claim} \label{prop-tent-2P3C4C6C7-free-with-ind-T0-claim-Z-no-2P3}
$\{x_0,x_1,x_2,y_0,y_1,y_2\} \cap Z = \emptyset$. 
\end{claim} 
\end{adjustwidth} 
{\em Proof of Claim~\ref{prop-tent-2P3C4C6C7-free-with-ind-T0-claim-Z-no-2P3}.} Suppose otherwise. By symmetry, we may assume that $\{x_0,x_1,x_2\} \cap Z \neq \emptyset$. (In particular, $Z \neq \emptyset$.) By the definition of a tent, $Z$ can be partitioned into cliques, pairwise anticomplete to each other; so, $G[Z]$ is $2P_3$-free. Since $\{x_0,x_1,x_2\} \cap Z \neq \emptyset$, and since $x_0,x_1,x_2$ is an induced two-edge path in $G$, it follows that $\{x_0,x_1,x_2\} \cap N_G(Z) \neq \emptyset$. But since $N_G(Z) \subseteq F_2 \cup F_3 \cup W$ (by the definition of a tent), this contradicts Claims~\ref{prop-tent-2P3C4C6C7-free-with-ind-T0-claim-W-no-2P3} and~\ref{prop-tent-2P3C4C6C7-free-with-ind-T0-claim-F2F3-no-2P3}.~$\blacklozenge$

\begin{adjustwidth}{1cm}{1cm}  
\begin{claim} \label{prop-tent-2P3C4C6C7-free-with-ind-T0-claim-Y-no-2P3}
$\{x_0,x_1,x_2,y_0,y_1,y_2\} \cap Y = \emptyset$. 
\end{claim} 
\end{adjustwidth} 
{\em Proof of Claim~\ref{prop-tent-2P3C4C6C7-free-with-ind-T0-claim-Y-no-2P3}.} By Claims~\ref{prop-tent-2P3C4C6C7-free-with-ind-T0-claim-W-no-2P3},~\ref{prop-tent-2P3C4C6C7-free-with-ind-T0-claim-F2F3-no-2P3}, and~\ref{prop-tent-2P3C4C6C7-free-with-ind-T0-claim-Z-no-2P3}, we have that $\{x_0,x_1,x_2,y_0,y_1,y_2\} \cap (W \cup F_2 \cup F_3 \cup Z) = \emptyset$. So, we may assume that $F_2 = F_3 = W = Z = \emptyset$, for otherwise, instead of $G$, we consider the tent $G \setminus (F_2 \cup F_3 \cup W \cup Z)$, with an associated tent partition $(A_0,A_1;B_0,B_1,B_2,B_3;C_1,C_2,C_3;\emptyset;\emptyset;\emptyset;Y;\emptyset)$. 

Now, suppose that $\{x_0,x_1,x_2,y_0,y_1,y_2\} \cap Y \neq \emptyset$. By symmetry, we may assume that $\{y_0,y_1,y_3\} \cap Y \neq \emptyset$. Consequently, $\{y_0,y_1,y_2\} \cap (Y \cup C_2 \cup C_3) \neq \emptyset$. Since $Y \cup C_2 \cup C_3$ is a clique (by the definition of a tent), and since $y_0,y_1,y_2$ is an induced two-edge path in $G$, it follows that $\{y_0,y_1,y_2\} \cap N_G(Y \cup C_2 \cup C_3) \neq \emptyset$. Since $F_2 = F_3 = W = \emptyset$, the definition of a tent implies that $N_G(Y \cup C_2 \cup C_3) = C_1 \cup B_2 \cup B_3$. We now have that $\{y_0,y_1,y_2\}$ intersects both $Y$ and $C_1 \cup B_2 \cup B_3$; since $y_1$ is complete to $\{y_0,y_2\}$, whereas $Y$ and $C_1 \cup B_2 \cup B_3$ are anticomplete to each other, we may now assume by symmetry that $y_0 \in Y$ and $y_2 \in C_1 \cup B_2 \cup B_3$, and consequently, $y_1 \in C_2 \cup C_3$. Since $B_2$ is anticomplete to $C_3$, and since $B_3$ is anticomplete to $C_2$, we may now assume by symmetry that $y_1 \in C_2$ and $y_2 \in B_2 \cup C_1$, so that one of the following holds: 
\begin{enumerate}[(1)] 
\item $y_0 \in Y$, $y_1 \in C_2$, $y_2 \in B_2$; 
\item $y_0 \in Y$, $y_1 \in C_2$, $y_2 \in C_1$. 
\end{enumerate} 
In either case, we know that $x_0,x_1,x_2 \in V(G) \setminus N_G(\{y_0,y_1,y_2\})$. We now remind the reader that $F_2 = F_3 = W = Z = \emptyset$, a fact that we implicitly use below. 

Suppose first that~(1) holds. It then follows from the definition of a tent that $V(G) \setminus N_G(\{y_0,y_1,y_2\}) = B_0 \cup B_1 \cup B_3$. So, $x_0,x_1,x_2$ is an induced two-edge path in $G[B_0 \cup B_1 \cup B_2]$. But this is impossible since, by the definition of a tent, $B_0,B_1,B_2$ are disjoint cliques, pairwise anticomplete to each other. 

Suppose now that~(2) holds. It then follows from the definition of a tent that $V(G) \setminus N_G(\{y_0,y_1,y_2\}) = A_0 \cup A_1 \cup B_3$. So, $x_0,x_1,x_2$ is an induced two-edge path in $G[A_0 \cup A_1 \cup B_3]$. But this is impossible since, by the definition of a tent, $A_0 \cup A_1 \cup B_3$ is a clique.~$\blacklozenge$

\medskip 

By Claims~\ref{prop-tent-2P3C4C6C7-free-with-ind-T0-claim-W-no-2P3},~\ref{prop-tent-2P3C4C6C7-free-with-ind-T0-claim-Z-no-2P3}, and~\ref{prop-tent-2P3C4C6C7-free-with-ind-T0-claim-Y-no-2P3}, we have that $\{x_0,x_1,x_2,y_0,y_1,y_2\} \cap (W \cup Y \cup Z) = \emptyset$. So, $G \setminus (W \cup Y \cup Z)$ is not $2P_3$-free. But this contradicts Propositions~\ref{prop-T0T1-2P3C4C6Free} and~\ref{prop-tent-simplicial-universal-T0T1}(\ref{ref-no-WYZ-thickening-T0T1}): indeed, by Proposition~\ref{prop-tent-simplicial-universal-T0T1}(\ref{ref-no-WYZ-thickening-T0T1}), $G \setminus (W \cup Y \cup Z)$ is a thickening of $T_0$ or $T_1$, and so by Proposition~\ref{prop-T0T1-2P3C4C6Free}, $G \setminus (W \cup Y \cup Z)$ is $2P_3$-free. 
\end{proof}

\begin{lemma} \label{lemma-T0-tent} Let $G$ be a $(2P_3,C_4,C_6,C_7)$-free graph that contains an induced $T_0$. Then $G$ is a tent. 
\end{lemma} 
\begin{proof} Let $T$ be an induced subgraph of $G$ that is isomorphic to $T_0$, with the vertices of $T$ labeled as in Figure~\ref{fig:T0T1-labeled} (left). We now define sets $A_0,A_1,B_0,B_1,B_2,B_3,C_1,C_2,C_3,F_2,F_3,W,Y,Z$ as follows: 
\begin{itemize} 
\item for each $i \in \{0,1\}$, set $A_i := \big\{x \in V(G) \mid N_G[x] \cap V(T) = N_T[a_i]\big\}$; 
\item for each $i \in \{0,1,2,3\}$, set $B_i := \big\{x \in V(G) \mid N_G[x] \cap V(T) = N_T[b_i]\big\}$; 
\item for each $i \in \{1,2,3\}$, set $C_i := \big\{x \in V(G) \mid N_G[x] \cap V(T) = N_T[c_i]\big\}$; 
\item for each $i \in \{2,3\}$, set $F_i := \big\{x \in V(G) \mid N_G[x] \cap V(T) = V(T) \setminus \{b_i,c_i\}\big\}$; 
\item set $W := \big\{x \in V(G) \mid N_G[x] \cap V(T) = V(T)\big\}$; 
\item set $Y := \big\{x \in V(G) \mid N_G[x] \cap V(T) = \{c_2,c_3\}\big\}$; 
\item set $Z := \big\{x \in V(G) \mid N_G[x] \cap V(T) = \emptyset\big\}$. 
\end{itemize} 
We now prove a sequence of claims, which together establish that $G$ is a tent with an associated tent partition $(A_0,A_1;B_0,B_1,B_2,B_3;C_1,C_2,C_3;F_2,F_3;W;Y;Z)$. 

\begin{adjustwidth}{1cm}{1cm} 
\begin{claim} \label{lemma-T0-tent-claim-sets-partition} 
Sets $A_0,A_1,B_0,B_1,B_2,B_3,C_1,C_2,C_3,F_2,F_3,W,Y,Z$ form a partition of $V(G)$. Moreover, all the following hold: 
\begin{itemize} 
\item for each $i \in \{0,1\}$, $a_i \in A_i$; 
\item for each $i \in \{0,1,2,3\}$, $b_i \in B_i$; 
\item for each $i \in \{1,2,3\}$, $c_i \in C_i$. 
\end{itemize}
In particular, $A_0,A_1,B_0,B_1,B_2,B_3,C_1,C_2,C_3$ are all nonempty. 
\end{claim} 
\end{adjustwidth} 
\noindent 
{\em Proof of Claim~\ref{lemma-T0-tent-claim-sets-partition}.} The fact that sets $A_0,A_1,B_0,B_1,B_2,B_3,C_1,C_2,C_3,F_2,F_3,W,Y,Z$ form a partition of $V(G)$ follows immediately from Lemma~\ref{lemma-T0-attach}. The rest is immediate from the definitions of the relevant sets.~$\blacklozenge$

\begin{adjustwidth}{1cm}{1cm}
\begin{claim} \label{lemma-T0-tent-claim-sets-cliques}
Sets $A_0,A_1,B_0,B_1,B_2,B_3,C_1,C_2,C_3,F_2,F_3,W,Y$ are all cliques. Moreover, all the following hold: 
\begin{itemize} 
\item at most one of $F_2,F_3,Y$ is nonempty; 
\item $A_0$ is complete to $A_1$; 
\item $W$ is complete to $A_0,A_1,B_0,B_1,B_2,B_3,C_1,C_2,C_3,F_2,F_3$. 
\end{itemize} 
\end{claim} 
\end{adjustwidth} 
\noindent 
{\em Proof of Claim~\ref{lemma-T0-tent-claim-sets-cliques}.} First, with Proposition~\ref{prop-non-adj-comp-clique} in mind, we observe the following: 
\begin{itemize} 
\item distinct, nonadjacent vertices $b_2,b_3$ are both complete to $A_0 \cup A_1 \cup W$; 
\item distinct, nonadjacent vertices $a_0,c_1$ are both complete to $B_0 \cup W$; 
\item for all $i \in \{1,2,3\}$, distinct, nonadjacent $a_1,c_i$ are both complete to $B_i \cup W$; 
\item distinct, nonadajcent vertices $b_0,b_1$ are both complete to $C_1 \cup W$; 
\item for all $i \in \{2,3\}$, distinct, nonadjacent vertices $b_i,c_1$ are both complete to $C_i \cup W$; 
\item distinct, nonadjacent vertices $a_1,c_1$ are complete to $F_2 \cup F_3 \cup W$. 
\end{itemize} 
So, by Proposition~\ref{prop-non-adj-comp-clique}, all the following are cliques: $A_0 \cup A_1 \cup W$, $B_0 \cup W$, $B_1 \cup W$, $B_2 \cup W$, $B_3 \cup W$, $C_1 \cup W$, $C_2 \cup W$, $C_3 \cup W$, $F_2 \cup F_3 \cup W$. Thus, sets $A_0,A_1,B_0,B_1,B_2,B_3,C_1,C_2,C_3,F_2,F_3,W$ are all cliques,  $A_0$ is complete to $A_1$, and $W$ is complete to $A_0,A_1,B_0,B_1,B_2,B_3,C_1,C_2,C_3,F_2,F_3$. Furthermore, $Y$ is clique, for if some distinct $y,y' \in Y$ were nonadjacent, then $G[y,c_3,y',b_1,a_1,b_2]$ would be a $2P_3$, a contradiction. It remains to show that at most one of $F_2,F_3,Y$ is nonempty. 

Suppose first that $F_2$ and $F_3$ are both nonempty. Fix $f_2 \in F_2$ and $f_3 \in F_3$. We saw above that $F_2 \cup F_3 \cup W$ is a clique, and consequently, $f_2,f_3$ are adjacent. But now $f_2,c_3,c_2,f_3,f_2$ is a 4-hole in $G$, a contradiction. This proves that at most one of $F_2,F_3$ is nonempty. 

Next, suppose that $F_2$ and $Y$ are both nonempty, and fix some $f_2 \in F_2$ and $y \in Y$. If $y,f_2$ are adjacent, then $y,c_2,c_1,f_2,y$ is a 4-hole, a contradiction. On the other hand, if $y,f_2$ are nonadjacent, then $G[b_0,f_2,b_1,y,c_2,b_2]$ is a $2P_3$, again a contradiction. This proves that at most one of $F_2,Y$ is empty. Similarly, at most one of $F_3,Y$ is empty, and we are done.~$\blacklozenge$


\begin{adjustwidth}{1cm}{1cm} 
\begin{claim} \label{lemma-T0-tent-claim-A0-adj} 
$A_0$ is complete to $B_0,B_2,B_3$ and anticomplete to $B_1,C_1,C_2,C_3$. 
\end{claim} 
\end{adjustwidth} 
{\em Proof of Claim~\ref{lemma-T0-tent-claim-A0-adj}.} If some $a_0' \in A_0$ and $b_0' \in B_0$ are nonadjacent, then $G[b_2,a_0',b_3,b_0',c_1,b_1]$ is a $2P_3$, a contradiction. Therefore, $A_0$ is complete to $B_0$. 

If some $a_0' \in A_0$ and $b_2' \in B_2$ are nonadjacent, then $a_0',a_1,b_2',c_2,c_1,b_0,a_0'$ is a 6-hole in $G$, a contradiction. Therefore, $A_0$ is complete to $B_2$, and analogously, $A_0$ is complete to $B_3$. 

If some $a_0' \in A_0$ and $b_1' \in B_1$ are adjacent, then $a_0',b_1',c_1,b_0,a_0'$ is a 4-hole in $G$, a contradiction. Therefore, $A_0$ is anticomplete to $B_1$. 

If some $a_0' \in A_0$ and $c_1' \in C_1$ are adjacent, then $a_0',c_1',c_2,b_2,a_0'$ is a 4-hole in $G$, a contradiction. Therefore, $A_0$ is anticomplete to $C_1$. 

If some $a_0' \in A_0$ and $c_2' \in C_2$ are adjacent, then $a_0',b_0,c_1,c_2',a_0'$ is a 4-hole in $G$, a contradiction. Therefore, $A_0$ is anticomplete to $C_2$, and analogously, $A_0$ is anticomplete to $C_3$.~$\blacklozenge$

\begin{adjustwidth}{1cm}{1cm} 
\begin{claim} \label{lemma-T0-tent-claim-A1-adj} 
$A_1$ is complete to $B_1,B_2,B_3$ and anticomplete to $B_0,C_1,C_2,C_3$. 
\end{claim} 
\end{adjustwidth} 
{\em Proof of Claim~\ref{lemma-T0-tent-claim-A1-adj}.} The proof of this claim is completely analogous to the proof of Claim~\ref{lemma-T0-tent-claim-A0-adj} (we simply swap the roles of $A_0$ and $A_1$, and we swap the roles of $B_0$ and $B_1$).~$\blacklozenge$ 

\begin{adjustwidth}{1cm}{1cm}
\begin{claim} \label{lemma-T0-tent-claim-Bi-anticomp} 
$B_0,B_1,B_2,B_3$ are pairwise anticomplete to each other. 
\end{claim} 
\end{adjustwidth} 
{\em Proof of Claim~\ref{lemma-T0-tent-claim-Bi-anticomp}.} If some $b_0' \in B_0$ and $b_1' \in B_1$ are adjacent, then $b_0',b_1',a_1,a_0,b_0'$ is a 4-hole in $G$, a contradiction. Therefore, $B_0$ is anticomplete to $B_1$. 

If for some $i \in \{2,3\}$, some $b_0' \in B_0$ is adjacent to some $b_i' \in B_i$, then $b_0',b_i',c_i,c_1,b_0'$ is a 4-hole in $G$, a contradiction. Thus, $B_0$ is anticomplete to $B_2,B_3$. 

If for some distinct $i,j \in \{1,2,3\}$, some $b_i' \in B_i$ and $b_j' \in B_j$ are adjacent, then $b_i',b_j',c_j,c_i,b_i'$ is a 4-hole in $G$, a contradiction. Therefore, $B_1,B_2,B_3$ are pairwise anticomplete to each other.~$\blacklozenge$ 

\begin{adjustwidth}{1cm}{1cm} 
\begin{claim} \label{lemma-T0-tent-claim-Ci-comp} 
$C_1,C_2,C_3$ are pairwise complete to each other. 
\end{claim} 
\end{adjustwidth} 
{\em Proof of Claim~\ref{lemma-T0-tent-claim-Ci-comp}.} Suppose otherwise, and fix distinct $i,j \in \{1,2,3\}$ such that some $c_i' \in C_i$ and $c_j' \in C_j$ are nonadjacent to each other. Let $k$ be the unique index in $\{1,2,3\} \setminus \{i,j\}$. Then $c_i',c_k,c_j',b_j,a_1,b_i,c_i'$ is a 6-hole in $G$, a contradiction.~$\blacklozenge$

\begin{adjustwidth}{1cm}{1cm} 
\begin{claim} \label{lemma-T0-tent-claim-C1-Bi} 
$C_1$ is complete to $B_0,B_1$ and anticomplete to $B_2,B_3$. 
\end{claim} 
\end{adjustwidth} 
{\em Proof of Claim~\ref{lemma-T0-tent-claim-C1-Bi}.} If some $c_1' \in C_1$ and $b_0' \in B_0$ are nonadajcent, then $G[b_0',a_0,b_2,b_1,c_1',c_3]$ is a $2P_3$, a contradiction. Therefore, $C_1$ is complete to $B_0$, and analogously, $C_1$ is complete to $B_1$. 

If some $c_1' \in C_1$ and $b_2' \in B_2$ are adjacent, then $c_1',b_2',a_1,b_1,c_1'$ is a 4-hole in $G$, a contradiction. Therefore, $C_1$ is anticomplete to $B_2$, and analogously, $C_1$ is anticomplete to $B_3$.~$\blacklozenge$ 

\begin{adjustwidth}{1cm}{1cm}
\begin{claim} \label{lemma-T0-tent-claim-C2-Bi} 
$C_2$ is complete to $B_2$ and anticomplete to $B_0,B_1,B_3$. 
\end{claim} 
\end{adjustwidth}
{\em Proof of Claim~\ref{lemma-T0-tent-claim-C2-Bi}.} If some $c_2' \in C_2$ and $b_2' \in B_2$ are nonadjacent, then $G[b_2',a_1,b_3,b_0,c_1,c_2']$ is a $2P_3$, a contradiction. Therefore, $C_2$ is complete to $B_2$. 

If some $c_2' \in C_2$ is adjacent to some $b_0' \in B_0$, then $c_2',b_0',a_0,b_2,c_2'$ is a 4-hole in $G$, a contradiction. Therefore, $C_2$ is anticomplete to $B_0$. 

If for some $i \in \{1,3\}$, some $c_2' \in C_2$ is adjacent to some $b_i' \in B_i$, then $c_2',b_i',a_1,b_2,c_2'$ is a 4-hole in $G$, a contradiction. Therefore, $C_2$ is anticomplete to $B_1,B_3$.~$\blacklozenge$ 

\begin{adjustwidth}{1cm}{1cm}
\begin{claim} \label{lemma-T0-tent-claim-C3-Bi} 
$C_3$ is complete to $B_3$ and anticomplete to $B_0,B_1,B_2$. 
\end{claim} 
\end{adjustwidth}
{\em Proof of Claim~\ref{lemma-T0-tent-claim-C3-Bi}.} The proof of this claim is analogous to the proof of Claim~\ref{lemma-T0-tent-claim-C2-Bi}.~$\blacklozenge$

\begin{adjustwidth}{1cm}{1cm}
\begin{claim} \label{lemma-T0-tent-claim-F2} 
$F_2$ is complete to $A_0,A_1,B_0,B_1,B_3,C_1,C_3$ and anticomplete to $B_2,C_2$. 
\end{claim} 
\end{adjustwidth}
{\em Proof of Claim~\ref{lemma-T0-tent-claim-F2}.} First, with Proposition~\ref{prop-non-adj-comp-clique} in mind, we observe the following: 
\begin{itemize} 
\item for all $i \in \{0,1\}$, distinct, nonadjacent vertices $b_i,b_3$ are complete to $F_2 \cup A_i$; 
\item distinct, nonadjacent vertices $a_0,c_1$ are complete to $F_2 \cup B_0$; 
\item for all $i \in \{1,3\}$, distinct, nonadjacent vertices $a_1,c_i$ are complete to $F_2 \cup B_i$; 
\item distinct, nonadjacent vertices $b_1,c_3$ are complete to $F_2 \cup C_1$; 
\item distinct, nonadjacent vertices $b_3,c_1$ are complete to $F_2 \cup C_3$. 
\end{itemize} 
Proposition~\ref{prop-non-adj-comp-clique} now guarantees that $F_2 \cup A_0$, $F_2 \cup A_1$, $F_2 \cup B_0$,  $F_2 \cup B_1$, $F_2 \cup B_3$, $F_2 \cup C_1$, $F_2 \cup C_3$ are all cliques. Consequently, $F_2$ is complete to $A_0,A_1,B_0,B_1,B_3,C_1,C_3$. 

Next, if some $f_2 \in F_2$ and $b_2' \in B_2$ are adjacent, then $f_2,b_2',c_2,c_3,f_2$ is a 4-hole in $G$, a contradiction. Thus, $F_2$ is anticomplete to $B_2$. 

Finally, if some $f_2 \in F_2$ and $c_2' \in C_2$ are adjacent, then $f_2,c_2',b_2,a_1,f_2$ is a 4-hole in $G$, a contradiction. Thus, $F_2$ is anticomplete to $C_2$.~$\blacklozenge$ 

\begin{adjustwidth}{1cm}{1cm}
\begin{claim} \label{lemma-T0-tent-claim-F3} 
$F_3$ is complete to $A_0,A_1,B_0,B_1,B_2,C_1,C_2$ and anticomplete to $B_3,C_3$. 
\end{claim} 
\end{adjustwidth}
{\em Proof of Claim~\ref{lemma-T0-tent-claim-F3}.} The proof of this claim is analogous to the proof of Claim~\ref{lemma-T0-tent-claim-F2}.~$\blacklozenge$

\begin{adjustwidth}{1cm}{1cm}
\begin{claim} \label{lemma-T0-tent-claim-Y} 
$Y$ is complete to $C_2,C_3$ and anticomplete to $A_0,A_1,B_0,B_1,B_2,B_3,C_1$. 
\end{claim} 
\end{adjustwidth} 
{\em Proof of Claim~\ref{lemma-T0-tent-claim-Y}.} If some $y \in Y$ is nonadjacent to some $c_2' \in C_2$, then we observe that $T' := G[(V(T) \setminus \{c_2\}) \cup \{c_2'\}]$ is a $T_0$, and $y$ has a unique neighbor in $V(T')$, contrary to Lemma~\ref{lemma-T0-attach}. Therefore, $Y$ is complete to $C_2$, and analogously, $Y$ is complete to $C_3$. 

It remains to show that $Y$ is anticomplete to $A_0,A_1,B_0,B_1,B_2,B_3,C_1$. Suppose otherwise, and fix $X \in \{A_0,A_1,B_0,B_1,B_2,B_3,C_1\}$ such that some $y \in Y$ and $x' \in X$ are adjacent. Let $x$ be the unique vertex in $X \cap V(T)$. Then $T' := G[(V(T) \setminus \{x\}) \cup \{x'\}]$ is a $T_0$, and $y$ has exactly three neighbors (namely, $x',c_2,c_3$) in it. It then follows from Lemma~\ref{lemma-T0-attach} that there exists some $v \in V(T')$ such that $N_G(y) \cap V(T') = N_{T'}[v]$. But now $N_{T'}[v] = \{x',c_2,c_3\}$, and consequently, $v \in \{x',c_2,c_3\}$. Since $d_{T'}(c_2) = d_{T'}(c_3) = 3$, we deduce that $v = x'$, and consequently, $N_{T'}(x') = \{c_2,c_3\}$. But then $N_T(x) = \{c_2,c_3\}$, which is impossible since no vertex of $T$ has neighborhood (in $T$) precisely $\{c_2,c_3\}$.~$\blacklozenge$ 

\begin{adjustwidth}{1cm}{1cm}
\begin{claim} \label{lemma-T0-tent-claim-Z} 
$Z$ is anticomplete to $A_0,A_1,B_0,B_1,B_2,B_3,C_1,C_2,C_3,Y$. 
\end{claim} 
\end{adjustwidth} 
{\em Proof of Claim~\ref{lemma-T0-tent-claim-Z}.} We first show that $Z$ is anticomplete to $A_0,A_1,B_0,B_1,B_2,B_3,C_1,C_2,C_3$. Suppose otherwise, and fix $X \in \{A_0,A_1,B_0,B_1,B_2,B_3,C_1,C_2,C_3\}$ such that some $z \in Z$ and $x' \in X$ are adjacent. Let $x$ be the unique vertex in $V(T) \cap X$. Then $T' := G[(V(T) \setminus \{x\}) \cup \{x'\}]$ is a $T_0$, and $y$ has a unique neighbor in $V(T')$, contrary to Lemma~\ref{lemma-T0-attach}. This proves that $Z$ is indeed anticomplete to $A_0,A_1,B_0,B_1,B_2,B_3,C_1,C_2,C_3$. 

It remains to show that $Z$ is anticomplete to $Y$. Suppose otherwise, and fix adjacent $z \in Z$ and $y \in Y$. But then $G[z,y,c_3,b_0,a_0,b_2]$ is a $2P_3$, a contradiction.~$\blacklozenge$ 

\begin{adjustwidth}{1cm}{1cm} 
\begin{claim} \label{lemma-T0-tent-claim-Y-order} 
$Y$ can be ordered as $Y = \{y_1,\dots,y_t\}$ so that $N_G[y_t] \subseteq \dots \subseteq N_G[y_1]$. 
\end{claim} 
\end{adjustwidth} 
{\em Proof of Claim~\ref{lemma-T0-tent-claim-Y-order}.} We may assume that $Y \neq \emptyset$, for otherwise, the result is immediate (with $t = 0$). By Claim~\ref{lemma-T0-tent-claim-sets-cliques}, $Y$ is a clique, and $F_2 = F_3 = \emptyset$. So, by Claim~\ref{lemma-T0-tent-claim-sets-partition}, we have that sets $A_0,A_1,B_0,B_1,B_2,B_3,C_1,C_2,C_3,W,Y,Z$ form a partition of $V(G)$; this, together with Claims~\ref{lemma-T0-tent-claim-Y} and~\ref{lemma-T0-tent-claim-Z}, and together with the fact that $Y$ is a clique, implies that for all $y \in Y$, we have that $C_2 \cup C_3 \cup Y \subseteq N_G[y] \subseteq C_2 \cup C_3 \cup W \cup Y$. Therefore, it suffices to show that $Y$ can be ordered as $Y = \{y_1,\dots,y_t\}$ so that $N_G(y_t) \cap W \subseteq \dots \subseteq N_G(y_1) \cap W$. But since $Y$ and $W$ are both cliques (by Claim~\ref{lemma-T0-tent-claim-sets-cliques}), this follows immediately from Proposition~\ref{prop-C4Free-CoBip}.~$\blacklozenge$

\begin{adjustwidth}{1cm}{1cm} 
\begin{claim} \label{lemma-T0-tent-claim-Z-order} 
If $Z \neq \emptyset$, then $Z$ can be partitioned into nonempty cliques $Z_1,\dots,Z_{\ell}$, pairwise anticomplete to each other, and moreover, for each $i \in \{1,\dots,\ell\}$, $Z_i$ can be ordered as $Z_i = \{z_1^i,\dots,z_{t_i}^i\}$ so that $N_G[z_{t_i}^i] \subseteq \dots \subseteq N_G[z_1^i]$. 
\end{claim} 
\end{adjustwidth} 
{\em Proof of Claim~\ref{lemma-T0-tent-claim-Z-order}.} Assume that $Z \neq \emptyset$. First, since vertices $b_1,a_1,b_2$ induce a $P_3$ in $G$, and since they are all anticomplete to $Z$, we see that $G[Z]$ is $P_3$-free (otherwise, $G$ would contain an induced $2P_3$, a contradiction). Therefore, by Proposition~\ref{prop-P3-free}, every component of $G[Z]$ is a complete graph, or in other words, $Z$ can be partitioned into nonempty cliques $Z_1,\dots,Z_{\ell}$, pairwise anticomplete to each other. 

Now, fix $i \in \{1,\dots,\ell\}$. By Claims~\ref{lemma-T0-tent-claim-sets-partition} and~\ref{lemma-T0-tent-claim-Z}, we see that every vertex $z \in Z_i$ satisfies $Z_i \subseteq N_G[z] \subseteq F_2 \cup F_3 \cup W \cup Z_i$. Thus, it is enough to show that $Z_i$ can be ordered as $Z_i = \{z_1^i,\dots,z_{t_i}^i\}$ so that $N_G(z_{t_i}^i) \cap (F_2 \cup F_3 \cup W) \subseteq \dots \subseteq N_G(z_1^i) \cap (F_2 \cup F_3 \cup W)$. In view of Proposition~\ref{prop-C4Free-CoBip}, it suffices to show $F_2 \cup F_3 \cup W$ is a clique (since we already know that $Z_i$ is a clique). But this follows immediately from Claim~\ref{lemma-T0-tent-claim-sets-cliques}.~$\blacklozenge$ 

\medskip 

Our proof is now complete: Claims~\ref{lemma-T0-tent-claim-sets-partition}-\ref{lemma-T0-tent-claim-Z-order} together guarantee that $G$ is indeed a tent, and that $(A_0,A_1;B_0,B_1,B_2,B_3;C_1,C_2,C_3;F_2,F_3;W;Y;Z)$ is an associated tent partition. 
\end{proof}

\begin{theorem} \label{thm-T0-tent} For any graph $G$, the following are equivalent: 
\begin{itemize} 
\item $G$ is $(2P_3,C_4,C_6,C_7)$-free and contains an induced $T_0$; 
\item $G$ is a tent. 
\end{itemize} 
\end{theorem}
\begin{proof} 
This follows immediately from Proposition~\ref{prop-tent-2P3C4C6C7-free-with-ind-T0} and Lemma~\ref{lemma-T0-tent}. 
\end{proof}

\begin{corollary} \label{cor-T0-tent} For any graph $G$, the following are equivalent: 
\begin{enumerate}[(a)] 
\item $G$ is $(2P_3,C_4,C_6,C_7)$-free, contains an induced $T_0$, and contains no simplicial vertices; 
\item $G$ has exactly one nontrivial anticomponent, and this anticomponent is a thickening of~$T_0$ or~$T_1$; 
\item $G$ can be obtained from a thickening of~$T_0$ or~$T_1$ by possibly adding universal vertices to it. 
\end{enumerate} 
\end{corollary} 
\begin{proof} 
Fix a graph $G$. By Propositions~\ref{prop-non-trivial-anticomp-univ-vertices} and~\ref{prop-T0T1-2P3C4C6Free}, (b) and (c) are equivalent.\footnote{Indeed, by Proposition~\ref{prop-T0T1-2P3C4C6Free}, any thickening of $T_0$ or $T_1$ is anticonnected (and it obviously contains at least two vertices), and so by Proposition~\ref{prop-non-trivial-anticomp-univ-vertices}, (b) and (c) are equivalent.} We will complete the proof by showing that (a) implies (c), and that (b) implies (a). 

Suppose first that (a) holds, so that $G$ is $(2P_3,C_4,C_6,C_7)$-free, contains an induced $T_0$, and contains no simplicial vertices. Since $G$ is $(2P_3,C_4,C_6,C_7)$-free, Theorem~\ref{thm-T0-tent} guarantees that $G$ is a tent. Let $(A_0,A_1;B_0,B_1,B_2,B_3;C_1,C_2,C_3;F_2,F_3;W;Y;Z)$ be a tent partition of the tent $G$. The result now readily follows from Proposition~\ref{prop-tent-simplicial-universal-T0T1}. Indeed, since $G$ contains no simplicial vertices, Proposition~\ref{prop-tent-simplicial-universal-T0T1}(\ref{ref-tent-Y-simplicial}-\ref{ref-tent-Z-simplicial}) guarantees that $Y = Z = \emptyset$; but then Proposition~\ref{prop-tent-simplicial-universal-T0T1}(\ref{ref-tent-W-universal}) guarantees that all vertices in $W$ are universal in $G$, and Proposition~\ref{prop-tent-simplicial-universal-T0T1}(\ref{ref-no-WYZ-thickening-T0T1}) guarantees that $G \setminus W$ is a thickening of $T_0$ or $T_1$. Thus, (c) holds. 

Suppose now that (b) holds. Then $G$ has exactly one nontrivial anticomponent, and this anticomponent, call it $Q$, is a thickening of~$T_0$ or~$T_1$. Since $T_0$ is an induced subgraph of $T_1$, we see that $Q$ (and therefore $G$ as well) contains an induced $T_0$. Further, by Proposition~\ref{prop-tent-simplicial-universal-T0T1}, $Q$ is $(2P_3,C_4,C_6,C_7)$-free and contains no simplicial vertices. But then Propositions~\ref{prop-one-nontrivial-anticomp-simplicial} and~\ref{prop-H-free-no-universal} together guarantee that $G$ is also $(2P_3,C_4,C_6,C_7)$-free and contains no simplicial vertices. Thus, (a) holds. 
\end{proof}

\section{The structure of $\boldsymbol{(2P_3,C_4,C_6)}$-free graphs that contain an induced $\boldsymbol{C_7}$ or $\boldsymbol{T_0}$} \label{sec:structure} 

The following theorem gives a full structural description of $(2P_3,C_4,C_6)$-free graphs that contain an induced $C_7$ or $T_0$. Recall that the 7-saucer was defined in section~\ref{subsec:SpecPartDef}, whereas the tent was defined in section~\ref{sec:withT0}. 

\begin{theorem} \label{thm-main-withC7T0-full} For any graph $G$, the following are equivalent: 
\begin{itemize} 
\item $G$ is $(2P_3,C_4,C_6)$-free and contains an induced $C_7$ or $T_0$; 
\item $G$ is a 7-saucer or a tent. 
\end{itemize} 
\end{theorem} 
\begin{proof} 
This follows immediately from Theorems~\ref{thm-7-saucer} and~\ref{thm-T0-tent}. 
\end{proof}

By combining Theorem~\ref{thm-main-withC7} and Corollary~\ref{cor-T0-tent}, we obtain the following. (Recall that the family $\mathcal{M}$ was defined in subsection~\ref{subsec:familyM}, and that graphs $T_0$ and $T_1$ are represented in Figure~\ref{fig:3pentagonT0T1}.) 

\begin{theorem} \label{thm-main-withC7T0} For any graph $G$, the following are equivalent: 
\begin{enumerate}[(a)] 
\item $G$ is $(2P_3,C_4,C_6)$-free, contains an induced $C_7$ or $T_0$, and contains no simplicial vertices; 
\item $G$ has exactly one nontrivial anticomponent, and this anticomponent is a thickening of a graph in $\mathcal{M} \cup \{T_0,T_1\}$; 
\item $G$ can be obtained from a thickening of a graph in $\mathcal{M} \cup \{T_0,T_1\}$ by possibly adding universal vertices to it. 
\end{enumerate} 
\end{theorem} 
\begin{proof} 
This follows immediately from Theorem~\ref{thm-main-withC7} and Corollary~\ref{cor-T0-tent}. 
\end{proof} 

\section{Clique-width} \label{sec:cwd}

The {\em clique-width} of a graph $G$, denoted by $\text{cwd}(G)$, is the minimum number of labels needed to construct $G$ using the following four operations: 
\begin{enumerate} 
\item creation of a new vertex $v$ with label $i$; 
\item disjoint union of two labeled graphs; 
\item joining by an edge every vertex labeled $i$ to every vertex labeled $j$ (where $i \neq j$); 
\item renaming label $i$ to label $j$.  
\end{enumerate} 

The main goal of this section is to prove Theorem~\ref{thm-cwd-in-class-with-C7-T0}, which states that $(2P_3,C_4,C_6)$-free graphs that contain an induced $C_7$ or an induced $T_0$, and contain no simplicial vertices, have bounded clique-width. First, in subsection~\ref{subsec:cwd-substitution}, we prove some technical results concerning the behavior of clique-width under ``substitution.'' Then, in subsection~\ref{subsec:cwd-2P3C4C6Free-withC7T0}, we use Theorem~\ref{thm-main-withC7T0} and the results of subsection~\ref{subsec:cwd-substitution} to prove Theorem~\ref{thm-cwd-in-class-with-C7-T0}.

\subsection{Clique-width and substitution} \label{subsec:cwd-substitution} 

Given graphs $G$ and $H$ on disjoint vertex sets, and given a vertex $v \in V(G)$, we say that a graph $G^*$ is obtained by {\em substituting} $H$ for $v$ in $G$ provided that the following hold: 
\begin{itemize} 
\item $V(G^*) = \big(V(G) \setminus \{v\}\big) \cup V(H)$; 
\item $G^*[V(H)] = H$; 
\item if $|V(G)| \geq 2$, then $G^* \setminus V(H) = G \setminus v$;\footnote{If $|V(G)| = 1$, then $G \setminus v$ is not defined (since our graphs are assumed to be nonnull).} 
\item for all $u \in V(G) \setminus \{v\}$, the following hold: 
\begin{itemize} 
\item if $uv \in E(G)$, then $u$ is complete to $V(H)$ in $G^*$, 
\item if $uv \notin E(G)$, then $u$ is anticomplete to $V(H)$ in $G^*$. 
\end{itemize} 
\end{itemize}

\begin{lemma} \label{lemma-cwd-substitution} Let $G$ and $H$ be graphs on disjoint vertex sets, let $v \in V(G)$, and let $G^*$ be the graph obtained by substituting $H$ for $v$ in $G$. Then $\text{cwd}(G^*) = \max\big\{\text{cwd}(G),\text{cwd}(H)\big\}$. 
\end{lemma} 
\begin{proof} 
Set $t := \max\big\{\text{cwd}(G),\text{cwd}(H)\big\}$. Since $H$ is an induced subgraph of $G^*$, and since $G$ is isomorphic to an induced subgraph of $G^*$, it is clear that $\text{cwd}(G^*) \geq t$. It remains to show that $\text{cwd}(G^*) \leq t$. We proceed as follows. First, we independently construct the graph $G$ using only labels $1,\dots,t$, as in the definition of clique-width; by symmetry, we may assume that the vertex $v$ is assigned label $1$ when it is first created during the construction of $G$. Next, we construct the graph $H$ using only the labels $1,\dots,t$, as in the definition of clique-width, and then we rename all the labels used on $H$ as $1$. (So, at this point, all vertices of $H$ have label $1$.) We now follow the procedure that we originally followed in our construction of $G$, with one single modification: we do not create the vertex $v$, and instead, at the point when we would have created $v$, we take the disjoint union of the graph we have created and the graph $H$ (with all vertices of $H$ labeled~$1$).\footnote{After this point (i.e.\ the creation of $H$ instead of $v$, with all vertices of $H$ labeled~$1$), we follow exactly the same procedure as we did when we created $G$. (However, $v$ never gets created, and instead, we have $H$.)} This way, we construct the graph $G^*$ using only labels $1,\dots,t$. Thus, $\text{cwd}(G^*) \leq t$. 
\end{proof}

\begin{proposition} \label{prop-cwd-Kk} For all positive integers $k$, we have that 
\begin{displaymath} 
\begin{array}{rcl} 
\text{cwd}(K_k) & = & \left\{\begin{array}{lll} 
1 & \text{if} & k = 1, 
\\
2 & \text{if} & k \geq 2. 
\end{array}\right. 
\end{array} 
\end{displaymath} 
\end{proposition} 
\begin{proof} 
It is clear that $\text{cwd}(K_1) = 1$ and $\text{cwd}(K_2) = 2$. Moreover, it is clear that for all positive integers $k$, $K_{k+1}$ can be obtained by substituting a copy of $K_2$ for a vertex of a copy of $K_k$. The result now follows from Lemma~\ref{lemma-cwd-substitution} via an easy induction on $k$. 
\end{proof} 

\begin{lemma} \label{lemma-cwd-thickening} Let $G^*$ be a thickening of a graph $G$. Then $\text{cwd}(G) \leq \text{cwd}(G^*) \leq \max\{\text{cwd}(G),2\}$. Consequently, if $G$ has at least one edge, then $\text{cwd}(G^*) = \text{cwd}(G)$. 
\end{lemma} 
\begin{proof} 
Clearly, any graph that has at one edge has clique-width at least 2; therefore, the second statement follows from the first. Meanwhile, the first statement follows from Lemma~\ref{lemma-cwd-substitution} and Proposition~\ref{prop-cwd-Kk}, since $G^*$ can be obtained from $G$ by substituting complete graphs for the vertices of $G$. 
\end{proof}

\begin{lemma} \label{lemma-add-univ-cwd} If a graph $G$ can be obtained from a graph $Q$ by possibly adding universal vertices to $Q$, then $\text{cwd}(Q) \leq \text{cwd}(G) \leq \max\big\{\text{cwd}(Q),2\big\}$. 
\end{lemma} 
\begin{proof} 
Fix graphs $G$ and $Q$ such that $G$ can be obtained from $Q$ by possibly adding universal vertices to $Q$. We must show that $\text{cwd}(Q) \leq \text{cwd}(G) \leq \max\big\{\text{cwd}(Q),2\big\}$. We may assume that $G \neq Q$, for otherwise, the result is immediate. Since $Q$ is an induced subgraph of $G$, it is clear that $\text{cwd}(Q) \leq \text{cwd}(G)$. It remains to show that $\text{cwd}(G) \leq \max\big\{\text{cwd}(Q),2\big\}$. 

Set $W := V(G) \setminus V(Q)$, and set $k := |W|$. Then $G$ can be obtained from a copy of $K_2$ by first substituting $Q$ for one vertex of the $K_2$, and then substituting a copy of $K_k$ for the other vertex of the $K_2$. So, by Lemma~\ref{lemma-cwd-substitution}, we have that $\text{cwd}(G) \leq \max\big\{\text{cwd}(K_2),\text{cwd}(Q),\text{cwd}(K_k)\big\}$. By Proposition~\ref{prop-cwd-Kk}, we have that $\text{cwd}(K_2) = 2$ and $\text{cwd}(K_k) \leq 2$, and we deduce that $\text{cwd}(G) \leq \max\big\{\text{cwd}(Q),2\big\}$, which is what we needed to show. 
\end{proof}

\begin{lemma} \label{lemma-thickening-add-universal} Let $G$ be a graph that can be obtained from a thickening of a graph $H$ by possibly adding universal vertices to it. Then $\text{cwd}(G) \leq \max\{\text{cwd}(H),2\} \leq \max\{|V(H)|,2\}$. 
\end{lemma} 
\begin{proof} 
Let $H^*$ be a thickening of $H$ such that $G$ can be obtained from $H^*$ by possibly adding universal vertices to it. By Lemma~\ref{lemma-add-univ-cwd}, we know that $\text{cwd}(G) \leq \max\{\text{cwd}(H^*),2\}$, and by Lemma~\ref{lemma-cwd-thickening}, we know that $\text{cwd}(H^*) \leq \max\big\{\text{cwd}(H),2\big\}$. So, $\text{cwd}(G) \leq \max\{\text{cwd}(H),2\}$. Finally, it is clear that $\text{cwd}(H) \leq |V(H)|$ (indeed, we can create $H$ by simply using a different label for each vertex of $H$), and the result follows. 
\end{proof}

\subsection{The clique-width of $\boldsymbol{(2P_3,C_4,C_6)}$-free graphs that contain an induced $\boldsymbol{C_7}$ or an induced $\boldsymbol{T_0}$} \label{subsec:cwd-2P3C4C6Free-withC7T0}

\begin{theorem} \label{thm-cwd-in-class-with-C7-T0} Let $G$ be a $(2P_3,C_4,C_6)$-free graph that contains an induced $C_7$ or an induced $T_0$. Assume that $G$ contains no simplicial vertices. Then $\text{cwd}(G) \leq 12$. 
\end{theorem} 
\begin{proof} 
By Theorem~\ref{thm-main-withC7T0}, we know that $G$ can be obtained from a thickening of a graph in $\mathcal{M} \cup \{T_0,T_1\}$ by possibly adding universal vertices to it. Since every vertex in $\mathcal{M} \cup \{T_0,T_1\}$ has at most 12 vertices, Lemma~\ref{lemma-thickening-add-universal} guarantees that $\text{cwd}(G) \leq 12$. 
\end{proof}

\section{Some algorithmic considerations} \label{sec:Alg} 

$(2P_3,C_4,C_6)$-free graphs that contain an induced $C_7$ are trivially recognizable in $O(n^7)$ time. Similarly, since $T_0$ has nine vertices, $(2P_3,C_4,C_6,C_7)$-free graphs that contain an induced $T_0$ can be recognized in $O(n^9)$ time. Further, as stated in the introduction, \textsc{Graph Coloring} can be solved in polynomial time for graphs of bounded clique-width~\cite{Rao}, and consequently, Theorem~\ref{thm-cwd-in-class-with-C7-T0} guarantees that \textsc{Graph Coloring} can be solved in polynomial time for $(2P_3,C_4,C_6)$-free graphs that contain an induced $C_7$ or an induced $T_0$. However, as pointed out in the introduction, coloring algorithms relying on clique-width are very slow (albeit polynomial). 

Using our structural results, as well as various results from the literature, we can in fact recognize and color $(2P_3,C_4,C_6)$-free graphs that contain an induced $C_7$ or an induced $T_0$ in only $O(n^3)$ time. We note, however, that the coloring algorithm relies on integer programming and is therefore not combinatorial. 

\subsection{Recognition} \label{subsec:AlgRec} 

We sketch the algorithm for recognizing $(2P_3,C_4,C_6)$-free graphs that contain an induced $C_7$; the algorithm for recognizing $(2P_3,C_4,C_6,C_7)$-free graphs that contain an induced $T_0$ is very similar, except that it relies on Theorem~\ref{thm-T0-tent} and Corollary~\ref{cor-T0-tent}, rather than on Theorems~\ref{thm-7-saucer} and~\ref{thm-main-withC7}. 

So, let us suppose that we are given an input graph $G$ for which we need to determine whether it is $(2P_3,C_4,C_6)$-free with an induced $C_7$. By Theorem~\ref{thm-7-saucer}, we just need to check if $G$ is a 7-saucer. We start by computing a maximal sequence $v_1,\dots,v_t$ of vertices such that for all $i \in \{1,\dots,t\}$, the vertex $v_i$ is simplicial in the graph $G \setminus \{v_1,\dots,v_{i-1}\}$. This can trivially be done in $O(n^4)$ time: we search for a simplicial vertex by examining the neighborhood of each vertex, and if we find a simplicial vertex, we delete it and repeat the process until no simplicial vertices remain. However, by using an algorithm described in the introduction of~\cite{HHMMSimplicial},\footnote{Technically, we need a minor modification of this algorithm. See Lemma~2.5 of~\cite{FIK}.} the sequence $v_1,\dots,v_t$ can in fact be found in only $O(n^3)$ time. Set $A := \{v_1,\dots,v_t\}$. If $V(G) = A$, then $v_1,\dots,v_t$ is a simplicial elimination ordering of $G$, and in this case, $G$ is chordal (by~\cite{FulkersonGross}) and therefore does not belong to our class (because it does not contain an induced $C_7$). Let us now assume that $A \subsetneqq V(G)$, and consider the graph $G' := G \setminus A$. We form the set $W$ of all universal vertices of $G'$, and we form the graph $G'' := G' \setminus W$. Note that the twin relation in $G''$ is an equivalence relation,\footnote{We technically mean the binary relation on $V(G'')$ such that vertices $x,y \in V(G'')$ are related if and only if $N_{G''}[x] = N_{G''}[y]$, that is, if and only if either $x = y$ or $x,y$ are twins in $G''$.} and by~\cite{CapEvenHoleFree, Spinrad}, the twin classes of $G''$ can be computed in $O(n^2)$ time.\footnote{More precisely, an exercise from~\cite{Spinrad} asserts that such an algorithm exists, and a detailed proof can be found in~\cite{CapEvenHoleFree}.} Now, form the set $X$ by taking exactly one vertex out of each twin class of $G''$, and set $G_X := G[X]$. Clearly, $G''$ is a thickening of $G_X$. We then check if $G_X$ is isomoprhic to any graph in $\mathcal{M}$; since all graphs in $\mathcal{M}$ have at most 12 vertices, this can be done in $O(1)$ time. If $G_X$ is not isomoprhic to any graph in $\mathcal{M}$, then Theorem~\ref{thm-main-withC7} guarantees that $G$ does not belong to our class, i.e.\ that $G$ either is not $(2P_3,C_4,C_6)$-free or does not contain an induced $C_7$.\footnote{Technically, we are also using the fact that no graph in $\mathcal{M}$ contains a pair of twins.} Let us now assume that $G_X$ is isomorphic to a graph in $\mathcal{M}$. Then $G'$ admits a special partition, and this special partition can easily be ``reconstructed'' using the twin classes of $G_X$ and the set $W$. So, we now have a special partition $(X_0,\dots,X_6;Y_0,\dots,Y_6;Z_0,\dots,Z_6;W)$ of $G'$. It remains to examine $A$. If $A = \emptyset$, then $G$ is a 7-saucer with an associated 7-saucer partition $(X_0,\dots,X_6;Y_0,\dots,Y_6;Z_0,\dots,Z_6;W;A)$, and we are done. So, assume that $A \neq \emptyset$. If $A$ is not anticomplete to $X_0 \cup \dots \cup X_6$, or if there exists some $i \in \mathbb{Z}_7$ such that there is an edge between $A$ and $Y_i$, and simultaneously, $Z_{i+2} \neq \emptyset$, then $G$ is not a 7-saucer, and we are done.\footnote{The reader may possibly object that all that this proves is that $(X_0,\dots,X_6;Y_0,\dots,Y_6;Z_0,\dots,Z_6;W;A)$ is not a 7-saucer partition of $G$, but that it does not preclude the possibility that a different 7-saucer partition of $G$ might exist. However, by examining the proof of Lemma~\ref{lemma-main-withC7}, and keeping in mind that graphs that admit a special partition do not contain simplicial vertices (by Theorem~\ref{thm-main-withC7}), we see that this cannot possibly happen.} Let us assume that $A$ is anticomplete to $X_0 \cup \dots \cup X_6$, and that for all $i \in \mathbb{Z}_7$, either $A$ is anticomplete to $Y_i$, or $Z_{i+2} = \emptyset$. We then compute the vertex sets of the components $A_1,\dots,A_{\ell}$ of $A$ (using, for example, breadth-first-search), and we check if $A_1,\dots,A_{\ell}$ are all cliques. If not, then $G$ is not a 7-saucer. So, let us assume that $A_1,\dots,A_{\ell}$ are indeed cliques. We then compute the degrees (in $G$) of all vertices in $A$, and for each $i \in \{1,\dots,\ell\}$, we order $A_i$ as $A_i = \{a_1^i,\dots,a_{r_i}^i\}$ so that $N_G[a_{r_i}^i] \leq \dots \leq N_G[a_1^i]$, and we check if $N_G[a_{r_i}^i] \subseteq \dots \subseteq N_G[a_1^i]$. If we do indeed have that $N_G[a_{r_i}^i] \subseteq \dots \subseteq N_G[a_1^i]$ for all $i \in \mathbb{Z}_7$, then $G$ is a 7-saucer with an associated 7-saucer partition $(X_0,\dots,X_6;Y_0,\dots,Y_6;Z_0,\dots,Z_6;W;A)$, and we are done. Otherwise, $G$ is not a 7-saucer. Clearly, the running time of the algorithm is $O(n^3)$.

\subsection{Coloring} \label{subsec:AlgCol}

Let us sketch an $O(n^3)$ coloring algorithm for $(2P_3,C_4,C_6)$-free graphs that contain an induced $C_7$ or $T_0$. We note that our algorithm is robust in the following sense: for any input graph $G$, we either produce an optimal coloring of $G$, or we correctly certify that the graph $G$ is not in our class, i.e.\ that it is either not $(2P_3,C_4,C_6)$-free, or it contains no induced~$C_7$ and no induced~$T_0$. 

Given an input graph $G$, we first compute a maximal sequence $v_1,\dots,v_t$ of vertices such that for all $i \in \{1,\dots,t\}$, the vertex $v_i$ is simplicial in the graph $G \setminus \{v_1,\dots,v_{i-1}\}$; as explained in section~\ref{subsec:AlgRec}, this can be done in $O(n^3)$ time. If $V(G) = \{v_1,\dots,v_t\}$, then $v_1,\dots,v_t$ is in fact a simplicial elimination ordering of $G$, and we can obtain an optimal coloring of $G$ by simply coloring $G$ greedily using the ordering $v_t,\dots,v_1$. So, assume that $\{v_1,\dots,v_t\} \subsetneqq V(G)$, and set $G' := G \setminus \{v_1,\dots,v_t\}$. We then compute the set $W$ of all universal vertices of $G'$, and we set $G'' := G \setminus W$. Once again as in subsection~\ref{subsec:AlgRec}, we compute the twin classes of $G''$ in $O(n^2)$ time. Now, by Theorem~\ref{thm-main-withC7T0-full}, either $G''$ is a thickening of a graph in $\mathcal{M} \cup \{T_0,T_1\}$, or $G$ does not belong to our class. But note that each graph in $\mathcal{M} \cup \{T_0,T_1\}$ has at most 12 vertices. We now simply check how many twin classes $G''$ has: if it has more than 12, then it is not a thickening of a graph in $\mathcal{M} \cup \{T_0,T_1\}$, and consequently, $G$ does not belong to our class. So, let us assume that $G''$ has at most 12 twin classes. We can then use integer programming to compute an optimal coloring of $G''$ in $O(n+m)$ time, as explained in~\cite{Koutecky, Lampis}.\footnote{We do note that the $O(n+m)$ hides a large multiplicative constant.} We now finish coloring $G$ in the obvious way: we assign all vertices of $W$ distinct colors, not used on $G''$, and then we complete our coloring by assigning colors to $v_t,\dots,v_1$ greedily. The total running time of our coloring algorithm is $O(n^3)$.

\small{
 
}

\end{document}